\documentclass[11pt,notitlepage,twoside]{article}
\usepackage{latexsym}
\usepackage{amsmath}
\usepackage{amsfonts}
\usepackage{amssymb}
\usepackage{amscd}
\usepackage{color}
\renewcommand{\a }{\alpha }
\renewcommand{\b }{\beta }
\renewcommand{\d}{\delta }
\newcommand{\D }{\Delta }

\newcommand{\e }{\varepsilon }
\newcommand{\g }{\gamma}

\renewcommand{\l }{\lambda }
\renewcommand{\L }{\Lambda }

\newcommand{\n }{\nabla }

\renewcommand{\o }{\omega }

\renewcommand{\O }{\Omega }

\newcommand{\ov}{\overline}
\newcommand{\intbar}{\mathop{\int\makebox(-13.5,0){\rule[4pt]{.7em}{0.3pt}}%
\kern-6pt}\nolimits}
\newcommand{\wtilde }{\widetilde}

\newcommand{\be}{\begin{equation}}
\newcommand{\ee}{\end{equation}}
\newcommand{\bes}{\begin{equation*}}
\newcommand{\ees}{\end{equation*}}
\newcommand{\ba}{\begin{eqnarray}}
\newcommand{\ea}{\end{eqnarray}}
\newcommand{\bas}{\begin{eqnarray*}}
\newcommand{\eas}{\end{eqnarray*}}
\newenvironment{pf}{\noindent{\sc Proof}.\enspace}{\rule{2mm}{2mm}\medskip}
\newenvironment{pfn}{\noindent{\sc \bf Proof }}{\rule{2mm}{2mm}\medskip}

\newcommand{\R}{\mathbb{R}}

\newcommand{\N}{\mathbb{N}}
\renewcommand{\o }{\omega }

\author{ Rufaidah Alharbi\thanks{E-mail : \texttt{451214282@qu.edu.sa}} \, Mohamed Ben Ayed\thanks{E-mail : \texttt{M.BenAyed@qu.edu.sa}} and
Khalil El Mehdi\thanks{Corresponding author. E-mail : \texttt{K.Jiyid@qu.edu.sa} }  
   \\ \\
{\footnotesize
 Department of Mathematics, College of Science, Qassim University, Buraydah, Saudi Arabia.}
}
\date{}

\title{\bf  Blowing-up Solutions with Residual Mass in a Slightly Subcritical Dirichlet Problem
}
\begin{document}

\newtheorem{lem}{Lemma}[section]
\newtheorem{pro}[lem]{Proposition}
\newtheorem{thm}[lem]{Theorem}
\newtheorem{rem}[lem]{Remark}
\newtheorem{cor}[lem]{Corollary}
\newtheorem{df}[lem]{Definition}

\maketitle

\noindent{\bf Abstract:} In this paper, we study  the Dirichlet elliptic problem $(\mathcal{P}_\varepsilon)$:  $-\Delta u +V\,u = u^{p-\varepsilon}$, $u>0$ in $\Omega$, $u=0$ on $\partial\Omega$, where $\Omega\subset \R^n$ ( $n\geq 3$)  is a bounded  domain, $V$ is a smooth positive function on $\overline{\Omega}$,  $p+1= 2n/(n-2)$ is the critical Sobolev exponent, and $\varepsilon >0$ is a small parameter. 
First, we show that, unlike the case of weak convergence to zero, interior bubbling solutions with a nonzero weak limit cannot occur in low dimensions. We then treat the general setting by removing the restriction that blow-up points are confined to the interior. Using delicate asymptotic expansions of the gradient of the associated functional, we prove that in dimensions $n=4$ and $n=5$, single blow-up point cannot coexist with residual mass.\\
We further elucidate the role of the sign of the normal derivative of the potential $V$ on the boundary: if it is positive, any single blow-up solution with residual mass must occur in the interior; if it is negative at some boundary point, boundary blow-up solutions with residual mass can be constructed. Finally, we construct both simple and non-simple interior blow-up solutions exhibiting residual mass, without any assumption on the sign of the normal derivative of $V$. These results provide new insights into the interaction between the potential, the geometry of the domain, and the critical nonlinearity.
\bigskip


\bigskip

\noindent {\bf MSC  $2020$}: 35A15, 35J20,  35J25.

\section{Introduction}
This paper is concerned with the existence and non-existence of blowing-up solutions with residual mass for the following slightly subcritical problem:
\begin{align}
(\mathcal{P}_{V, \e}):\qquad 
\begin{cases}
-\Delta u +V u = u^{p-\varepsilon}, u>0 \quad &\mbox{in}\quad \Omega,
\\
\quad u =0\,\quad &\mbox{on}\quad \partial\Omega,
\end{cases}
\end{align}
where $\Omega \subset \mathbb{R}^n$ ($n \geq 3$) is a smooth bounded domain, $V \in C^3(\overline{\Omega})$ is positive, $p+1 = \frac{2n}{n-2}$ is the critical Sobolev exponent for the embedding $H^1_0(\Omega) \hookrightarrow L^q(\Omega)$, and $\varepsilon > 0$ is a small parameter.

Problems of the form $(\mathcal{P}_{V,\varepsilon})$ appear in several areas of applied mathematics, for example, in nonlinear optics, where they model laser beam propagation \cite{FM, CGL, 6}, and in population dynamics, where they arise as reaction-diffusion equations describing species migration in shifting habitats \cite{BLZ, GLO}. Furthermore, the problem is closely related to the Yamabe problem in conformal geometry, as discussed in \cite{DH} and the references therein.

A key challenge posed by problem $(\mathcal{P}_{V,\varepsilon})$ is that its solutions develop blow-up phenomena in the limit $\varepsilon \to 0$. To elucidate this phenomenon,  we note that the problem $(\mathcal{P}_{V,\varepsilon})$ has a variational formulation:  its solutions correspond to positive critical points of the associated functional defined by 
\be\label{funct}
 I_{\varepsilon}(u) := \frac{1}{2}  \int_{ \Omega} ( |\n u |^{2}+V\,u^{2}) - \frac{ 1 }{ p+1-\varepsilon}  \int_{ \Omega}  | u |^{p+1-\varepsilon} , 
\ee
defined on  $H^1_0(\Omega)$  equipped with the norm $\|.\|$ and its corresponding inner product given by:
\be\label{F}
\| u \|^2= \int_{\Omega} \left(| \n u|^2 + V u^2\right) \quad \mbox{ and } \quad (u,v)=\int_{\Omega} \left(\n u \cdot \n v + V u \, v\right).
\ee
Sequences converging to a critical point (for instance, minimizing sequences) are bounded. One would then like to extract a convergent subsequence that also converges in $L^{p+1-\varepsilon}(\O)$. When $\varepsilon > 0$, this can be achieved using the compactness of the Sobolev embedding $H^1_0(\O) \hookrightarrow L^{p+1-\varepsilon}(\O)$. However, when $\varepsilon = 0$, compactness is lost and the embedding is only continuous: convergence in $L^{p+1}(\O)$ may fail to hold. This is both a key feature and a major difficulty in the study of solutions to problem $(\mathcal{P}_{V,\varepsilon})$ as $\varepsilon$ tends to zero.
Blow-up phenomena may occur around certain points in $\Omega$ or on its boundary $\partial\Omega$. To be clearer, let us specify  what is meant by a blow-up point. If $u_\varepsilon$ is a bounded solution to the problem $(\mathcal{P}_{V,\varepsilon})$, we say that it develops a blow-up point at $a$  if it satisfies
$$
  \lim \inf_{ r \to 0} \lim \sup_{ \e \to 0 } \int _{ B(a, r)\, \cap \,\O}  |\n u_\e| ^2 (x) dx  \,\Big(\mbox{or}\,\int _{ B(a, r)\,\cap\,\O }  u_\e ^{2n/(n-2)} (x) dx \Big) >0. 
  $$
In other words, as their name indicates, a fixed part of the energy of $u_\varepsilon$  is concentrated in the neighborhood of these points. Sequences that admit such points cannot converge in $H^1_0(\Omega)$ (or $L^{p+1}(\Omega)$). Indeed, if $(u_\varepsilon)$ converges, it must converge to its weak limit which does not "see" the concentration. In particular, the limit $u_0$ would satisfy $\| u_0\| < \lim\| u_\varepsilon\|$, which contradicts the convergence. This is a motivation for studying them: it helps to clarify what the obstacles to good convergence that can occur.

Notice that, by the concentration compactness principle \cite{L, S}, if $(u_\varepsilon)$ is a bounded sequence of solutions to $(\mathcal{P}_{V,\e})$, then (up to a subsequence) it has to satisfy 
\begin{align}
  &  u_{\varepsilon}= u_0+ \sum_{i=1}^N P\delta_{(a_{i,\e},\lambda_{i,\e})}+v_{\e}  \qquad \mbox{ with } \quad \Vert v_{\e}\Vert \longrightarrow 0   \quad as  \quad \e \longrightarrow 0,  \label{KR1}  \\ 
  &\lambda_{i,\e}>0,\,a_{i,\e}\in\O,\,  \lambda_{i,\e}d(a_{i,\e}, \partial\O) \longrightarrow \infty , \, a_{i,\e}\longrightarrow  a_i \in \ov{\O} ,\,  \forall i \in \lbrace 1,\cdots,N\rbrace , \label{KR2}\\
 & \e_{ij} :=\left(\frac{\l_{i,\e}}{\l_{j,\e}} + \frac{\l_{j,\e}}{\l_{i,\e}}+ \l_{i,\e}\l_{j,\e}|a_{i,\e}-a_{j,\e}|^2\right)^{(2-n)/2} \longrightarrow 0 \quad \forall i\neq j  , \label{KR3}
  \end{align}
$u_0$ is either a solution to $(\mathcal{P}_{V,0})$ or is identically zero, $P\delta_{(a_{i,\e},\lambda_{i,\e})}:=P\delta_i$ is the standard projection onto $H^1_0(\O)$ of the function 
 \begin{equation}\label{delta}
   \delta_{(a_{i,\e},\lambda_{i,\e})} (x) = c_0\frac{\lambda_{i,\e}^{(n-2)/{2}}}{(1+\lambda_{i,\e}^2|
 x-a_{i,\e}|^2)^{(n-2)/{2}}}:=\d_i(x), \quad x \in \R^n,\, c_0=(n(n-2))^{\frac{n-2}{4}},
 \end{equation}
 defined by 
\begin{equation}\label{KR5} 
\Delta P\delta_i= \Delta \delta_i \quad\mbox{in}\quad\Omega,\quad P\delta_i=0\quad\mbox{on}\quad \partial\Omega .
\end{equation}
The functions $\d_i$ are the only solutions to  problem \cite{CGS}
$$
 -\D \o = \o ^\frac{n+2}{n-2} ,\qquad \o > 0\,\quad \mbox{ in } \quad \R^n.
$$
In the critical case where $\varepsilon = 0$, it is well known that the existence of solutions to problem $(\mathcal{P}_{V,0})$ depends on the geometry of the domain, the properties of the potential $V$, and the dimension $n$. For instance, when $V$ is constant and the domain is star-shaped, the problem admits no solutions. In view of the extensive literature on the subject, we simply refer to the pioneering works of Brezis–Nirenberg \cite{BN} and Bahri–Coron \cite{BC}.

For $\varepsilon > 0$, two types of questions have been addressed in the study of problem $(\mathcal{P}_{V,\varepsilon})$: the first describes the blow-up behavior of its solutions, while the second constructs such solutions for small values of the parameter $\varepsilon$. For $V \equiv 0$, Schoen \cite{Sch} showed  an alternative: in the decomposition \eqref{KR1} either $u_0 = 0$ or $N = 0$. More precisely, he  proved that either $(u_\varepsilon)$ blows up and converges weakly to $0$, in which case the blow-up points lie in $\Omega$, or it converges strongly to a  solution of $(\mathcal{P}_{0,0})$. Still in the case $V \equiv 0$, Atkinson-Peletier \cite{AP}, Brezis- Peletier \cite{BP3}, Rey \cite{R29}, Han \cite{H15}, and Hebey \cite{16}  studied the asymptotic behavior of single blow-up solutions to $(\mathcal{P}_{0,\varepsilon})$ (that is, $N=1$ in the decomposition \eqref{KR1}), whereas Rey \cite{R31}, Bahri-Li-Rey \cite{BLR}, and Rey \cite{R32} considered the case of multiple blow-up solutions. In the same setting, the authors in \cite{R31}, \cite{BLR}, \cite{R32} and \cite{CPS} constructed blow-up solutions   to $(\mathcal{P}_{0,\varepsilon})$ (see the survey \cite{P2013} for a complete review). For $V \not\equiv 0$ and $n = 3$, the asymptotic behavior of single blow-up solutions  to $(\mathcal{P}_{V,\varepsilon})$ was studied by \cite{FKK}. More recently, in \cite{JMAA}, the authors provided an exhaustive description of the asymptotic behavior of multiple interior blow-up solutions for dimensions $n \geq 4$. For the inverse problem, the authors of \cite{2ABE} and \cite{OM} constructed interior blow-up solutions to $(\mathcal{P}_{V,\varepsilon})$  that converge weakly to zero. Thus, it is natural to ask whether it is possible to construct blowing-up solutions of the problem $(\mathcal{P}_{V,\varepsilon})$ with residual mass, that is, solutions possessing a non-zero weak limit. In a recent paper \cite{KK}, the authors considered the analogous problem with Neumann boundary conditions and proved the non-existence of interior blowing-up solutions with residual mass in low dimensions ($3\leq n\leq 5$), whereas they constructed interior blowing-up solutions with residual mass in higher dimensions ($n\geq 7$). We will prove that these results also hold for our problem (see Theorems \ref{t0}, \ref{t4}, \ref{t5} and \ref{t6}). Although the underlying structure remains similar, the Dirichlet boundary condition in the present paper necessitates the introduction of the bubble projection, which adds further complexity. In particular, the error term between the bubble and its projection becomes more intricate, requiring refined estimates and a more delicate treatment of error terms than in [\cite{KK}. Consequently, our situation demands more technical computations. Moreover, in this paper we study the general case by removing the assumption that blow-up points are confined to the interior of the domain, a restriction imposed in \cite{KK} for the Neumann problem. 
The main difficulty lies in showing that the blow-up points necessarily belong to the interior of the domain. By performing delicate asymptotic expansions of the gradient (see the commentary  after Theorem \ref{t0}), we successfully resolved this question for single blowing-up solutions with residual mass in the cases $n=4$ and $n=5$. Indeed, in Theorem \ref{t1} we prove the non-coexistence of blow-up dynamics and residual mass in that dimensions. We believe the arguments developed here should also address the case of multiple blow-up points with residual mass, and we plan to return to this intricate question in the future. We further clarify the role of the sign of the normal derivative of $V$ on the boundary. Specifically, we prove that when the normal derivative of $V$ is positive on the boundary, all single blowing-up solutions with residual mass have to blow-up in the interior (see Theorem \ref{t2}). Conversely, when the normal derivative is negative at some boundary point, we are able to construct solutions with residual mass  with boundary blow-up points (see Theorem \ref{t3}).
These results show that the role played by the potential $V$ in these equations is reminiscent of that played by the functions used to prescribe scalar curvature on Riemannian manifolds with boundary. As far as we know, our results are the first to demonstrate how the sign of the normal derivative of the potential $V$ on the boundary influences these equations, revealing a direct analogy with the prescribed scalar curvature problem on Riemannian manifolds with boundary.

The remainder of the paper is structured  as follows. Section $2$ states our main results. Section $3$ analyzes the infinite-dimensional component of the solutions. Section $4$ provides precise estimates of the gradient of $I_\e$. In Section $5$, we construct both simple and non-simple interior blow-up solutions  with residual mass (Theorems \ref{t4}-\ref{t6} below). Section $6$ deals with the impossibility of coexistence between blow-up points and residual mass in low dimensions (Theorems \ref{t0} and \ref{t1} below). Section $7$ shows how the sign of $\frac{\partial V}{\partial\nu}$ on the boundary determines the location of blow-up points for single blow-up solutions with residual mass (Theorems \ref{t2} and \ref{t3} below). Section $8$ concludes the paper and outlines potential directions for future research. Finally, we collect some useful estimates in the appendix (Section $9$).

  \section{ Main Results}
   We start by ruling out the existence of interior blowing-up solutions with residual mass for small dimensions. 
\begin{thm}\label{t0} 
Let $3\leq n\leq 5$, $\Omega$ be a smooth bounded domain in $\R^n$   and $ u_0 $ be a non-degenerate solution of  $(\mathcal{P}_{V,0})$ (that is, for $\e=0$). Then, problem $(\mathcal{P}_{V,\varepsilon})$ has no   solutions $ u_\e $ such that  
\begin{align*}
  &  u_{\varepsilon}= u_0+ \sum_{i=1}^N P\delta_{(a_{i,\e},\lambda_{i,\e})}+v_{\e}  \qquad \mbox{ with } \quad \Vert v_{\e}\Vert \longrightarrow 0   \quad as  \quad \e \longrightarrow 0,   \\ 
  &a_{i,\e}\in\O,\,  \lambda_{i,\e} \longrightarrow \infty , \, a_{i,\e}\longrightarrow  a_i \in \O ,\,  \forall i \in \lbrace 1,\cdots,N\rbrace ,  \quad \e_{ij} \to 0  \quad \forall i\neq j . 
  \end{align*} where $ \e_{ij} $ is defined in \eqref{KR3}. 
\end{thm}

Next, we seek to understand the general case by relaxing the assumption that blow-up points belong to the interior of the domain, that is, we want to exclude the case where the points $a_{i,\e}'s$ converge to some points  belonging to the boundary  ($a_i\in\partial\O$). To address this challenging issue, we focus on the case of single blow-up  point and carry out precise asymptotic expansions of the gradient of the associated functional $I_\e$. For dimension five, the contribution of the error term $v_\e$ is negligible compared to the principal terms, which leads directly to the desired result. In dimension four, however, the error term is of the same order as the principal terms. To overcome this difficulty, we decompose the error term $v_\e$ into odd and even parts and refine  the estimates of the integrals involving $v_\e$. This allows us to obtain the result (see Section $5$ for details). Unfortunately, we were unable to include the case of $n=3$. We believe that the result should be true and that, to obtain it, we should use the projection used in \cite{2ABE} instead of the one we used in this paper. Furthermore, we expect that all these results extend naturally to the case of multiple concentration points. To be more precise, we 
 present our nonexistence result concerning single blowing-up solutions with residual mass.
\begin{thm}\label{t1} 
Let $n\in\{4, 5\}$, $\Omega$ be any smooth bounded domain in $\R^n$   and $ u_0 $ be a non-degenerate solution of  $(\mathcal{P}_{V,0})$ (for $\e=0$). Then, problem $(\mathcal{P}_{V,\varepsilon})$ has no   solutions $ u_\e $ such that  
\begin{align} 
 &  u_\e= u_0+ P\d_{(a_{\e},\l_{\e})}   + v_\e \quad \mbox{ with }       \label{dec} \\
 &  \| v_\e\|\to 0, \quad   a_\e\in\O,\quad \l_{\e}\,d(a_\e,\partial\O) \to \infty\quad\mbox{as}\quad \e\to 0. \label{cdec}
  \end{align}
\end{thm}

Given the result of Theorem \ref{t1}, one naturally wonders what occurs in higher dimensions. In this context, we state the following result.
\begin{thm}\label{t2} 
Let $n \geq  6 $, $\Omega$ be any smooth bounded domain in $\R^n$, $ u_0 $ be a non-degenerate solution of  $(\mathcal{P}_{V,0})$ and $V:\overline{\O}\to \mathbb{R}$ be a $C^3$ positive function such that $\frac{\partial V}{\partial \nu}>0$ on the boundary of $\O$. Then, any solution $u_\varepsilon$ to $(\mathcal{P}_{V,\e})$ taking the form \eqref{dec} under condition \eqref{cdec} must satisfy $d(a_\e,\partial\O) \not\to 0$ as $\varepsilon \to 0$.
\end{thm}

In Theorem \ref{t2}, we assumed that the normal derivative of $V$ is positive on the boundary, which allowed us to prove that the blow-up point  belongs to the interior of the domain, that is, $a_\e$ converges to $a\in\O$. To establish the necessity of this assumption, we now show that if the normal derivative of $V$ is negative at some  point $b\in\partial\O$, then there exists a solution such that $a_\e$ converges to $b$. More precisely, we have:
\begin{thm}\label{t3} 
Let $n \geq  7 $, $\Omega$ be any smooth bounded domain in $\R^n$, $ u_0 $ be a non-degenerate solution of  $(\mathcal{P}_{V,0})$ and $V:\overline{\O}\to \mathbb{R}$ be a $C^3$ positive function, $b$ be a non‑degenerate critical point  of the restriction of $V$ to the boundary such that $\frac{\partial V}{\partial \nu}(b)<0$. Then,  there exists a solution $u_\varepsilon$ to $(\mathcal{P}_{V,\e})$ of the form \eqref{dec} with \eqref{cdec} and $a_\e \to b$ as $\e\to 0$.

Consequently, for $\varepsilon > 0$ sufficiently small, there exist at least $q$ boundary blowing-up solutions with residual mass to $(\mathcal{P}_{V,\e})$, where $q$ denotes the number of critical points of the restriction of $V$ to the boundary at which the normal derivative of $V$ is negative.
\end{thm}

Note that Theorem \ref{t3} does not exclude the possibility of the existence of interior blowing-up solutions of $(\mathcal{P}_{V,\e})$ carrying residual mass. Theorem \ref{t2}, on the other hand, suggests that if the normal derivative of $V$ is positive on the boundary, then such solutions must blow up in the interior. In fact, in higher dimensions,  one can always construct interior blowing-up solutions with residual mass without imposing any condition on the sign of the normal derivative of $V$ on the boundary. Our next results aim to construct interior blowing-up solutions with residual mass, either with isolated or with clustered bubbles. Here, isolated bubbles are understood as a finite sum of  $P\d_{a_{i,\e},\l_{i,\e}}$ whose concentration points $a_{i,\e}'s$ are uniformly separated from one another. In contrast, clustered bubbles refer to a sum of $P\d_{a_{i,\e},\l_{i,\e}}$ with comparable concentration rates $\l_{i,\e}'s$ and whose centers $a_{i,\e}'s$ converge to a common point $b$ as $\e\to 0$.  Next, we begin by constructing blowing-up solutions of $(\mathcal{P}_{V,\e})$  with residual mass and isolated bubbles. Our result is as follows:
\begin{thm}\label{t4} 
Let $n \geq  7 $, $\Omega$ be any smooth bounded domain in $\R^n$, $ u_0 $ be a non-degenerate solution of  $(\mathcal{P}_{V,0})$ and $V:\overline{\O}\to \mathbb{R}$ be a $C^3$ positive function having $r$ non-degenerate critical points $b_1$, ..., $b_r$. Then, for any $N\leq r$, there exits $\e_*>0$ such that for any $\e\in (0,\e_*]$, problem $(\mathcal{P}_{V,\e})$  has a solution $u_{\e, b_{i_1},...,b_{i_N}}$ satisfying: $u_{\e, b_{i_1},...,b_{i_N}}$ develops exactly one bubble at each point $b_{i_j}$ and converges weakly to $u_0$ in $H^1_0(\O)$ as $\e \to 0$. More precisely there exist $\l_{j,\e} > 0$  and  $a_{j,\e}\to b_{i_j}$ with $ j \in \{1, \cdots, N \}$ such that
$$
\bigg |\bigg | u_{\e,b_{i_1},...,b_{i_N}}-u_0 -\sum_{j=1}^N P\d_{(a_{ j,\e}, \l_{ j,\e})} \bigg |\bigg | \to 0,\quad  \quad\mbox{as}\quad \e\to 0.
$$
Moreover, we have 
$$ 
\lim_{ \e \to 0} \l_{ j , \e }^2 {\e } = \frac{ \g_n }{ \ov{c}_3} V(b_{i_j}) \qquad \mbox{ and } \qquad  | a_{ j, \e } - b_{i_j} | \leq c ( \e^{(n-6) / 4 } | \ln \e |  +  \e^{1/2}  )
 $$ where $ \ov{c}_3$ and $ \g_n$ are positive constants defined in Proposition \ref{rp43}. 

Consequently, for $\e>0$ sufficiently small, $(\mathcal{P}_{V,\e})$  has at least $2^r-1$ interior blowing-up solutions which converge weakly to $u_0$.
\end{thm}

We note that a result similar to Theorem  \ref{t4} was obtained in \cite{ABE} in the context of the scalar curvature problem on high-dimensional spheres. In that work, the nonlinear term is multiplied by a non-constant Morse function, which prevents the formation of blowing-up solutions with clustered bubbles and consequently simplifies the analysis. In contrast, in our setting, clustered bubbling does occur, requiring a more delicate and detailed analysis. The following results  construct interior blowing-up solutions with residual mass and clustered bubbles centered at a critical point of $V$. To this end, we introduce some notation. For $N \in \mathbb{N}$ and $b$ a critical point of $V$, we define the following function:
\be\label{FN}
F_{N,b}(y_1,\cdots , y_N)=\sum_{j=1}^N D^2V(b)(y_j,y_j) -  \sum_{l\neq k}\frac{1}{|y_l-y_k|^{n-2}},
\ee
where $(y_1,\cdots , y_N)\in (\mathbb{R}^n)^N$ such that $y_i\neq y_j$ if $i\neq j$.

Our result can be stated as follows:
\begin{thm}\label{t5}
Let $n \geq  7 $, $\Omega$ be any smooth bounded domain in $\R^n$, $ u_0 $ be a non-degenerate solution of  $(\mathcal{P}_{V,0})$, $V:\overline{\O}\to \mathbb{R}$ be a $C^3$ positive function, $b$ be a non-degenerate critical point of $V$, and $N\in \N$ with $N\geq 2$. Assume that $F_{N,b}$ has a non-degenerate critical point $(\bar{y}_1,\cdots, \bar{y}_N)$. Then, there exits $\e_*>0$ such that for any $\e\in (0,\e_*]$, problem $(\mathcal{P}_{V,\e})$  has a solution $u_{\e, b}$ satisfying:
\begin{align}  
 &  \Big\| u_{\e,b} -u_0- \sum_{i=1}^N P\d_{(a_{i,\e}, \l_{i,\e})}   \Big\| \to 0 \quad \mbox{ as } \, \,  \e \to 0   \quad \mbox{ with } \notag \\  
& \lim_{ \e \to 0 } {\l_{i,\e}} ^2 {\e} =  \frac{ \g_n }{ \ov{c}_3} V(b) \qquad \mbox{ and } \qquad 
\lim_{ \e \to 0 }  \e^{\frac{4 - n }{2n}} ( a_{i,\e} - b )  = \sigma \overline{y}_i  \qquad  \forall \, \,  i \in \{ 1 , \cdots , N \}  \label{t32} 
\end{align}
where  $\sigma$ is the constant defined by \eqref{s2} and $ \ov{c}_3$ and $ \g_n$ are positive constants defined in Proposition \ref{rp43}.\\
Moreover, if for each $N$,  $F_{N,b}$ has a non-degenerate critical point, then, problem $(\mathcal{P}_{V,\e})$ has an arbitrary number of  distinct solutions provided that $\e$ is small.
\end{thm}
\begin{rem}
  We remark that if $ b $  is a non-degenerate maximum point of $V $, then the second derivative of $ V $ satisfies: $ D^2 V(b )(x,x) \leq - c | x |^2 $ for all $x$. Consequently, for each  $N \geq 2 $, the function $  F_{N,b } $ achieves its maximum and thus possesses at least one critical point.
\end{rem}

The aim of the following result is to treat the case of several interior blow-up points. 
\begin{thm}\label{t6}
Let $n \geq 7$, $\Omega$ be any smooth bounded domain in $\R^n$, $ u_0 $ be a non-degenerate solution of  $(\mathcal{P}_{V,0})$, $V:\overline{\O}\to \mathbb{R}$ be a $C^3$ positive function having $r$ non-degenerate critical points $b_1$, ..., $b_r$.   Let $\ell \leq r$ and  $N_1, \cdots , N_\ell \in \N$. For $ N_j \geq 2$, we assume that the function $F_{N_j,b_{i_j}}$ has a non-degenerate critical point $(\overline{y}_{j,1},\cdots , \overline{y}_{j,N_j })$ . Then, there exists $\e_*>0$ small such that for any  $\e\in (0,\e_*]$,  problem $(\mathcal{P}_{V,\varepsilon})$ admits a solution $u_{\e, b_1,\cdots, b_r}$ such that 
$$
 \Big\| u_{\e, b_1,\cdots, b_N}-u_0  - \sum_{i=1}^{N_1}P\d_{(a_{1,i,\e}, \l_{1,i,\e})} - \cdots - \sum_{i=1}^{N_\ell} P\d_{(a_{\ell ,i, \e }, \l_{ \ell ,i,\e})}   \Big\| \to 0 \qquad \mbox{ as } \quad \e \to 0 
$$
with   the rates $\l_{ j , i , \e } $'s and the concentration points $a_{ j , i , \e }$'s satisfy the properties introduced in  \eqref{t32} (with $ b_{i_j} $ instead of $ b $). 
\end{thm}
\begin{rem}
\begin{enumerate}
\item
The interior blowing-up solutions with residual mass constructed in Theorems \ref{t4} and \ref{t6} can be combined to produce blowing-up solutions of $(\mathcal{P}_{V,\varepsilon})$  whose blow-up points at the interior split into two blocks: one consisting of isolated bubbles, the other of clustered bubbles. Furthermore, these theorems can be combined with Theorem \ref{t3} to produce blowing-up solutions of $(\mathcal{P}_{V,\varepsilon})$ exhibiting three types of blow-up points: interior isolated bubbles, interior clustered bubbles, and a boundary blow-up point.
\item
It would be natural to extend the results of this paper to similar boundary value problems with Robin conditions. We intend to return to these problems in future work.
\end{enumerate}
\end{rem}

The proof strategy of our results is based on refined asymptotic expansions of the gradient of the functional $I_\e$. The goal is to determine the equilibrium conditions that the blow-up parameters must satisfy. These conditions are obtained by evaluating the equation against vector fields corresponding to the dominant terms in the gradient with respect to the blow-up parameters. A careful analysis of these balancing conditions then provides the necessary information to prove our results.

\section{ Estimates for the Infinite-Dimensional Component}
Let $u_0$ be a non-degenerate solution of problem  $(\mathcal{P}_{V,0})$. It is well known (see \cite{BC}) that if the sequence $ (u_\e) $ of solutions of $(\mathcal{P}_{V,\e})$
 of the form \eqref{KR1} and satisfies conditions \eqref{KR2}, and \eqref{KR3},  then there exists a unique choice of parameters  $\lambda_{i,\e} $, $ a_{i,\e}$ and $v_\e $  such that: 
 
\begin{align}
 & u_\e= \alpha_{0,\e}\,  u_0 +\sum_{i=1}^N \alpha_{i,\e}  P\delta_{(a_{i,\e},\lambda_{i,\e})} +v_\e,  \qquad \mbox{ with }\quad \alpha_{0,\e}\longrightarrow 1,  \label{KA4} \\  
 & \begin{cases}
\alpha_{i,\e}\longrightarrow 1, \quad \l_{i,\e}d(a_{i,\e},\partial\O)\longrightarrow \infty  \quad as \quad \e\longrightarrow 0 \quad \forall i\in \lbrace 1,\cdots,N\rbrace , \\
\e_{ij} \longrightarrow 0  \quad \forall i\neq j , \quad \Vert v_\e \Vert\longrightarrow 0 \quad as \quad \e\longrightarrow 0 \quad \mbox{ and }  \quad v_\e \in E_{a_\e,\lambda_\e,u_0}, 
\end{cases}  \label{KA5}
\end{align}
where $ a_\e= (a_{1,\e},\cdots,a_{N,\e}) \in \O^N ,\lambda_\e=(\lambda_{1,\e},\cdots,\lambda_{N,\e})\in (0,\infty)^N$
 and $ E_{a_\e,\lambda_\e,u_0}$ denotes
\begin{align}\label{KA6}
 E_{a_\e,\lambda_\e,u_0} :=  \Big\{ v\in  H^1_0(\Omega): \Big(v,P\delta_{(a_{i,\e},\lambda_{i,\e})} \Big)
 & = \Big(v,\frac{\partial P\delta_{(a_{i,\e},\lambda_{i,\e})}}{\partial\l_{i,\e}}\Big)
= \Big(v,\frac{\partial P\delta_{(a_{i,\e},\lambda_{i,\e})}}{\partial (a_{i,\e})_j}\Big) \notag \\
& =(v,u_0)=0 \quad \forall \, \,  1\leqslant i \leqslant N ,  \quad \forall \, \,  1\leqslant j \leqslant n \Big\} . 
\end{align}
For simplicity of notation, we will henceforth write $\alpha_i $, $a_i $, $ \lambda_i $, $ E_{a,\lambda,u_0}  $ in place of  $ \alpha_{i,\e} $, $ a_{i,\e} $, $\lambda_{i,\e} $ and $ E_{a_\e,\lambda_\e,u_0} $,  respectively.
For $ \mu_0 $ positive small, we introduce the following sets: 
\begin{align}\label{KA7}
\mathbf{O} (N,\mu_0) = & \big\{ (\alpha, \lambda, a,v) \in \mathbb{R}_+^{N+1} \times \mathbb{R}_+^N \times \Omega^N \times H^1_0(\O): \vert\alpha_i -1\vert < \mu_0,   \,  \e \ln \l_i < \mu_0,\notag \\
 & \qquad  \qquad \qquad\qquad \l_id( a_i, \partial\O) > \mu_0^{-1}, \e_{ij} < \mu_0 , v \in  E_{a,\lambda,u_0}, \Vert v \Vert < \mu_0 \big\}, \\ 
\mathbf{B}(N,\mu_0) = & \big\{(\alpha,\lambda,a)\in \mathbb{R}_+^{N+1} \times \mathbb{R}_+^N \times \Omega^N :(\alpha,\lambda,a, 0)\in \mathbf{O} (N,\mu_0)\big\}. \label{KA8} 
\end{align}
As is customary in problems of this type, we begin by analyzing the $ v$-component of $ u_\e $, which represents the infinite-dimensional part. To this end, we perform an expansion of the functional $ I_\e $, defined in \eqref{funct}, with respect to $ v \in E_{a,\lambda,u_0} $ satisfying $ \Vert v\Vert < \mu_0 $. 
Let $ (\alpha,\lambda,a) $ $ \in \mathbf{B} (N,\mu_0) $, taking $  W =\alpha_0 \,u_0+ \sum_{i=1}^N \alpha_i P\delta_{(a_i,\lambda_i)} $,   and $ v \in  E_{a , \lambda,u_0} $ with $ \Vert v \Vert < \mu_0 $,  we see that
\begin{align}
 & I_{\varepsilon}(W+v) =I_{\varepsilon}(W)-\left\langle L_{\varepsilon}, v\right\rangle +\frac{1}{2} q_{\varepsilon}(v)+E_{\varepsilon}(v) ,  \quad \mbox{ where }   \label{KA14} \\
   &  \left\langle L_{\varepsilon}, v\right\rangle = \int_{\Omega}W^{p-\varepsilon}v \, \, ,  \qquad q_{\varepsilon}(v) =\int_{\Omega}\left|\nabla v \right|^{2} +\int_{\Omega}V\, v^{2} -(p-\varepsilon)\int_{\Omega}W^{p-1-\varepsilon} v^2 ,  \label{KA88} \\
 & E_{\varepsilon}(v)=O\left(\|v\|^{\operatorname{min}(3, p+1-\varepsilon)}\right) \, , \, \, E_{\varepsilon}^{\prime}(v)=O\left(\|v\|^{\operatorname{min}(2, p-\varepsilon)}\right) \, , \, \,\mbox{and}\notag\\
 &  E_{\varepsilon}^{\prime \prime} (v) = O\left(\|v\|^{\text {min }(1, p-1-\varepsilon)}\right) .  \label{KA15} 
\end{align} 
By closely following the proof of Proposition $2 $ in \cite{KK}, one finds that $q_\e$ is non-degenerate on the space $E_{a,\lambda,u_0}$.
 We now  address the infinite-dimensional variable 
$v$. Specifically, we prove  
\begin{pro}\label{p2}
Let $ n\geq 3 $, $\mu_0 $ positive small,  $ (\a,\l,a)\in  \mathbf{B}(N,\mu_0) $. Then, there exists a unique $ \ov{v}_\e \in E_{a,\lambda,u_0} $ which  extremizes $ I_\e (\a_0 \, u_0 +\sum_{i=1}^{N} \a_i \, P\d_{(a_i,\l_i)} + {v} ) $ with respect to $ v \in E_{a,\lambda,u_0} $ with $ \Vert v \Vert $ small.
As a result, we derive 
$$ \langle I_\e^{\prime} (\a_0 \, u_0 +\sum_{i=1}^{N} \a_i \, P\d_{(a_i,\l_i)} + \ov{v}_\e ) , h \rangle = 0 \quad \forall h \in E_{a,\lambda,u_0} . $$
Additionally, we have the following estimate: $ \Vert \ov{v}_\e \Vert \leq   c \,   R_{(\e,a,\lambda)} $, where 
  \be  R_{(\e,a,\lambda)}= \e+ \sum_{i=1}^N\big(R_{33}(i)+R_{34}(i)\big)+\sum_{j \neq i} \e_{ij}^{\frac{n+2}{2(n-2)}}\ln^{\frac{n+2}{2n}} \e_{ij}^{-1}+ (\mbox{ if }\, n\leq 5)\, \sum_{j\neq i} \e_{ij}  \label{Teal} \ee 
  
 \begin{align} & \mbox{ with } \quad    R_{33}(i) = \frac{1}{\lambda_i^{ (n-2) / {2}}}  \, \,  (\mbox{if } n \leq 5 ) \quad ; \quad \frac{ \ln^{ 2 / 3}  ( \lambda_i ) }{\lambda_i^{2}} \, \,   (\mbox{if }  n=6) \quad ; \quad  \frac{1}{\lambda_i^{2}} \, \,  (\mbox{if }  n \geq 7 ) ,  \label{R33}\\
 &\qquad\quad R_{34}(i)= \frac{1}{(\l_id_i)^{(n+2)/2}} +(\mbox{if}\, \, n\leq 5)\, \frac{1}{(\l_id_i)^{n-2}} +(\mbox{if}\,\, n=6)\, \frac{\ln (\l_i d_i ) }{(\l_id_i)^4}.\label{R34}
 \end{align}
\end{pro}
 \begin{pf} From the implicit function theorem, in conjunction with estimate \eqref{KA15} and the fact that $q_\e$ is non-degenerate on the space $E_{a,\lambda,u_0} $, it follows that for small $\e$, there exists $\ov{v}_\e \in E_{a,\lambda,u_0}  $ with  $ \Vert \ov{v}_\e\Vert =O(\Vert L_\e \Vert)$, where $ L_\e $ is defined in \eqref{KA88}. Thus, the next step is to estimate $ \| L  _\e\|  $. \\
 Writing $ W:= \alpha_0 \, u_0 +\sum_{i=1}^{N} \alpha_i \, P\delta_i$ with $ P\delta_i :=  P\delta_{(a_i,\lambda_i)} $, 
 we obtain
\begin{align} 
  L_\e(v)   =  &\alpha_0^{p-\e} \int_{\O} u_0^{p-\e} v +  \sum_{i=1}^N \alpha_i^{p-\e} \int_{\O} P\delta_i^{p-\e} v   \notag \\
    &  + (\mbox{if }\, n\leq 5) O  \Big[ \sum_{j\neq i} \int_{\O} P\delta_i^{p-1-\e} P\delta_j  | v|  + \sum_{i=1}^N \int_\O P\delta_i ^{p-1-\e} u_0| v |  + \sum_{i=1}^N \int_\O P\delta_i  u_0^{p-1-\e} | v | \Big ]   \notag  \\
   &  + (\mbox{if } \, n\geq 6) O\Big[ \sum_{j\neq i} \int_{\O} (P\delta_i P\delta_j)^{\frac{p-\e}{2}}  | v|  + \sum_{i=1}^N \int_\O (P\delta_i u_0 )^{\frac{p-\e}{2}} | v | \Big] . \label{k1}
\end{align}
Notice that, since $ \e \ln \l_i \to 0 $ and $ \O $ is bounded, Taylor's expansion implies that 
 \begin{align}\label{eq:k20}
\d_i^{-\e} &= c_0^{-\e} \l_i^{\frac{ - \e (n-2)}{2}} \big(1+\frac{n-2}{2}\e \, \ln (1+\l_i^{2}|x-a_i|^2)\big)+O\big(\e^2 \, \ln^2 (1+\l_i^{2}|x-a_i|^2)\big)\notag\\
&=1+o(1).
 \end{align}
Invoking Estimate $E2$ in \cite{B}, Lemma $6.6$ in \cite{BVP}, Lemma \ref{A1}, and equation \eqref{eq:k20}, leads to 
\begin{align}
 & \int_{\O} (P\delta_i \, P\delta_j)^{\frac{p-\e}{2}}  | v | 
\leq c \Vert v \Vert \Big(\int_{\O} (\d_i \, \d_j)^{\frac{n}{n-2}}\Big)^{\frac{n+2}{2n}}
\leq c  \e_{ij}^{\frac{n+2}{2(n-2)}}\ln^{\frac{n+2}{2n}} (\e_{ij}^{-1} ) \| v \|,  \label{k3} \\
 &  \int_{\O}P\delta_i^{\frac{p-\e}{2}} \, u_0^{\frac{p-\e}{2}} | v | \leq c \int_\O \delta_i^{{(n+2) } / [2(n-2)]} | v | 
  \leq  c \| v \| \frac{( \ln \lambda_i)^{ (n+2) / (2n)}}{\lambda_i^{ (n+2) / {4}}}. \label{k4}
  \end{align}
  Observing that  $ 1< \frac{2n}{n+2} < \frac{8n}{n^2-4} $ for $ n \leq 5 $, we may invoke Estimate $E2$ in \cite{B}, Lemma $6.6$ in \cite{BVP} and Lemma \ref{A1} to estimate the remainder terms  in \eqref{k1}: 
  \begin{align}
 & \int_{\O} P\delta_i^{p-1-\e} \, P\delta_j | v | \leq c \| v \| \Big( \int_\O \delta_i^{ 8n / (n^2 -4) } \, \delta_j ^{ 2n / (n+2 )} \Big) ^{ (n+2) / (2n) } 
\leq c \| v \| \e_{ij}, \label{k5} \\ 
 & \int_{\O} P\delta_i^{p-1-\e} \, u_0 | v | \leq c \| v \| \Big( \int_{\O} \delta_i^{ 8n / ( n^2-4) } \Big)^{ ( n+2 ) / {2n}} \leq c \| v \| R_{33}(i),  \label{k6} \\
 & \int_{\O} P\delta_i \, u_0^{p-1-\e} \, | v | \leq c \| v \| \Big( \int_{\O} \delta_i^{ 2n / ( n + 2 ) } \Big)^{ ( n+2 ) / {2n}}  \leq c \| v \| R_{33}(i) . \label{k7}
\end{align} 
To handle the first integral on the right-hand side of \eqref{k1}, we proceed as follows:
\be \label{k8}
\int_{\O} u_0^{p-\e} \, v  =  (u_0,v) +O(\int_{\O} \e | v | )  = O(\e \| v \|) . 
\ee 
To handle the second integral on the right-hand side of \eqref{k1}, let $ \theta_i := \d_i - P \d_i$, we make use of \eqref{eq:k20}, yielding:
\begin{align*}
\int_{\O}& P\delta_i^{p-\e}v =\int_\O\d_i^{p-\e}v+O\big(\int_\O\d_i^{p-1}\theta_i |v|\big)\notag\\
& = c_0^{-\e} \l_i^{ - \frac{\e (n-2) }{2}} \int_{\O} \delta_i^{p}v+ O \Big(  \e \int_{\O} \delta_i^{p} \ln(1+\lambda_{i}^2 | x-a_i| ^{2}) | v | \Big)+O\Big(\Big(\int_\O (\d_i^{p-1}\theta_i)^{\frac{2n}{n+2}}\Big)^{\frac{n+2}{2n}} \|v\|\Big) . 
\end{align*}
However, because $ v \in E_{a,\lambda,u_0} $, it follows that
\begin{align*}
\int_{\O} \delta_i^{p}v &  = -\int_{\O} V\, P\delta_i \, v = O\Big(\int_{\O} \delta_i | v | \Big)= O\Big(   \| v \| \big( \int _\O \delta_i ^{ 2n / (n+2 ) } \big) ^{ (n+2 ) / (2n) }\Big)=O\Big(   R_{33}(i) \| v \| \Big),
\end{align*}
 where  $ R_{33}(i) $ is defined in \eqref{R33}.  \\
Note that, writing $B_i=B(a_i,d_i)$ and $B_i^c=\O\setminus B_i$ with $d_i=d(a_i,\partial\O)$, we observe that
\begin{align*}
\Big(\int_{B_i^c} (\d_i^{p-1}\theta_i)^{\frac{2n}{n+2}}\Big)^{\frac{n+2}{2n}}&\leq \Big(\int_\O \d_i^{p+1}\Big)^{\frac{n+2}{2n}}\leq \frac{c\qquad}{(\l_id_i)^{(n+2)/2}},\\
\Big(\int_{B_i} (\d_i^{p-1}\theta_i)^{\frac{2n}{n+2}}\Big)^{\frac{n+2}{2n}}&\leq |\theta_i|_\infty \Big(\int_{B_i} \d_i^{\frac{8n}{n^2-4}}\Big)^{\frac{n+2}{2n}}\leq cR_{34}(i), 
\end{align*}
where $R_{34}(i)$ is defined in \eqref{R34}. 
We conclude that
\begin{equation}
\label{k11a}
\Big| \int_{\O} P\delta_i^{p-\e} v \Big| \leq c \| v \| \Big( R_{33}(i )+ R_{34}(i)+ \e \Big) . 
\end{equation}
Gathering together \eqref{k1}, \eqref{k3}–\eqref{k8}, and \eqref{k11a}, the desired conclusion follows directly.
\end{pf}

\section{Precise Estimates of the Gradient of the Functional}

In this section, we provide precise estimates for the gradient of the functional $I_\e$ defined in \eqref{funct}. To this end, we observe that for $u, h \in H^1_0(\O)$, the following relation holds:
\begin{align}\label{k12}
\left\langle I_{\e}^{\prime}( u),  h \right\rangle=\int_{\O}\n u.\n  h + \int_\O V\,u\, h - \int_{\O}| u|^{p-1-\e} u\,h .  
\end{align}
In \eqref{k12}, we consider the decomposition
$$
u =\a_0 \,u_0+\sum_{i=1}^{N}\a_i\, P\d_{(a_i,\lambda_i)}+\bar{v}_{\e}:=\a_0 \,u_0+ \sum_{i=1}^N \a_iP\d_i + \bar{v}_{\e}:=W+\bar{v} _\e,
$$
 where $\bar{v}_{\e}$ is defined in Proposition \ref{p2}. For simplicity of notation, we will henceforth write $v_\e$ in place of $\bar{v}_{\e}$.\\
 We will take $ h = \psi_i \in \lbrace P\d_{(a_i,\l_i)} ,\l_i\partial P\d_{(a_i,\lambda_i)}/\partial\l_i,\l_i^{-1}\partial P\d_{(a_i,\l_i)}/\partial a_i \rbrace $ with $1\leq i\leq N$.
 Our goal is to derive precise estimates for  $\left\langle I_{\e}^{\prime}( u), \psi_i\right\rangle $ for $1\leq i\leq N$. To this end, using \eqref{k12} and the assumption that $ v \in E_{a,\lambda,u_0} $, we deduce that  
\begin{align}\label{lion1}
\Big\langle I^{\prime}_\e (u),&\psi_i\Big\rangle   =  \a_0 (u_0,\psi_i)+\a_i \int_{\O} V \, P\d_i \, \psi_i +\sum_{j\neq i} \a_j \,\int_{\O}  V\, P\d_j \, \psi_i\notag\\
&+ \a_i \int_{\O} \n P\d_i \n \psi_i+ \sum_{j\neq i} \a_j \int_{\O} \n P\d_j \n \psi_i- \int_{\O} |u|^{p-1-\e} \, u \psi_i . 
\end{align}

 \subsection{Asymptotic Estimate in the Gluing Parameter}
 
 We begin by estimating the gradient of $I_\e$ with respect to the parameter $\a_0$. 
 \begin{pro}\label{rp41}
Let $ n\geq 3 $, $\mu_0 $ positive small,  $ (\a,\l,a)\in  \mathbf{B}(N,\mu_0) $, $ v \in E_{a, \l , u_0} $, and  $ u=\a_0 \, u_0 +\sum_{i=1}^{N} \a_i \, P\d_{(a_i,\l_i)} + {v} $. Then, the following estimate holds:
\begin{align*}
 & \left\langle I^{\prime}_\e(u),u_0\right\rangle = \a_0 c_1 \Big( 1 - \a_0^{p-1-\e} \Big) + O\Big(R_{\a 0}\Big), \qquad \mbox{ where } \\
  &  c_1 := \int_{\O} u_0^{p+1}(x) \, dx \qquad  \mbox{ and } \qquad R_{\a 0}= \e+\Vert {v} \Vert^{2}+\sum_{i=1}^N \Big( \frac{1}{\l_i^{4}}+ \frac{1}{\l_i^{ (n-2) / {2}}}  \Big). \end{align*}
\end{pro}
\begin{pf}
Notice that, using Lemma \ref{A1}, we obtain 
\be\label{r41}
(P\d_i,u_0)=\int_\O u_0^p P\d_i\leq c\int_\O \d_i\leq \frac{c}{\l_i^{ (n-2) / {2}}}.
\ee
Now, using \eqref{eq:k20}, we get
 \begin{align} \label{rm3} 
\int_{\O} \vert u \vert^{p-1-\e} \, u \, u_0  = & \int_{\O} W^{p-\e} \, u_0 + (p-\e)\int_{\O}W^{p-1-\e} \, {v} \, u_0  +O(\| v\|^2).
\end{align}
To proceed further, we first observe that for  $\beta>0$, there exists $c>0$ independent of $\e$ such that
\be\label{ba}
|u_0^{\beta -\e}(x)- u_0^\beta (x)|\leq c\, \e\quad \forall \, x\in\O. 
\ee
Second, using \eqref{eq:k20}, \eqref{ba} and Lemma \ref{A1}, we derive
\begin{align}\label{rm5}
\int_{\O} W^{p-1-\e} \, {v} \, u_0 &= \a_0^{p-1-\e} \int_{\O} u_0^{p-\e} \, {v} + \sum_{i=1}^N O \Big(  \int_{\O} u_0^{p-1-\e} P\d_i \vert v\vert + \int_{\O} P\d_i^{p-1-\e} \vert {v} \vert\Big)\notag\\
&=   O\Big(\e \Vert {v} \Vert + \| v\|\sum_{i=1}^N R_{33}(i)\Big), 
\end{align} 
where we have used the fact that $ {v} \in E_{a,\lambda,u_0} $, and where $R_{33}$ is defined in \eqref{R33}.

Now, arguing as in \eqref{rm5}, we get
\begin{align}\label{rm6}
\int_{\O}  W^{p-\e} \, u_0 & = \a_0^{p-\e} \int_{\O} u_0^{p+1-\e} + \sum_{i=1}^N O \Big(\int_{\O} u_0^{p-\e} \, P\d_i +  \int_{\O} P\d_i^{p-\e} \, u_0 \Big)\notag \\
&   = \a_0^{p-\e} c_1 + O \Big(\e+\sum_{i=1}^{N} \frac{1}{\l_i^{ (n-2) / {2}}}\Big) . 
\end{align}
Putting together \eqref{r41}-\eqref{rm6} in \eqref{k12}, and using the fact that $v\in E_{a,\l, u_0}$, we conclude the proof of the desired result.
\end{pf}

 Our next goal is to estimate the gradient of $I_\e$ with respect to $ \a_i $, for $1\leq i \leq N$. 
\begin{pro}\label{rp42}
Let $ n\geq 3 $, $\mu_0 $ positive small,  $ (\a,\l,a)\in  \mathbf{B}(N,\mu_0) $, $ v \in E_{a, \l , u_0} $, and  $ u=\a_0 \, u_0 +\sum_{i=1}^{N} \a_i \, P\d_{(a_i,\l_i)} + {v} $. Then, for $ 1\leq i\leq N $,  the following estimate holds:
\begin{align*}
\left\langle I^{\prime}_\e(u),P\d_{(a_i,\l_i)}\right\rangle = \a_i \, S_n \Big(1-\a_i^{p-1-\e} c_0^{-\e}\l_i^{ { - \e(n-2) } / {2}} \Big)+O\Big(R_{\a i} \Big),
\end{align*}
where $ S_n = c_0^{{2n} / (n-2)} \int_{\mathbb{R}^{n}} \frac{dx}{(1+\vert x \vert^{2})^{n}} $, $\bar{c}_1=c_0^{p+1}\int_{\R^n}\frac{\ln (1+|x|^2)}{(1+|x|^2)^n}dx$, and where, denoting $d_i=d(a_i,\partial\O)$, the term $R_{\a i} $ is defined as
$$ R_{\a i}  :=  \Vert v \Vert^{2}+\e +\frac{1}{(\l_id_i)^{n-2}}  + \frac{1}{\l_i^{2}} + \sum_{j=1}^N \Big[ \frac{1}{\l_j^{ (n-2) / {2}}} + \frac{1}{\l_j^4} \Big]  +  \sum_{j\neq i } \e_{ij } . $$ 
\end{pro}
\begin{pf}
 By applying Estimate $(B2)$ in \cite{R-G} and Lemma \ref{A1}, we derive
  \begin{align}\label{k50} 
  (P\d_i,P\d _i )&= S_n+ O\Big(\frac{1}{(\l_id_i)^{n-2}}+ (\mbox{if } \, n=3)\frac{1}{\l_i}+(\mbox{if } \, n=4)\frac{\ln \l_i}{\l_i^{2}} +(\mbox{if } \, n\geq 5)\frac{1}{\l_i^{2}} \Big). 
\end{align}
Moreover, for $ j \neq i $,  Lemma 6.6 in \cite{BVP} and  Lemma \ref{A1} yield that 
\begin{align}\label{k51}
(P\d_j,P\d_i)&= \int_{\O} \n P\d_j \n P\d_i+\int_{\O} V \, P\d_j \,P\d_i=-\int_\O \D P\d_iP\d_j +\int_{\O} V \, P\d_j \,P\d_i = O\Big
(\e_{ij}\Big).
\end{align}
However,  we have
\begin{align} \label{rm7} 
\int_{\O} \vert u \vert^{p-1-\e} \, u \, P\d_i = \int_{\O} W^{p-\e} \, P\d_i + (p-\e)\int_{\O}W^{p-1-\e} \, {v} \, P\d_i +O(\| v\|^2).
\end{align}
Furthermore, as in \eqref{rm5}, we have
\begin{align}\label{rm8}
\int_{\O} W^{p-1-\e} \, {v} \, P\d_i & = \int_{\O} \big(\sum_{j=1}^N \a_jP\d_j\big)^{p-1-\e} \, {v}P\d_i +O\Big(\| v  \|\sum_{j=1}^N R_{33}(j)\Big).
\end{align}
and, using \eqref{k11a}, we get 
\begin{align}
  \int_{\O} \big(\sum_{j=1}^N \a_jP\d_j\big)^{p-1-\e} \, & {v}P\d_i  = \a_i^{p-1-\e}\int_{\O}P\d_i^{p-\e} v  + (\mbox{if } n \geq 6 ) \sum_{j \neq i} O \Big(  \int_\O (P \d_i P \d_j)^{\frac{p-\e}{2}} | v | \Big) \notag  \\
   &  \quad + (\mbox{if } n \leq 5 ) \sum_{j \neq i} O \Big( \int_\O P\d_i ^{p-1\e} P \d_j | v | +  \int_\O P\d_j ^{p-1 - \e} P \d_i | v | \Big) \notag  \\
 & = O \Big( \| v \| \Big[ \e + R_{33}(i) + R_{34}(i) + \sum_{j\neq i} \big(\e_{ij}^{\frac{n+2}{2(n-2)}}(\ln \e_{ij}^{-1})^{\frac{n+2}{2n}} + \e_{ij} \big) \Big] \Big) . \label{rm11}
 \end{align}

We now turn to the first integral on the right-hand side of \eqref{rm7}. By arguing as in \eqref{rm5} and then using Lemma $A.1$ in \cite{BCH} and estimate \eqref{eq:k20}, we find that
\begin{align}\label{kkk4*}
\int_{\O} W^{p-\e} \, P\d_i 
 &=   \int_{\O} \big(\sum_{j=1}^N \a_jP\d_j\big)^{p-\e} \, P\d_i + O\Big( \sum_{j=1}^N \frac{1}{\l_j^{(n-2)/2}}\Big).
\end{align}
Concerning the integral in the right hand-side of \eqref{kkk4*},  we note that
\begin{align}\label{rm12}
 \int_{\O}\big(\sum_{j=1}^N\a_j P\d_j)^{p-\e}P\d_i   
&=\a_i^{p-\e}\int_{\O}( P\d_i)^{p+1-\e}+
 O\Big(\sum_{j\neq i}\e_{ij}\Big),
\end{align}
where  we have used  Lemma 6.6 from \cite{BVP} and \eqref{eq:k20}.\\
 Now, using \eqref{eq:k20},  we see that
\begin{align}\label{r19}
\int_{\O}P\d_i^{p+1-\e}
&=c_0^{-\e} \l_i^{\frac{ - \e (n-2)}{2}} \int_{\R^n}\d_i^{p+1} +O\Big(\e \int_{\R^n}\d_i^{p+1} \ln (1+\l_i^{2}|x-a_i|^2)\Big)\notag\\
&\qquad  +  O\Big(|\theta_i|_\infty\int_\O\d_i^p+\int_{\R\setminus \O}\d_i^{p+1}\Big)\notag\\
&=c_0^{-\e} \l_i^{\frac{ - \e (n-2)}{2}}S_n+ +O\big(\e+\frac{1}{(\l_id_i)^{n-2}}\big).
\end{align}
Putting together \eqref{rm12} and \eqref{r19}, we obtain
\begin{align}\label{rm14b}
 \int_{\O}\big(\sum_{j=1}^N\a_j P\d_j\big)^{p-\e}P\d_i 
  = c_0^{-\e} \a_i^{p-\e}\l_i^{\frac{ - \e (n-2)}{2}}S_n +O\Big(\e+\frac{1}{(\l_id_i)^{n-2}}+\sum_{j\neq i} \e_{ij}\Big). 
\end{align}
Combining the results in \eqref{kkk4*} and \eqref{rm14b}, we arrive at
\begin{align}\label{rm16}
\int_{\O} W^{p-\e} \, P\d_i 
 = c_0^{-\e} \a_i^{p-\e}\l_i^{\frac{ - \e (n-2)}{2}}S_n +O\Big(\e+\frac{1}{(\l_id_i)^{n-2}}+\sum_{j\neq i} \e_{ij}+ \sum_{j=1}^N \frac{1}{\l_j^{(n-2)/2}}\Big).
\end{align}
Using \eqref{rm7}, together with \eqref{rm11} and \eqref{rm16}, we derive:
\begin{align} \label{rm17} 
\int_{\O}  \vert  u \vert^{p-1-\e} \, u \, P\d_i = & c_0^{-\e}\a_i^{p-\e} \l_i^{\frac {-\e (n-2)}{2}}S_n \notag \\
 &  + O\Big(\|v\|^2 +\e+ \frac{1}{(\l_id_i)^{n-2}}+\sum_{j\neq i}\e_{ij}+\sum_{j=1}^N \Big[ \frac{ 1 } { \l_j ^{ (n-2) / {2 } }}  + \frac{1}{\l_j^4} \Big] \Big) .
\end{align} 
Our Proposition follows directly from the combination of \eqref{k12}, \eqref{r41}, \eqref{k50}, \eqref{k51} and \eqref{rm17}.
 \end{pf}
 
 \subsection{Asymptotic Estimate in the Concentration Rate}
 We now aim to analyze the gradient of $I_\e$ with respect to $ \l_i $, for $1\leq i \leq N$. 

\begin{pro}\label{rp43}
Let $ n\geq 3 $, $\mu_0 $ positive small,  $ (\a,\l,a)\in  \mathbf{B}(N,\mu_0) $, $ v \in E_{a, \l , u_0} $, and  $ u=\a_0 \, u_0 +\sum_{i=1}^{N} \a_i \, P\d_{(a_i,\l_i)} + {v} $. Then, for $ 1\leq i\leq N $,  the following estimate holds:
\begin{align*}
\Big\langle  & I^{\prime}_\e(u),\l_i \frac{\partial P\d_i}{\partial \l_i}\Big\rangle = \a_ic_0\bar{c}_2\frac{(n-2)}{2}\frac{H(a_i,a_i)}{\l_i^{n-2}}\left(1-2c_0^{-\e}\l_i^{-\frac{\e(n-2)}{2}}\a_i^{p-1-\e}\right)
  -\a_i \, V(a_i)  \frac{ \gamma_n }{\l_i^{2}}\\
  &+ \bar{c}_3 \, \e \, c_0^{-\e} \a_i^{p-\e} \l_i^{\frac{-\e (n-2)}{2}}+ \a_0 (1-\a_0 ^ {p-1-\e}) (u_0,\l_i\frac{\partial P\d_i}{\partial \l_i}) + \, c_0^{-\e}\a_i^{p-1-\e} \l_i^{\frac{-\e(n-2)}{2}}  \a_0 \bar{c}_2 \frac{(n-2)}{2}\frac{ u_0(a_i) }{\l_i^{ \frac{n-2} {2}}}\\
  &+ \sum_{j\neq i} \a_j\, c_0\bar{c}_2 \left(\l_i \frac{\partial \e_{ij}}{\partial \l_i}+\frac{(n-2)}{2}\frac{H(a_i,a_j)}{(\l_i\l_j)^{\frac{n-2}{2}}}\right) \Big(1- \sum_{ \ell = i, j } c_0^{-\e}\a_\ell ^{p-1-\e} \, \l_\ell ^{ - \frac{\e (n-2) }{2}}  \Big) + R_\l(i),   
\end{align*}

 \begin{align*}
\mbox{ where } &  R_\l(i) =O\Big(\e^2 +\| v\|^2   + \sum_{j\neq k} \big[\e_{kj}^{\frac{n}{n-2}} \ln ( \e_{kj}^{-1}) + \e_{kj}^2 \big] +  \sum_{j=1}^N \big[ \frac{ \ln \l_j }{ \l_j ^{n/2}} + \frac{1}{\l_j^4} + \frac{\ln ( \l_j d_j) }{(\l_jd_j)^{n}} \big]  \Big) \\
  & \qquad +(\mbox{if}\, n=3) \sum_{j=1}^N O\Big( \frac{1}{\l_j}+ \frac{1}{(\l_jd_j)^2}\Big), \\ 
  & \bar{c}_2 = \int_{\mathbb{R}^{n}} \frac{ c_0^{p}}{(1+\vert x \vert^{2})^{\frac{n+2}{2}}} dx \, \, ; \quad  \bar{c}_3 = \frac{(n-2)^{2}}{4} \int_{\mathbb{R}^{n}} \frac{c_0^{p+1} (\vert x \vert^{2}-1)\ln(1+\vert x \vert^{2})}{(1+\vert x \vert^{2})^{n+1}} \, dx  , \\ 
&  \gamma_3 = \gamma_4=0, \, \,  \gamma_n = \frac{(n-2)}{2} \, c_0^{2} \int_{\mathbb{R}^{n}} \frac{|x|^{2}-1}{(1+|x|^{2})^{n-1}} \, dx \mbox{  if } n\geq5  . \end{align*}
\end{pro}
\begin{pf}
We estimate each term on the right-hand side of \eqref{lion1} with $\psi_i=\l_i\partial P\d_i /\partial\l_i$. 
For the second integral in \eqref{lion1}, using Lemma \ref{A1}, we obtain for $ n \in \{ 3 , 4 \} $, 
\begin{align}\label{k53}
\int_{\O} V \, P\d_i \, \l_i \frac{\partial P\d_i}{\partial \l_i}=O\Big(\int_{\O}P\d_i^{2}\Big)=O\Big(\frac{1}{\l_i} \mbox{ if } n = 3 \, \, ; \quad \frac{ \ln \l _i }{\l_i ^2 } \mbox{ if } n = 4  \Big).  
\end{align}
For $ n\geq 5 $, expanding $ V $ around $a_i$ in $B_i:=B(a_i, d_i/2)$ and using  Lemma \ref{A1} once again, straightforward computations yield
\begin{align}\label{k54}
\int_{\O}& V\, P\d_i \, \l_i \frac{\partial P\d_i}{\partial \l_i}
= \int_{B_i}  V P\d_i \, \l_i \, \frac{\partial P\d_i}{\partial \l_i}+ O\Big(\int_{ \O \setminus B_i } P\d_i^{2}\Big) \notag \\
&=V(a_i)\int_{B_i} \d_i \, \l_i \, \frac{\partial \d_i}{\partial \l_i}+O\Big(\int_{B_i}|x-a_i|^2\d_i^2 + \int_{B_i} \big(\theta_i+|\l_i\frac{\partial\theta_i}{\partial\l_i}|\big)\d_i  +\frac{1}{\l_i^2 (\l_id_i)^{n-4}}\Big) .
\end{align}
Note that if $n\geq 5 $, we have
$$
\int_{B_i} \big(\theta_i+|\l_i\frac{\partial\theta_i}{\partial\l_i}|\big)\d_i \leq \frac{ c }{\l_i^{ ( n-2 ) / {2}}d_i^{n-2}}\int_{B_i}\d_i \leq \frac{ c }{\l_i^2(\l_id_i)^{n-4}} .
$$
Now, we observe that
$$
\int_{B_i}|x-a_i|^2\d_i^2=O\big(\frac{1}{\l_i^4}\big)+(\mbox{if}\, n=4)\,O\big(\frac{d_i^2}{\l_i^2}\big)+(\mbox{if}\, n= 5)\,O\big(\frac{d_i}{\l_i^3}\big)+ \mbox{if}\, n= 6)\,O\big(\frac{\ln(\l_id_i)}{\l_i^4}\big).
$$
Combining these estimates with estimates $(52)$ and $(53)$ in \cite{BVP}, we find that \eqref{k54} becomes
\begin{align}\label{k54bis}
\int_{\O} V \, P\d_i \, \l_i \frac{\partial P\d_i}{\partial \l_i}
= & - V(a_i) \, \gamma_n \frac{ 1 }{\l_i^{2}}+(\mbox{if}\, n=4)\,O\Big(\frac{ \ln \l_i }{\l_i^2}\Big)+(\mbox{if}\, n\geq 5)\,O\Big(\frac{1}{\l_i^2(\l_id_i)^{n-4}}\Big)\notag\\
& + (\mbox{if}\, n=6)\,O\Big(\frac{\ln(\l_id_i)}{\l_i^4}\Big) +  (\mbox{if}\, n=3)\,O\Big(\frac{1 }{\l_i }\Big).
\end{align}
We remark that 
$$ \frac{1}{\l_i^2(\l_id_i)^{n-4}} = \Big( \frac{1}{\l_i^{n/2}}\Big)^{4/n} \Big( \frac{ 1 }{(\l_id_i)^{n}} \Big)^{(n-4)/n} \leq c \Big( \frac{1}{\l_i^{n/2}} +  \frac{ 1 }{(\l_id_i)^{n}} \Big). $$ 
For the third term appearing on the right-hand side of \eqref{lion1}, for $ n \geq 5$, using Lemma \ref{A1}, we get 
\begin{align}\label{k55}
\Big| \int_{\O} V \, P\d_j \, \l_i \frac{\partial P\d_i}{\partial \l_i} \Big|  \leq c \int_\O \Big[ \d_j \d_i\Big]^{\frac{n-4}{n-2} } \d_j ^\frac{2}{n-2}  \d_i ^\frac{2}{n-2}  & \leq c \Big( \e_{ij}^\frac{n}{n-2} \ln ( \e_{ij}^{-1})  \Big)^{ \frac{n-4}{n}} \Big( \frac{ \ln \l_j }{ \l_j ^{n/2}} \Big)^\frac{2}{n}  \Big( \frac{ \ln \l_i }{ \l_i ^{n/2} } \Big)^\frac{2}{n}  \notag  \\ 
  & \leq c  \e_{ij}^\frac{n}{n-2} \ln ( \e_{ij}^{-1})  + \frac{ \ln \l_j }{ \l_j ^{n/2}} + \frac{ \ln \l_i }{ \l_i ^{n/2} } 
 \end{align}
and for $ n \in \{3,4\} $, we have 
$$ \Big| \int_{\O} V \, P\d_j \, \l_i \frac{\partial P\d_i}{\partial \l_i} \Big|  \leq c \Big( \int_\O \d_j ^2 \Big)^{\frac{1}{2}} \Big( \int_\O \d_i ^2 \Big)^{\frac{1}{2}} \leq c \Big( \frac{1}{ \sqrt{ \l_j \l_ i } } \, \, \mbox{ if } n = 3 \, \, ; \, \,  \frac{ \sqrt{ \ln \l_ j } } { \l_j }  \frac{ \sqrt{ \ln \l_ i } } { \l_i } \, \, \mbox{ if } n = 4 \Big) . $$
 Estimates for the fourth and fifth terms on the right-hand side of \eqref{lion1} are provided in \cite{BCH} as follows:
\begin{align*}
& \int_{\O} \n P\d_i \n (\l_i\frac{\partial P\d_i}{\partial \l_i})=c_0\bar{c}_2 \frac{(n-2)}{2}\frac{H(a_i,a_i)}{\l_i^{(n-2)/2}}+O\left(\frac{\ln(\l_id_i)}{(\l_id_i)^n}\right) \\
 &\int_{\O} \n P\d_j \n (\l_i \frac{\partial P\d_i}{\partial \l_i})=c_0\bar{c}_2\left(\l_i\frac{\partial\e_{ij}}{\partial\l_i}+\frac{(n-2)}{2}\frac{H(a_i,a_j)}{(\l_i\l_j)^{\frac{n-2}{2}}}\right)+O\Big(\e_{ij}^{\frac{n}{n-2}}\ln(\e_{ij}^{-1})+\sum_{l=i,j}\frac{\ln(\l_ld_l)}{(\l_ld_l)^n}\Big).
\end{align*} 
It remains to estimate the nonlinear term in \eqref{lion1}.  We have
\begin{align} \label{ms1} 
\int_{\O} &\vert  u \vert^{p-1-\e} \, u \, \l_i\frac{\partial P\d_i}{\partial\l_i}
=\int_{\O} W^{p-\e} \, \l_i\frac{\partial P\d_i}{\partial\l_i} + (p-\e)\int_{\O}W^{p-1-\e} \, {v} \, \l_i\frac{\partial P\d_i}{\partial\l_i} +O(\| v\|^2) . 
\end{align}
For the second integral in \eqref{ms1}, following the computations done in \eqref{rm8}-\eqref{rm11}, we get 
\begin{align} \label{ms2} 
\int_{\O}W^{p-1-\e} \, {v} \, \l_i \frac{\partial P\d_i}{\partial\l_i} = & \a_i^{p-1-\e}\int_{\O}P\d_i^{p-1-\e}\l_i\frac{\partial P\d_i}{\partial\l_i} v + \sum O\Big(\e_{ij}^{\frac{n+2}{n-2}}\ln^{\frac{n+2}{n}}\e_{ij}^{-1} + \e_{ij}^2 \Big)\notag\\
& +\,O\Big(\sum_{j=1}^N\frac{1}{\l_j^{n-2}} + \frac{1}{ \l_j ^4 }\Big) +(\mbox{ if } \, n=6)\,O\Big(\sum_{j=1}^N\frac{\ln^{4/3}\l_j}{\l_j^4})\Big)\, + O(\| v\|^2 ) .
\end{align}
In addition, following the proof of \eqref{k11a}, we get 
 \be\label{last}
 p\int_{\O}P\d_i^{p-1-\e} \l_i\frac{\partial P\d_i}{\partial\l_i}v = O\big(\| v\| \big[ R_{33}(i) + \e +  R_{34}(i) \big]\big),
 \ee
 where $ R_{33}(i)$ and $ R_{34}(i)$ are defined in \eqref{R33} and \eqref{R34} respectively. This completes the estimate of the second integral in \eqref{ms1}.
We now turn to the first integral on the right-hand side of \eqref{ms1}. We observe that
\begin{align}
\int_{\O} W^{p-\e} \, \l_i\frac{\partial P\d_i}{\partial\l_i} 
& =  \int_{\O} \big(\sum_{j=1}^N \a_jP\d_j\big)^{p-\e} \, \l_i\frac{\partial P\d_i}{\partial\l_i}  + ( p - \e ) \int_{\O} \big(\sum_{j=1}^N \a_jP\d_j\big)^{p-\e - 1 } ( \a_0 u_0 ) \, \l_i\frac{\partial P\d_i}{\partial\l_i}  \notag\\
&+ \int_{\O}( \a_0  u_0 ) ^{p-\e} \, \l_i\frac{\partial P\d_i}{\partial\l_i}   
 + O \Big( \int_{\O}  \big(\sum_{j=1}^N \a_jP\d_j\big)^{ \frac{p-\e - 1 }{ 2 } } u_0 ^{ \frac{p - \e + 1 }{ 2 } } P\d_i   \Big)\notag\\
 & + (\mbox{if } n=3) O \Big( \int_\O \Big[\big(\sum_{j=1}^N \a_jP\d_j\big) u_0^4 P\d_i + \big(\sum_{j=1}^N \a_jP\d_j\big)^{3-\e} u_0^2 P\d_i \Big]\Big). \label{kkk4}
\end{align}
To estimate \eqref{kkk4}, first, the two last integrals satisfy: for $ n = 3 $, we have
\be \label{13kk} \int_\O \big(\sum_{j=1}^N \a_jP\d_j\big) u_0^4 P\d_i  + \int_\O \big(\sum_{j=1}^N \a_jP\d_j\big)^{3-\e}  u_0^2 P\d_i \leq c  \sum_{ j=1 }^N \int_\O (\d_j^{ 2 } + \d_j^4 ) \leq c \sum_{ j=1 }^N \frac{ 1 }{ \l_j  }
. \ee
Second, using Lemma \ref{A1}, \eqref{ba} and \eqref{r41}, the third and the fourth integrals can be estimated as 
\begin{align}
&  \int_{\O}( \a_0  u_0 ) ^{p-\e} \, \l_i\frac{\partial P\d_i}{\partial\l_i} = \a_0^{p-\e}(u_0,\l_i\frac{\partial P\d_i}{\partial \l_i}) +
 O \Big( \frac{ \e} { \l_i ^{ {(n-2)} / { 2 } }} \Big) , \label{11kk}\\
 & \int_{\O}  \big(\sum_{j=1}^N \a_jP\d_j\big)^{ (p-\e - 1 ) / 2 } u_0 ^{ (p - \e + 1 ) / 2 }  P\d_i  \leq c  \sum  \int_\O \d_j ^{ n / (n-2) }  \leq c \sum \frac{ \ln \l_j }{ \l_j ^{n/2}} . \label{12kk}
 \end{align}
Concerning the first integral in the right hand-side of \eqref{kkk4},  we write
\begin{align}\label{rm14}
 \int_{\O}&\big(\sum_{j=1}^N\a_j P\d_j\big)^{p-\e}\l_i\frac{\partial P\d_i}{\partial\l_i}    = \a_i^{p-\e}\int_{\O}P\d_i^{p-\e} \l_i\frac{\partial P\d_i}{\partial\l_i}  +(p-\e)\sum_{j\neq i}\a_i^{p-1-\e}\a_j\int_{\O}P\d_i^{p-1-\e} P\d_j \l_i\frac{\partial P\d_i}{\partial\l_i} \notag\\
 &+ \sum_{j\neq i}\a_j^{p-\e}\int_{\O} P\d_j^{p-\e}\l_i\frac{\partial P\d_i}{\partial\l_i}  + \sum_{j\neq r} O \Big( \int_\O( \d_k \d_j)^{ n / (n-2)} \, \, \mbox{ if } n \geq 4; \quad \int_\O \d_k^4  \d_j ^ 2 \, \, \mbox{ if } n = 3 \Big) . 
\end{align}
We now proceed to estimate each integral appearing on the right-hand side of \eqref{rm14}. We remark that the two last integrals can be deduced from Estimate $E2$ in \cite{B} and Lemma $6.6$ in \cite{BVP}.  For the first integral, we write
\begin{align}\label{k43}
\int_{\O}P \d_i^{p-\e} \, \l_i \frac{\partial P\d_i}{\partial \l_i} & = \int_{\O} \d_i^{p-\e} \, \l_i \frac{\partial \d_i}{\partial \l_i}-\int_{\O} \d_i^{p-\e} \, \l_i \frac{\partial \theta_i}{\partial \l_i} -(p-\e)\int_{\O} \d_i^{p-1-\e} \, \l_i \frac{\partial \d_i}{\partial \l_i}\theta_i\notag\\
&\qquad
+O\Big(\int_\O \d_i^{p-1}\theta_i\big(\theta_i+\big|\l_i\frac{\partial \theta_i}{\partial\l_i}\big|\big)\Big).
\end{align}
However, following the proof of (91) in \cite{EM}, we easily see that
\be\label{l1}
\int_{\O} \d_i^{p-\e} \, \l_i \frac{\partial \d_i}{\partial \l_i}=-\bar{c}_3\e \l_i^{-\e(n-2)/2} +O\Big(\e^2 +\frac{\ln(\l_id_i)}{(\l_id_i)^n}\Big).
\ee
Using  Lemma \ref{A1}, let $ B_i := B(a_i, d_i / 2 )$, we get 
\begin{align} \label{l2}
\int_\O \d_i^{p-1}\theta_i\big(\theta_i+\big|\l_i\frac{\partial \theta_i}{\partial\l_i}\big|\big) & \leq \frac{ c }{ ( \l _i d_i^2)^{n-2} } \int_{ B_i} \d_i ^{p-1} + c \int_{ \O \setminus B_i } \d_i ^{ p+1 }  \notag \\
 & = O\Big( \frac{1}{(\l_id_i)^n} \Big)+   O\Big(\frac{1}{(\l_id_i)^2}\,(\mbox{if}\, n=3)\,;\, \frac{\ln(\l_id_i)}{(\l_id_i)^4}\,(\mbox{if}\, n=4) \Big),
\end{align}
\begin{align}\label{M1}
\int_{\O} \d_i^{p-\e} \, \l_i \frac{\partial \theta_i}{\partial \l_i} &=\int_{B_i} \d_i^{p-\e} \, \l_i \frac{\partial \theta_i}{\partial \l_i}+ O\Big(\frac{1}{(\l_id_i)^n}\Big)\notag\\
&=-\frac{(n-2)}{2}c_0^{-\e}\l_i^{\frac{-\e(n-2)}{2}}c_0\bar{c}_2\frac{H(a_i,a_i)}{\l_i^{n-2}}+O\Big(\frac{\e}{(\l_id_i)^{n-2}}+\frac{\ln(\l_id_i)}{(\l_id_i)^n}\Big).
\end{align}
By arguing in the same manner as in \eqref{M1}, we obtain
\begin{align}\label{M2}
-p\int_{\O} \d_i^{p-1-\e} \, \l_i \frac{\partial \d_i}{\partial \l_i}\theta_i=\frac{(n-2)}{2}c_0^{-\e}\l_i^{\frac{-\e(n-2)}{2}}c_0\bar{c}_2\frac{H(a_i,a_i)}{\l_i^{n-2}}+O\Big(\frac{\e}{(\l_id_i)^{n-2}}+\frac{\ln(\l_id_i)}{(\l_id_i)^n}\Big).
\end{align}
Combining the relations \eqref{k43}–\eqref{M2}, we deduce
\begin{align}\label{M3}
\int_{\O} P \d_i^{p-\e} \, \l_i \frac{\partial P\d_i}{\partial \l_i}   = & - \bar{c}_3 \, c_0^{-\e} \, \l_i^{ - \e (n-2) / {2}} \e+ (n-2)c_0^{-\e}\l_i^{\frac{-\e(n-2)}{2}}c_0\bar{c}_2\frac{H(a_i,a_i)}{\l_i^{n-2}}\notag\\
&+ O\Big( \e^{2}+\frac{\ln(\l_id_i)}{(\l_id_i)^n}\Big)+ (\,\mbox{if}\,\, n=3)\,O\Big(\frac{1}{(\l_id_i)^2}\Big). 
\end{align}
For the second integral on the right-hand side of \eqref{rm14}, we use Lemma \ref{A1} and \eqref{eq:k20} to write
 \begin{align}\label{bk0}
 p\int_{\O}P\d_i^{p-1-\e}\l_i\frac{\partial P\d_i}{\partial\l_i} P\d_j &= p\int_{\O}\d_i^{p-1-\e}\l_i\frac{\partial \d_i}{\partial\l_i} P\d_j +O\Big(\int_\O \d_i^{p-1}\d_j\theta_i\Big).
\end{align}
Note that, by estimate $F1$ in \cite{B}, Lemma $6.6$ in \cite{BVP}, and Lemma \ref{A1}, we obtain the following estimates: for $n=3$,
\be\label{bk1}
\int_\O \d_i^{p-1}\d_j\theta_i\leq \frac{1}{\sqrt{\l_i}d_i}\Big(\int_\O\d_i^4\d_j^2\Big)^{1/2}\Big(\int_\O\d_i^4\Big)^{1/2}\leq c\frac{\e_{ij}}{\l_id_i}=O\Big(\e_{ij}^2 +\frac{1}{(\l_id_i)^2}\Big),
\ee
and for $n\in\{4,5\}$
\begin{align*}
\int_\O \d_i^{p-1}\d_j\theta_i&\leq  |\theta_i|_\infty^{2/(n-2)}\int_{B_i} (\d_i\d_j)\d_i^{2/(n-2)}+\int_{\O\setminus B_i}(\d_i\d_j)\d_i^{4/(n-2)}\\
&\leq c\e_{ij}\ln^{(n-2)/n}(\e_{ij}^{-1})\frac{\ln^{2/n}(\l_id_i)}{(\l_id_i)^2}=O\Big(\e_{ij}^{n/(n-2)}\ln\e_{ij}^{-1}+ \frac{\ln(\l_id_i)}{(\l_id_i)^n}\Big),
\end{align*}
and for $n\geq 6$
\begin{align*}
\int_\O \d_i^{p-1}\d_j\theta_i&\leq c\e_{ij}\ln^{(n-2)/n}(\e_{ij}^{-1})\frac{1}{(\l_id_i)^2}=O\Big(\e_{ij}^{n/(n-2)}\ln(\e_{ij}^{-1})+ \frac{\ln(\l_id_i)}{(\l_id_i)^n}\Big).
\end{align*}
However, using Lemma $A7$ in \cite{BCH}, we find that
\begin{align*}
p&\int_{\O}\d_i^{p-1-\e}\l_i\frac{\partial \d_i}{\partial\l_i} P\d_j=c_0^{-\e}\l_i^{-\frac{\e(n-2)}{2}}p\int_{\O}\d_i^{p-1}\l_i\frac{\partial \d_i}{\partial\l_i} P\d_j +O\Big(\e \int_\O \d_i^p\ln (1+\l_i^2|x-a_i|^2)\d_j\Big)\\
&=c_0^{-\e}\l_i^{-\frac{\e(n-2)}{2}}\int_{\O}\n \big(\l_i\frac{\partial P\d_i}{\partial\l_i}\big) \n P\d_j+O\Big(\e\int_\O \d_i^{2/(n-2)}\ln (1+\l_i^2|x-a_i|^2)\d_i^{n/(n-2)}\d_j\Big)\notag\\
&=c_0^{-\e} \l_i^{ \frac{- \e (n-2)}{ 2}} c_0\bar{c}_2\Big(\l_i \frac{\partial \e_{ij}}{\partial \l_i}+ \frac{(n-2)}{2} \frac{H(a_i,a_j)}{(\l_i\l_j)^{\frac{n-2}{2}}}\Big)+O\Big(\e_{ij}^{\frac{n}{n-2}}\ln(\e_{ij}^{-1}) +\e\e_{ij}+\sum_{l=i,j}\frac{\ln(\l_ld_l)}{(\l_ld_l)^n}\Big).
\end{align*}
Thus, \eqref{bk0} becomes
 \begin{align}\label{M6}
 p\int_{\O}&P\d_i^{p-1-\e}\l_i\frac{\partial P\d_i}{\partial\l_i} P\d_j =c_0^{-\e} \l_i^{ - \e (n-2) / {2}} c_0\bar{c}_2\Big(\l_i \frac{\partial \e_{ij}}{\partial \l_i}+ \frac{(n-2)}{2} \frac{H(a_i,a_j)}{(\l_i\l_j)^{\frac{n-2}{2}}}\Big)+O\big(\e\e_{ij}\big)\notag\\
 &+(\mbox{if}\,\, n\geq 4)\,\, O\Big(\e_{ij}^{\frac{n}{n-2}}\ln(\e_{ij}^{-1}) + \sum_{l=i,j}\frac{\ln(\l_ld_l)}{(\l_ld_l)^n} \Big)+(\mbox{if}\,\, n=3)\,\, O\Big(\e_{ij}^2+ \sum_{l=i,j}\frac{1}{(\l_ld_l)^2} \Big).
 \end{align}
 The  third  integral on the right-hand side of \eqref{rm14} is handled as in the proof of \eqref{M6}. This yields
\begin{align}\label{M5}
\int_{\O} &P\d_j^{p-\e} \, \l_i \frac{\partial  P\d_i}{\partial \l_i}  = c_0^{-\e} \l_j^{ - \e (n-2) / {2}} c_0\bar{c}_2\Big(\l_i \frac{\partial \e_{ij}}{\partial \l_i}+ \frac{(n-2)}{2} \frac{H(a_i,a_j)}{(\l_i\l_j)^{\frac{n-2}{2}}}\Big)+O\big(\e\e_{ij}\big)\notag\\
 &+(\mbox{if}\,\, n\geq 4)\,\, O\Big(\e_{ij}^{\frac{n}{n-2}}\ln(\e_{ij}^{-1}) + \sum_{l=i,j}\frac{\ln(\l_ld_l)}{(\l_ld_l)^n} \Big)+(\mbox{if}\,\, n=3)\,\, O\Big(\e_{ij}^2+ \sum_{l=i,j}\frac{1}{(\l_ld_l)^2} \Big).
\end{align}
Putting together \eqref{M3}, \eqref{M5}, and \eqref{M6}, the estimate of the first integral of \eqref{kkk4} follows. 

Regarding the second integral in the right hand-side of \eqref{kkk4},  we observe that
\begin{align*} 
 \int_{\O} &\big(\sum_{j=1}^N\a_j P\d_j\big)^{p-\e-1}  \a_0u_0  \,\l_i\frac{\partial P\d_i}{\partial\l_i}  \notag\\
 &\quad = \a_i^{p-\e-1}\a_0\int_{\O} P\d_i ^{p-1-\e} \l_i\frac{\partial P\d_i}{\partial\l_i}  u_0  + \sum_{ j \neq i } O \Big( \int_\O P\d_i^{ p-1-\e} P\d_j u_0 + \int_\O P\d_j^{p-1-\e} P\d_i\, u_0 \Big). 
 \end{align*}
 Note that, for $ n \geq 4 $, applying \eqref{eq:k20}, Lemma \ref{A1} and using the fact that
 $$
 \d_k^{4/(n-2)}\d_r= \big[(\d_k\d_r)^{n/(n-2)}\big]^{2/n} \big[\d_k^{\frac{n}{n-2}}\big]^{2/n}\big[\d_r^{\frac{n}{n-2}}\big]^{(n-4)/n} \leq c \Big((\d_k\d_r)^{n/(n-2)} + \sum \d_j^{\frac{n}{n-2}}\Big) ,
 $$
 we obtain, by using Estimate $E2$ in \cite{B} and Lemma $6.6$ in \cite{BVP}, 
\be \label{mo1}  \int_\O P\d_i ^{ p-1-\e} P\d_j u_0+ \int_\O P\d_j ^{p-1-\e} P\d_i u_0  \leq c \Big( \e_{ij}^{ \frac{n}{n-2}} \ln \e_{ij}^{-1} +  \sum \frac{ \ln( \l_k d_k) }{ \l_k ^{n/2}} \Big) . \ee 
For $ n = 3$, we once again use \eqref{eq:k20} and Lemma \ref{A1} to derive that
$$  
\int_\O P\d_i ^{ p-1-\e} P\d_j u_0 + \int_\O P\d_j ^{p-1-\e} P\d_i u_0 \leq  c \int_\O \d_i ^{ 4} \d_j  + c \int_\O \d_j ^{ 4} \d_i.
 $$
However, Lemma $6.6$ of \cite{BVP} yields 
$$
 \int_\O \d_k ^{ 4} \d_r \leq \Big( \int_\O (\d_k^2\d_r)^2 \Big)^\frac{1}{2} \Big( \int_\O \d_k ^{4} \Big) ^ \frac{1}{2} \leq  \frac{ c \e_{kr} }{\sqrt{\l_k} }   \leq c \e_{kr}^2    + \frac{c }{ \l_k} $$
 which implies that 
 $$
 \int_\O P\d_i ^{ p-1-\e} P\d_j u_0 + \int_\O P\d_j ^{p-1-\e} P\d_i u_0 \leq  c \Big(\e_{ij}^2   + \frac{1}{\l_i} +\frac{1}{\l_j}\Big).
 $$
 Putting together the above estimates, we obtain
 \begin{align}\label{rm15} 
 p\int_{\O} \big(&\sum_{j=1}^N\a_j P\d_j\big)^{p-\e-1}  \a_0u_0  \,\l_i\frac{\partial P\d_i}{\partial\l_i} 
  =p \a_i^{p-\e-1}\a_0\int_{\O} P\d_i ^{p-1-\e} \l_i\frac{\partial P\d_i}{\partial\l_i}  u_0 \notag\\
  & + (\mbox{ if } \, n\geq 4)\, O\Big(\sum_{j\neq i} \e_{ij}^{ n / (n-2)}  +  \sum_{j=1}^N\frac { \ln \l_j  }{ \l_j ^{n/2} }\Big)
  + (\mbox{ if } \, n=3)\, O\Big(\sum_{j\neq i}\e_{ij}^2   + \sum_{j=1}^N \frac{1}{\l_j}\Big).
 \end{align}
  However, let $ B_i := B(a_i, d_i/2) $, using Lemma \ref{A1}, \eqref{eq:k20} and the fact that $ u_0 $ is bounded and using the argument of $(76)$ in  \cite{KK},  we get   
  \begin{align}\label{M77}
p  \int_{\O}P\d_i^{p-1-\e}\l_i\frac{\partial P\d_i}{\partial\l_i} u_0 & =  c_0^{-\e} \l_i ^{ -\e (n-2)/2} p \int_{ B_i} \d_i^{p-1}\l_i\frac{\partial \d_i}{\partial\l_i} u_0 \notag \\
  & + O \Big( \int_{ \O \setminus B_i} \d_i ^p + \e \int_{B_i} \d_i ^p \ln ( 1 + \l_i ^2 | x-a_i|^2 ) + \int_{B_i} \d_i ^{p-1} \big[ \theta_i + \big| \l_i \frac{ \partial \theta_i}{ \partial \l_i} \big| \big] \Big) \notag \\
 &  =-c_0^{-\e} \l_i^{ - \frac{\e (n-2)} {2}}\frac{n-2}{2} \bar{c}_2\frac{u_0(a_i)}{\l_i^{\frac{n-2}{2}}}+O\Big(\frac{\e}{\l_i^{\frac{n-2}{2}}}+ \frac{\ln(\l_id_i)}{\l_i^{\frac{n+2}{2}}}+\frac{1}{\l_i^{\frac{n-2}{2}} ( \l_i d_i)^2}\Big)\notag\\
 & + (\mbox{if}\,\, n=3)\,O\Big(\frac{1}{\sqrt{\l_i}(\l_id_i)}\Big)+ (\mbox{if}\,\, n=4)\,O\Big(\frac{\ln(\l_id_i)}{\l_i ( \l_i d_i)^2 }\Big).
 \end{align}
 Notice that 
 $$ \frac{1}{\l_i^{ (n-2) / {2}} ( \l_i d_i)^2} = \big( \frac{1}{\l_i^{{n}/ {2}} }\big)^{(n-2)/ {n}} \big( \frac{ 1 }{( \l_i d_i)^n} \big)^{{2}/{n}} \leq \frac{ c }{ \l_i^{{n}/ {2}} } +   \frac{ c }{( \l_i d_i)^n} . $$
 Combining the results in \eqref{lion1}, \eqref{k54bis}- \eqref{ms2}, 
\eqref{kkk4}-\eqref{13kk}, \eqref{rm14}, and \eqref{M5}- \eqref{last},    we easily derive the desired result.
\end{pf}

\subsection{Asymptotic Estimate in the Concentration Points}

In this subsection, our aim is to establish an expansion for the gradient of the functional  $ {I}_\e$ with respect to the concentration points $a_i's$. Since Theorems \ref{t1}–\ref{t3} only require the case of a single blowing-up solutions, we begin with this case in the first step. Then, as Theorems \ref{t4}–\ref{t6}6 deal with interior multiple- bubbling solutions, we provide in the second step the desired expansion for  the interior blow-up points.
\subsubsection{The Case of a Single Blow-Up Point}
Since the case  $n = 4$ requires more refined asymptotic estimates, we handle  $n \geq 5$ in this section and postpone the necessary estimates for $n = 4$ to the next section. 
\begin{pro}\label{pas-geq5}
Let $ n\geq 5 $, $ (\a,\l,a)\in  \mathbf{B}(1,\mu_0) $ with $ \mu_0 $ positive small, $ v \in E_{a, \l , u_0} $ and  $ u= \a_0 \, u_0+ \a \, P\d_{(a,\l)}+  {v} $. Then, the following estimate holds:
\begin{align*}
\Big\langle I'_\e(u)&,\frac{1}{\l} \frac{\partial P\d}{\partial a}\Big\rangle  =  \a \,\rho_n \, \frac{\n V(a)}{\l^{3}}  + \a_0 \big(1-\a_0^{p-1-\e}\big)\big( u_0,\frac{1}{\l}\,\frac{\partial P\d}{\partial a}\big)\\
& -\frac{\a}{2}\,\frac{c_0\,\bar{c}_2}{\l^{n-1}}\frac{\partial H}{\partial a}(a,a)\big[1-2\a^{p-1-\e}\l^{-\e\frac{n-2}{2}}\big]-\a_0\,\a^{p-1-\e}\bar{c_5}c_0^{-\e}\l^{-\e\frac{n-2}{2}}p\frac{\n u_0(a)}{\l^{n/2}}+ R_s(a)  , 
\end{align*}
where $ \bar{c}_2 $  is defined in Proposition \ref{rp43}, $\bar{c}_5=\frac{n-2}{n}c_0^p\int_{\R^n}\frac{|x|^2}{(1+|x|^2)^{\frac{n+4}{2}}}$,   $ \rho_n := \frac{n-2}{n} \, c_0^{2} \int_{\mathbb{R}^{n}} \frac{|x|^{2}}{(1+|x|^{2})^{n-1}} \, dx $, and 
\be \label{chi}
R_s(a) = O\Big(\frac{ d }{(\l d)^{n-1}} + \frac{ 1 }{(\l d)^{n}} + \e^2 + \frac{d^{\frac{ p+1-\e}{ 2 }}}{\l^{n/2}} + \| v \|^2\Big) + (\mbox{if } n=5,6) \, o\big(\frac{1}{\l^{\frac{n}{2}}}\big) + (\mbox{if } n \geq 7)\,  O (\frac{ 1} { \l ^4 }) .  \ee 
\end{pro}
\begin{pf}
The proof proceeds in a manner similar to that of Proposition \ref{rp43}. We estimate each term on the right-hand side of \eqref{lion1} with $\psi_i=\frac{1}{\l} \frac{\partial P\d}{\partial a}$ and $N=1$. We remark that the  integral $\int_{\O} \n P\d \n (\frac{1}{\l} \frac{\partial P\d}{\partial a})$ is computed in Lemma A.10 of \cite{BCH}. 
For the second integral in \eqref{lion1}, let $ B := B(a, d/2) $, using Lemma \ref{A1}, we obtain 
\begin{align}\label{k54as}
\int_{\O}& V\, P\d \, \frac{1}{\l} \frac{\partial P\d}{\partial a}
= \int_{B } V\d \, \frac{1}{\l} \frac{\partial \d}{\partial a}+ O\Big(\int_{ \O \setminus B } \d|\frac{1}{\l} \frac{\partial \d}{\partial a}|+ \int_{\O} \d|\frac{1}{\l} \frac{\partial \theta}{\partial a}| +\int_{\O} \theta|\frac{1}{\l} \frac{\partial \d}{\partial a}| \Big),
\end{align}
and 
\begin{align*}
&\int_{ \O \setminus B } \d|\frac{1}{\l} \frac{\partial \d}{\partial a}|\leq c\int_{ \O \setminus B }\frac{\d^2}{\l|x-a|}\leq \frac{c}{\l^2(\l d)^{n-3}} = O\Big( \frac{ d ^2 }{(\l d)^{n-1}}  \Big) ,\\
&\int_{\O} |\frac{1}{\l} \frac{\partial \theta}{\partial a}|\d \leq \| \frac{1}{\l} \frac{\partial \theta}{\partial a}\|_{L^{\frac{2n}{n-2}}}\Big(\int_\O\d^{\frac{2n}{n+2}}\Big)^{\frac{n+2}{2n}} \leq c\frac{R_{33}}{(\l d)^{n/2}} = O \Big( \frac{1}{(\l d)^{n}} +  \chi( \l ) \Big),\\
& \int_{\O} \theta|\frac{1}{\l} \frac{\partial \d}{\partial a}| = O\Big(\frac{1}{\l^{\frac{n-2}{2}}d^{n-2}} \int_{\O}\frac{\d}{\l|x-a|}\Big) = O\Big(\frac{1}{\l (\l d)^{n-2}}\Big) = O \Big( \frac{d}{(\l d)^{n-1}}\Big)  .  
\end{align*}
Now, expanding $V$ around $a$ in $B:=B(a, d/2)$, we observe that
\begin{align*}
\int_{B} V\d  \frac{1}{\l} \frac{\partial \d}{\partial a} & = \int_{B} \n V(a) (x-a)\d \, \frac{1}{\l} \frac{\partial \d}{\partial a}+ O\Big(\frac{1}{\l}\int_{B}|x-a|^2\d^2\Big)\\
 & = \rho_n \frac{\n V(a)}{\l^{3}} +  O \Big(\frac{ d^3 }{(\l d)^{n-1}} + \frac{1}{ \l^4} \Big)  . 
\end{align*}
Combining these estimates, we find that \eqref{k54as} becomes
\be \label{k54bisas}
\int_{\O} V \, P\d \, \frac{1}{\l} \frac{\partial P\d}{\partial a} =  \rho_n \frac{\n V(a)}{\l^{3}} +  O \Big(\frac{ d }{(\l d)^{n-1}} +  \chi ( \l ) \Big)  .
\ee 

 Before studying the nonlinear integral in \eqref{lion1}, we note that since $ u_0 $ is a positive solution of $ ( \mathcal{P}_{V,0} )$, we have $ \partial u_0 / \partial \nu \leq 0 $ on $ \partial \O $. Moreover, by the Hopf Lemma, there exists  $ c > 0 $ such that 
\be \label{u0-nu} \frac{ \partial u_0 }{ \partial \nu } \leq - c < 0 \qquad \mbox{ on } \, \, \partial \O . \ee
Furthermore, because $u_0$ is bounded in the $ C^1$-sense, it follows that
\be \label{u_0-bord} 
u_0 (x) \leq c d (x, \partial \O) \qquad \forall \, \, x \in \O .
 \ee   

It remains to estimate the nonlinear term in \eqref{lion1}. Using Lemma \ref{A1}, we obtain
\begin{align} \label{ms1as} 
\int_{\O} \vert  u \vert^{p-1-\e} \, u \, &\frac{1}{\l} \frac{\partial P\d}{\partial a}
=\int_{\O} W^{p-\e} \, \frac{1}{\l} \frac{\partial P\d}{\partial a} + (p-\e)\int_{\O}W^{p-1-\e} \, {v} \, \frac{1}{\l} \frac{\partial P\d}{\partial a} +O(\| v\|^2) . 
\end{align}
 The linear integral on $ v $ satisfies 
\begin{align} \label{ms2as} 
\int_{\O}W^{p-1-\e} \, {v} \, &\frac{1}{\l} \frac{\partial P\d}{\partial a} = \int_{\O}(\a P\d)^{p-1-\e}\frac{1}{\l} \frac{\partial P\d}{\partial a} v + O\Big(\int_{\O}\Big( \d^{p-2 }u_0+  u_0^{p-1-\e}\Big)\Big|\frac{1}{\l}\frac{\partial P\d}{\partial a}\Big||v|\Big).
\end{align}
However, invoking Lemma \ref{A1}  and \eqref{u_0-bord},  we see that
\begin{align}
&\int_{\O}P\d^{p-2-\e}u_0\big|\frac{1}{\l}\frac{\partial P\d}{\partial a}\Big||v|\leq d\int_B \d^{p-1}|v| + \frac{c}{\l d}\int_{\O\setminus B}\d^{p-1}|v|\leq c(d+\frac{1}{\l d})\| v\| R_{33},\label{as3}\\
&\int_\O u_0^{p-1-\e}\big|\frac{1}{\l}\frac{\partial P\d}{\partial a}\Big||v|\leq c\int_\O \big(\frac{1}{\l d}+ \frac{1}{\l |x-a_i|}\big)\d |v|\leq \frac{\| v\|}{\l d}R_{33} + \frac{\| v\|}{\l^2},\label{as4}
\end{align}
and, invoking \eqref{eq:k20} and Lemma \ref{A1},  we write
\begin{align}
p\int_{\O}P\d^{p-1-\e}\frac{1}{\l} \frac{\partial P\d}{\partial a} v&= p\int_{\O}\d^{p-1-\e}\frac{1}{\l} \frac{\partial \d}{\partial a} v +O\Big(\int_\O \d^{p-1}\theta |v|\Big)\notag\\
&=c_0^{-\e}\l^{-\e\frac{n-2}{2}}p\int_{\O}\d^{p-1}\frac{1}{\l} \frac{\partial \d}{\partial a} v +O\Big(\e \| v\|+\int_B \d^{p-1}\theta |v| + \frac{\| v\|}{(\l d)^{\frac{n+2}{2}}}\Big).\label{as5}
\end{align}
Note that
\be\label{as6}
p\int_{\O}\d^{p-1}\frac{1}{\l} \frac{\partial \d}{\partial a} v=-\int_\O V \frac{1}{\l} \frac{\partial \d}{\partial a} v=O\Big( \int_\O \frac{\d}{\l |x-a|}|v|\Big)=O\big(\frac{\|v\|}{\l^2}\big).
\ee
We also note that, for $n\leq 6$, we have
\be\label{as7}
\int_B \d^{p-1}\theta |v|\leq |\theta|_\infty \int_B \d^{4/(n-2)}|v| \leq c\frac{\| v\|}{(\l d^2))^{(n-2)/2}}R_{33}\leq 
\begin{cases}
\frac{\| v\|}{(\l d)^{n-2}}&\mbox{if} \,n\leq 5\\
\|v\| \frac{\ln^{\frac{2}{3}}(\l d)}{(\l d)^4}&\mbox{if}\, n=6,
\end{cases}
\ee
and, for $n\geq 7$,
\be\label{as8}
\int_B \d^{p-1}\theta |v|\leq \int_B \theta^{p/2}\d^{p/2}|v|\leq |\theta|_\infty^{p/2}\|v\| \Big(\frac{\ln(\l d)}{\l^{n/2}}\Big)^{\frac{n+2}{2n}}\leq \|v\| \frac{(\ln(\l d))^{\frac{n+2}{2n}}}{(\l d)^{(n+2)/2}}.
\ee
We now turn to the first integral on the right-hand side of \eqref{ms1as}. To this aim, observe that \eqref{kkk4} is valid with $ N = 1 $ and $ \l^{-1} \partial P \d / \partial a $ instead of $ \l_i \partial P \d_i / \partial \l_i$. Thus we need to estimate the terms corresponding  to those in \eqref{kkk4}. 
Using Lemma \ref{A1} and \eqref{ba}, we have
\begin{align}
\int_{\O}( \a_0  u_0 ) ^{p-\e} \, \frac{1}{\l} \frac{\partial P\d}{\partial a}
& = \a_0^{p-\e}\big(u_0,\frac{1}{\l_i} \frac{\partial P\d_i}{\partial a_i}\big)+ O\Big(\e \int_\O \big(\frac{1}{\l |x-a|}+ \frac{1}{\l d}\big)\d\Big)\notag\\
&=\a_0^{p-\e}\big(u_0,\frac{1}{\l_i} \frac{\partial P\d_i}{\partial a_i}\big)+ O\Big(\frac{\e}{(\l d)\l^{(n-2)/2}}\Big).\label{as10}
\end{align}
We also have, for $n\geq 4$,
\begin{align}\label{as11}
\int_{\O}  P\d^{p-\e - 1 }  &u_0  \, \frac{1}{\l} \frac{\partial P\d}{\partial a} = \int_{\O}  \d^{p-\e - 1 }  u_0  \, \frac{1}{\l} \frac{\partial \d}{\partial a} + O\Big(\int_{\O}  \d^{p-2}\theta \,  u_0  \, \big|\frac{1}{\l} \frac{\partial P\d}{\partial a}\big| + \int_{\O}  \d^{p- 1 }  u_0  \, \big|\frac{1}{\l} \frac{\partial \theta}{\partial a}\big| \Big). 
\end{align}
However, invoking Lemma $A.1$ of \cite{BCH} and the fact that near the boundary $u_0(x)\leq c d(x,\partial\O)$, we get
\begin{align*}
 \int_{\O}  \d^{p- 1 }  u_0  \, \big|\frac{1}{\l} \frac{\partial \theta}{\partial a}\big|\leq \frac{c}{\l d} \int_{\O}  \d^{p- 1 }  u_0  \,\theta & \leq \frac{c}{\l}|\theta|_\infty^{2/(n-2)}\int_B\d^{n/(n-2)} + \frac{c}{(\l d)^3 \l^{(n-2)/2}}\\
 &\leq \frac{c}{(\l d)^2}\frac{\ln(\l d)}{\l^{n/2}}+ \frac{c}{(\l d)^3 \l^{(n-2)/2}},
\end{align*}
and
\begin{align*}
\int_{\O}  \d^{p-2}\theta \,  u_0  \, \big|\frac{1}{\l} \frac{\partial P\d}{\partial a}\big|&\leq  d|\theta|_\infty^{2/(n-2)}\int_B\frac{\d^{n/(n-2)}}{\l |x-a|}+\frac{c}{\l d}\int_{\O\setminus B}\d^p\\
&\leq \frac{c}{(\l d)\l^{n/2}}+ \frac{c}{(\l d)^3 \l^{(n-2)/2}}.
\end{align*}
Thus, using \eqref{eq:k20}, estimate \eqref{as11} becomes
\begin{align}\label{as12}
\int_{\O}  P\d^{p-\e - 1 }  &u_0  \, \frac{1}{\l} \frac{\partial P\d}{\partial a}
 = c_0^{-\e}\l^{-\e \frac{n-2}{2}}\int_{B}  \d^{p - 1 }  u_0  \, \frac{1}{\l} \frac{\partial \d}{\partial a} + O\Big(\frac{1}{(\l d)\l^{n/2}}+ \frac{1}{(\l d)^3 \l^{\frac{n-2}{2}}}+ \frac{\e d}{\l^{\frac{n-2}{2}}}\Big).
\end{align}
Now, expanding $u_0$ around $a$, we obtain
\begin{align}\label{as12b}
\int_{B}  \d^{p - 1 }  u_0  \, \frac{1}{\l} \frac{\partial \d}{\partial a}&=\n u_0(a)\frac{c_0^p(n-2)}{n\l^{n/2}}\Big[\int_{\R^n}\frac{|x|^2}{(1+|x|^2)^{\frac{n+4}{2}}}+ O\Big(\frac{1}{(\l d)^2}\Big]+ O\Big(\frac{\ln(\l d)}{\l^{(n+4)/2}}\Big)\notag\\
&=\bar{c}_5\frac{\n u_0(a)}{\l^{n/2}}+O\Big(\frac{1}{(\l d)^2\l^{n/2}}+ \frac{\ln(\l d)}{\l^{(n+4)/2}}\Big).
\end{align}
For the remaining term  which corresponds to the case $n \geq 4$, we write
\begin{align}\label{as13}
\int_{\O}   P\d^{ \frac{p-\e - 1 }{ 2 } } u_0 ^{ \frac{p - \e + 1 }{ 2 } } \Big|\frac{1}{\l} \frac{\partial P\d}{\partial a} \Big|&\leq c\int_B d^{\frac{p - \e + 1 }{ 2 } } \frac{\d^{n/(n-2)}}{\l |x-a|}+ \int_{\O\setminus B}\frac{c}{\l d}\d^{\frac{n}{n-2}} \leq c\frac{d^{\frac{p - \e + 1 }{ 2 } }}{\l^{n/2}}+ \frac{\ln\l}{(\l d)\l^{\frac{n}{2}}}.
\end{align}
For the first term in the right hand-side, using Lemma $A.1$ of \cite{BCH} and \eqref{eq:k20},we write
\begin{align*}
 \int_{\O} P\d^{p-\e} \, \frac{1}{\l} \frac{\partial P\d}{\partial a} &  = c_0^{-\e}\l^{-\e\frac{n-2}{2}}\Big[\int_{B} \d^{p} \, \frac{1}{\l} \frac{\partial P\d}{\partial a} -(p-\e)\int_{B} \d^{p-1} \, \frac{1}{\l} \frac{\partial \d}{\partial a}\theta \Big]\notag\\
 &+O\Big(\frac{\e}{(\l d)^{n-2}}\Big)+O\Big(\int_B \d^{p-1}\theta^2\frac{1}{\l |x-a|}\Big)+ O\Big(\frac{1}{(\l d)^{n+1}}\Big).
\end{align*}
However, we have
$$
\int_B \d^{p-1}\theta^2\frac{1}{\l |x-a|}\leq |\theta|_\infty^2\int_B\frac{\d^{p-1}}{\l |x-a|}\leq c\chi_2(\l),
$$
where
$$
\chi_2 (\l)= \big( \frac{1}{ ( \l d )^4}\,(\mbox{if}\, n=4);\quad \frac{\ln(\l d)}{(\l d)^6}\,(\mbox{if}\, n=5);\quad \frac{1}{ ( \l d ) ^{n+1}}\,(\mbox{if}\, n\geq 6)\big). 
$$
Using Lemma $A.10$ in \cite{BCH}. we obtain
\begin{align*}
\int_{B} \d^{p} \, \frac{1}{\l} \frac{\partial P\d}{\partial a}
&=\int_\O \n P\d\,\n \big(\frac{1}{\l} \frac{\partial P\d}{\partial a}\big)+O\Big(\frac{1}{(\l d)^{n+1}}\Big) = - \frac{c_0\bar{c}_2}{2}\frac{1}{\l^{n-1}}\frac{\partial H}{\partial a}(a,a) +O\Big(\frac{1}{(\l d)^{n}}\Big),
\end{align*}
and 
\begin{align*}
-p\int_{B} \d^{p-1} \, \frac{1}{\l} \frac{\partial \d}{\partial a}\theta&=\int_\O \n P\d\,\n \big(\frac{1}{\l} \frac{\partial P\d}{\partial a}\big)+O\Big(\frac{1}{(\l d)^{n+1}}\Big) = - \frac{c_0\bar{c}_2}{2}\frac{1}{\l^{n-1}}\frac{\partial H}{\partial a}(a,a) +O\Big(\frac{1}{(\l d)^{n}}\Big).
\end{align*}
We derive that
\be\label{as18}
\int_{\O} P\d^{p-\e} \, \frac{1}{\l} \frac{\partial P\d}{\partial a}=-\frac{c_0\bar{c}_2}{\l^{n-1}}\frac{\partial H}{\partial a}(a,a) +O\Big(\frac{\e}{(\l d)^{n-2}}+\frac{1}{(\l d)^{n}}\Big).
\ee
Combining the results in \eqref{lion1}, \eqref{k54bisas}- \eqref{as10}, 
\eqref{as12}-\eqref{as18},  and Lemma $A.10$ in \cite{BCH},    we easily derive the desired result.
\end{pf}

\subsubsection{The Case of Multiple Interior Blow-Up Points}

\begin{pro}\label{p9}
Let $ n\geq 7 $, $ (\a,\l,a)\in  \mathbf{B}(N,\mu_0) $ with $ \mu_0 $ positive small and $ d_i := d(a_i, \partial \O ) \geq c > 0$, $ v \in E_{a, \l , u_0} $ and  $ u=\a_0\,u_0+ \sum_{i=1}^{N} \a_i \, P\d_{(a_i,\l_i)}+ {v} $. Then, for $ 1\leq i\leq N $, the following estimate holds:
\begin{align*}
\left\langle I'_\e(u),\frac{1}{\l_i} \frac{\partial P\d_i}{\partial a_i}\right\rangle  = & \a_i \,\rho_n \, \frac{\n V(a_i)}{\l_i^{3}}+ c_0\bar{c}_2 \, \sum_{j\neq i} \a_j \, \frac{1}{\l_i} \frac{\partial \e_{ij}}{\partial a_i} \Big(1- \sum_{k=i,j} c_0^{-\e} \a_k^{p-1-\e} \, \l_k^{{-\e (n-2)} / {2}}\Big)  +  R_{ai} , 
\end{align*}
where $ \bar{c}_2 $ is defined in Proposition \ref{rp43},   $ \rho_n := \frac{n-2}{n} \, c_0^{2} \int_{\mathbb{R}^{n}} \frac{|x|^{2}}{(1+|x|^{2})^{n-1}} \, dx $, and
\begin{align*} 
 R_{ai} = &O\Big(\e^2 +\|v\|^2+\frac{1}{\l_i^4} +  \sum \e_{kr} ^{ \frac{n}{n-2}} \ln ( \e_{kr}^{-1}) + \sum_{j\neq i}\e_{ij}^{\frac{n+1}{n-2}}\l_j|a_i-a_j|+\sum_j \frac{\ln\l_j}{\l_j^{n/2}}\Big) . 
\end{align*}
\end{pro}
\begin{pf}
 We estimate each term on the right-hand side of \eqref{lion1} with $\psi_i=\frac{1}{\l_i} \frac{\partial P\d_i}{\partial a_i}$.
By estimate \eqref{k54bisas}, we have
\begin{align}\label{clm1}
\int_{\O} V \, P\d_i \, \frac{1}{\l_i} \frac{\partial P\d_i}{\partial a_i}
&=  \rho_n \frac{\n V(a_i)}{\l_i^{3}}+\,O\Big(\frac{1}{\l_i^4}\Big).
\end{align}
For the third term appearing on the right-hand side of \eqref{lion1}, we invoke Lemma $6.6$ from \cite{BVP} and Lemma \ref{A1}, which give
\begin{align}\label{k55a}
\Big| \int_{\O} V \, P\d_j \, \frac{1}{\l_i} \frac{\partial P\d_i}{\partial a_i} \Big|  & \leq c\int_\O \frac{\d_j\d_i}{\l_i |x-a_i|}\leq \frac{c}{\l_i}\e_{ij} \leq  c \Big(\e_{ij}^{\frac{n}{n-2}} + \frac{1}{\l_i^{n/2}}\Big).
\end{align}
Using Lemmas $A.12$  in \cite{BCH}, we obtain 
\begin{align}
 \int_{\O} \n P\d_i \n (\frac{1}{\l_i} \frac{\partial P\d_i}{\partial a_i})&=O\left(\frac{1}{\l_i^{n-1}}\right),\label{foia}\\
 \int_{\O} \n P\d_j \n (\frac{1}{\l_i} \frac{\partial P\d_i}{\partial a_i})&=c_0\bar{c}_2\frac{1}{\l_i}\frac{\partial\e_{ij}}{\partial a_i} + O\Big(\frac{1}{(\l_i\l_j)^{\frac{n-2}{2}}}\frac{1}{\l_i }+\sum_{l=i,j}\frac{1}{\l_l^n}\Big) + O\Big( \e_{ij}^{\frac{n+1}{n-2}}\l_j|a_i-a_j|\Big).
 \label{fiia}
\end{align} 
It remains to estimate the nonlinear term in \eqref{lion1}. Using Lemma \ref{A1}, we obtain
\begin{align} \label{ms1a} 
\int_{\O} &\vert  u \vert^{p-1-\e} \, u \, \frac{1}{\l_i} \frac{\partial P\d_i}{\partial a_i}
=\int_{\O} W^{p-\e} \, \frac{1}{\l_i} \frac{\partial P\d_i}{\partial a_i} +O\Big(\e^2 + \| v\|^2 + \e_{ij}^{\frac{n+2}{n-2}}\ln^{\frac{n+2}{n}}\e_{ij}^{-1}\Big).
\end{align}
We now turn to the first integral on the right-hand side of \eqref{ms1a}. We have 
\begin{align}
\int_{\O} W^{p-\e} \, \frac{1}{\l_i} \frac{\partial P\d_i}{\partial a_i}
& =  \int_{\O} \big(\sum_{j=1}^N \a_jP\d_j\big)^{p-\e} \, \frac{1}{\l_i} \frac{\partial P\d_i}{\partial a_i} + ( p - \e ) \int_{\O} \big(\sum_{j=1}^N \a_jP\d_j\big)^{p-\e - 1 } ( \a_0 u_0 ) \, \frac{1}{\l_i} \frac{\partial P\d_i}{\partial a_i}  \notag\\
&+ \int_{\O}( \a_0  u_0 ) ^{p-\e} \, \frac{1}{\l_i}\frac{\partial P\d_i}{\partial a_i}   
 + O \Big( \int_{\O}  \big(\sum_{j=1}^N \a_jP\d_j\big)^{ \frac{p-\e - 1 }{ 2 } } u_0 ^{ \frac{p - \e + 1 }{ 2 } } P\d_i   \Big). \label{kkk4a}
\end{align}
Using Lemma \ref{A1} and \eqref{as10}, we obtain
\begin{align}
&  \int_{\O} (\a_ 0 u_0) ^{p-\e} \, \frac{1}{\l_i} \frac{\partial P\d_i}{\partial a_i} = \a_0^{p-\e}\big(u_0,\frac{1}{\l_i} \frac{\partial P\d_i}{\partial a_i}\big)+
 O \Big( \frac{ \e} { \l_i ^{ n/2 }} \Big) = O \Big( \frac{1}{ \l ^{n/2}} \Big)  . \label{11kka} \end{align}
Concerning the first integral in the right hand-side of \eqref{kkk4a}, we notice that \eqref{rm14} holds true with $ \l_i ^{-1} \partial P \d_i / \partial a_i $ instead of $ \l_i  \partial P \d_i / \partial \l_i $. Thus we need to estimate the term corresponding to those in \eqref{rm14}.  For the first integral, Using Lemma $B.2$ in \cite{BCH} and \eqref{as18}, we obtain
\begin{align}\label{k43a}
\int_{\O}P \d_i^{p-\e} \, \frac{1}{\l_i} \frac{\partial P\d_i}{\partial a_i} & = O\Big(\frac{1}{\l_i^{n-1}}+\frac{\e}{\l_i^{n-2}}\Big).
\end{align}
For the second integral, we use Lemma \ref{A1} together with \eqref{eq:k20} to write
 \begin{align}\label{bk0a}
 p\int_{\O}P\d_i^{p-1-\e}\frac{1}{\l_i} \frac{\partial P\d_i}{\partial a_i} P\d_j &= p\int_{\O}\d_i^{p-1-\e}\frac{1}{\l_i} \frac{\partial \d_i}{\partial a_i} P\d_j +O\Big(\int_\O \d_i^{p-1}\d_j\theta_i\Big).
\end{align}
Note that, by estimate $F1$ in \cite{B}, and Lemma \ref{A1}, we obtain 
\begin{align*}
\int_\O \d_i^{p-1}\d_j\theta_i&\leq |\theta_i|_{\infty}^{p-1}\Big(\int_\O(\d_i\d_j)^{\frac{n}{n-2}} \Big)^{(n-2)/n} \leq c\frac{\e_{ij}}{\l_i^2}\ln^{(n-2)/n}(\e_{ij}^{-1})=O\Big(\e_{ij}^{\frac{n}{n-2} } \ln(\e_{ij}^{-1})+ \frac{1}{\l_i^n}\Big).
\end{align*}
However, using \eqref{eq:k20}, we find that
\begin{align*}
p\int_{\O}&\d_i^{p-1-\e}\frac{1}{\l_i} \frac{\partial \d_i}{\partial a_i} P\d_j=c_0^{-\e}\l_i^{-\frac{\e(n-2)}{2}}p\int_{\O}\d_i^{p-1}\frac{1}{\l_i} \frac{\partial \d_i}{\partial a_i} P\d_j +O\Big(\e \int_\O \d_i^p\ln (1+\l_i^2|x-a_i|^2)\d_j\Big)\\
&=c_0^{-\e}\l_i^{-\frac{\e(n-2)}{2}}\int_{\O}\n \big(\frac{1}{\l_i} \frac{\partial P\d_i}{\partial a_i}\big) \n P\d_j+O\Big(\e\int_\O \d_i^{2/(n-2)}\ln (1+\l_i^2|x-a_i|^2)\d_i^{n/(n-2)}\d_j\Big).
\end{align*}
Thus, using \eqref{fiia}, \eqref{bk0a} becomes
 \begin{align}\label{M6a}
 p\int_{\O}P\d_i^{p-1-\e}\frac{1}{\l_i} \frac{\partial P\d_i}{\partial a_i} P\d_j &=c_0^{-\e} \l_i^{ \frac{- \e (n-2)}{ 2}} c_0\bar{c}_2\frac{1}{\l_i}\frac{\partial\e_{ij}}{\partial a_i} + O\Big(\frac{1}{(\l_i\l_j)^{\frac{n-2}{2}}}\frac{1}{\l_i }+\sum_{l=i,j}\frac{1}{\l_l^n}\Big)\notag\\
 &\quad+O\Big( \e_{ij}^{\frac{n+1}{n-2}}\l_j|a_i-a_j|+\e\e_{ij} + \e_{ij}^{\frac{n}{n-2}}\ln(\e_{ij}^{-1})\Big).
 \end{align}
  The third integral is treated in the same way as in the proof of \eqref{M6a}, yielding the same result as in the estimate of \eqref{M6a} but with $\l_i^{ - \e (n-2) / { 2}}$ replaced by $\l_j^{ - \e (n-2) / { 2}}$. Taking this into account, together with \eqref{k43a}  and \eqref{M6a},  the estimate of equation corresponding to \eqref{rm14} follows. 

Regarding the second integral in the right hand-side of \eqref{kkk4a}, using \eqref{mo1}, \eqref{as12} and \eqref{as12b}, we observe that
 \begin{align}\label{M7a}
  \int_{\O}P\d_i^{p-1-\e}&\frac{1}{\l_i} \frac{\partial P\d_i}{\partial a_i}u_0 =O\Big(\frac{\e}{\l_i^{ (n-2) / {2}}}+ \frac{1}{\l_i^{n/2}}\Big).
 \end{align}
Combining the results in \eqref{lion1}, \eqref{clm1}-\eqref{12kk}, and \eqref{11kka}-\eqref{M7a},    we derive the desired result.
\end{pf}

\section{Building Interior Blowing-Up Solutions with Residual Mass} 

The purpose of this section is to construct interior blow-up solutions of $(\mathcal{P}_{V,\e})$ possessing residual mass, which may manifest as either isolated or clustered bubbles. We begin with the case of isolated bubbles.
\subsection{Construction for Interior Isolated Bubbles and Residual Mass}
This section is dedicated to proving Theorem \ref{t4}. Our aim is to construct solutions of $(\mathcal{P}_{V,\e})$ that blow-up at $N$ interior points with residual mass as $\e \to 0$, where $1 \leq N \leq r$ and $r$ is the number of non‑degenerate critical points of $V$. We take $n \geq 7$, let $b_1,\dots,b_N$ be distinct non‑degenerate critical points of $V$, and let $u_0$ be a non‑degenerate solution of the limit problem $(\mathcal{P}_{V,0})$.
 As in the previous section, the construction of the desired solution follows  closely the approach developed in \cite{BLR}, \cite{EM}, \cite{Sym}, \cite{Sym25} and \cite{KK}. To this end, we set 
\begin{align}\label{5.1}
\mathbb{A}(N,u_0, \e) =\lbrace & (\alpha, \lambda, a, v) \in (\mathbb{R}_+)^{N+1} \times (\mathbb{R}_+)^N \times \Omega^N \times H^1_0(\Omega): \vert\alpha_i-1\vert < \varepsilon \,\ln^2 \varepsilon \quad \forall 0\leq i\leq N,\nonumber\\ 
& \frac{1}{c} < \lambda^2_i \, \e < c, \vert a_i - b_i \vert < \varepsilon^{1/5} \quad\forall 1\leq i \leq N, v \in E_{a,\lambda,u_0}\, \mbox{and} \,\| v \| <  \sqrt{\varepsilon} \rbrace, 
\end{align}
where $E_{a,\lambda,u_0}$ is introduced in \eqref{KA6} and $c$ is a fixed positive constant.\\
For $(\alpha, \lambda, a, v)\in \mathbb{A}(N,u_0, \e)$, we see that  $ u=\a_0 \, u_0 + \sum_{i=1}^N \a_i \, P\d_{(a_i, \l_i) }+ {v} $ is a critical point of $ I_\e $ if and only if  there exists $(A,B,C,D) \in \mathbb{R}^N \times \mathbb{R}^N \times (\mathbb{R}^n)^N \times \mathbb{R}$ such that
 \begin{align}
 &  \langle I^{\prime}_\e(u),u_0  \rangle = 0, \label{5.11}  \\
 & \langle I^{\prime}_\e(u),P\d_i \rangle =0 \quad \forall\,  i\in \{1,\cdots,N\} ,  \label{5.12} \\
 &  \langle I^{\prime}_\e(u),\a_i \, \frac{\partial P\d_i}{\partial \l_i}  \rangle  =B_i( {v} ,\l_i \, \frac{\partial^2P\d_i}{\partial \l_i^{2}})+\sum_{j=1}^n C_{ij} ( {v} ,\frac{1}{\l_i}\frac{\partial^2 P\d_i}{\partial \l_i \, \partial a_{ij}}) \quad \forall i \in \{1,\cdots,N\} , \label{5.13} \\
 &  \langle I^{\prime}_\e(u),\a_i \, \frac{\partial P\d_i}{\partial a_i}  \rangle =B_i( {v} ,\l_i \, \frac{\partial^2P\d_i}{\partial a_i\partial \l_i})+\sum_{j=1}^n C_{ij} ( {v} ,\frac{1}{\l_i}\frac{\partial^2 P\d_i}{\partial a_i \, \partial (a_i)_j})\quad \forall i \in \{1,\cdots,N\} \label{5.14}\\
 &I^{\prime}_\e(u)=\sum_{j=1}^NA_jP\d_j + \sum_{j=1}^N B_j\l_j\frac{\partial P\d_j}{\partial\l_j}+\sum_{i=1}^N\sum_{j=1}^n C_{ij}\frac{1}{\l_i}\frac{\partial P\d_i}{\partial (a_i)_j} +Du_0.\label{5.pv}
 \end{align} 
 It is clear that the element $\bar{v}_\e$ obtained in Proposition \ref{p2} satisfies \eqref{5.pv}. Thus, 
 $ u=\a_0 \, u_0 + \sum_{i=1}^N \a_i \, P\d_i+\bar{v}_\e$ is a critical point of $ I_\e $ if and only if $ (\a,\l,a) $ solves the  system \eqref{5.11}-\eqref{5.14}. 
 Note that, because $a_i$ lies near the critical point $b_i$ of $V$ and $\frac{1}{c} < \lambda_i^2 \varepsilon < c$, it follows that
\begin{align}\label{5.15}
|a_i-a_j|\geq c >0 \quad  \mbox{and} \quad \e_{ij}=O(\e^{ (n-2) / {2}}) . 
\end{align} 
Applying \eqref{5.15}, the remainder terms in  Propositions \ref{p2}, \ref{rp41}, \ref{rp42}, \ref{rp43}, and \ref{p9} satisfy
 \begin{align}\label{5.16}
\Vert \bar{v}_\e\Vert ; \, \,  R_{\a 0}; \, \,  R_{\a i} = O(\e), \qquad R_\l(i); \, \, R_{ai} \leq c (  \e^2 + \e^{{n} / {4}} |\ln \e| ) . 
\end{align} 
Taking $ (\a,\l,a,0)\in \mathbb{A}(N,u_0, \e)$ and  following closely the proof of  Lemma 3 of \cite{EM}, we see that, for  $\e$ small, 
the constants $ A_i$'s, $B_i$'s, $C_{ij}$'s and $D$ which appear in equations \eqref{5.11}-\eqref{5.pv} satisfy
\be\label{lem:5.1}
 A_i=O(\e \, \ln^2 \, \e), B_i=O(\e),\, D=O(\e \, \ln^2 \, \e), \quad C_{ij}=O(\e^{\frac{3}{2}})\quad \forall 1\leq i\leq N \quad \forall 1\leq j\leq n . 
\ee
Now, taking the following change of variables.
$$\beta_i=1-\a_i^{p-1}, 0\leq i\leq N; \, \quad \frac{1}{\l_i^2}=\frac{\bar{c}_3}{\gamma_nV(b_i)} \e (1+ \L_i ),1\leq i\leq N; \, \quad \zeta _i := a_i-b_i\, \, \, 1\leq i\leq N,$$ 
where $ \bar{c}_3 $ and $ \gamma_n $ are defined in Proposition \ref{rp43}, and using  Propositions \ref{rp41} and \ref{rp42} and estimate \eqref{5.16}, we get $ \beta_0 = O(\e) $ and $\beta_i=O(\e|\ln \e|) $ for $ 1\leq i\leq N $. \\
In addition, Propositions \ref{rp43} and \ref{p9} can be written as 
\begin{align*}
& \Big\langle   I^{\prime}_\e(u),\l_i \frac{\partial P\d_i}{\partial \l_i}\Big\rangle =  \a_i \bar{c}_3 \, \e   -\a_i \, V(a_i)\,  \frac{  \gamma_n }{\l_i^{2}} + O ( \e ^2 + \e ^{ (n-2)/4 } ) , \\
 & \left\langle I'_\e(u),\frac{1}{\l_i} \frac{\partial P\d_i}{\partial a_i}\right\rangle  =  \a_i \,\rho_n \, \frac{\n V(a_i)}{\l_i^{3}} + O ( \e ^2 + \e ^{ n /4 }  | \ln \e | ) . 
 \end{align*}
Now, writing 
$$
V(a_i)=V(b_i)+O(|b_i-a_i|^2),\quad \n V (a_i)=D^2 \, V(b_i)\, (\zeta_i,\cdot)+O(|\zeta_i|^2), 
$$
we see that the system \eqref{5.11}-\eqref{5.14} is equivalent to the following system:
 \begin{align*}
 \begin{cases}
 \beta_0=O(\e)\\
 \beta_i=O(\e|\ln \e|) & 1\leq i\leq N \\
 \L_i=O(|\zeta_i|^2 +  \e  + \e^{ (n-6 )/ 4} ) & 1\leq i\leq N\\
 D^2\, V (b_i)\, (\zeta_{i} , \cdot ) = O( |\zeta_i|^2 +  \e^{1/2} + \e^{ (n-6) / 4 } |\ln \e|) & 1\leq i\leq N.
 \end{cases}
 \end{align*}
Thus,  Brouwer's fixed point Theorem gives the existence of a solution $ ( \b^\e, \L^\e , \zeta^\e)$  of the previous system which implies the existence of a critical point $ u_\e $ of $ I_\e $. In addition, it is clear that $ \zeta^\e $ satisfies: $ | \zeta^\e | \leq c (  \e^{1/2} + \e^{ (n-6) / 4 } |\ln \e|) $. Following \cite{EM},  $ u_\e > 0 $ and therefore it is a solution of $ ( \mathcal{P}_{V, \e} ) $ and satisfies the desired estimates cited in the theorem. 

Notice that, for each collection of non-degenerate critical points $ (b_{i_1}, \cdots , b_{i_\ell } )$, we have constructed a solution of $ ( \mathcal{P}_{ V, \e} )$ blowing up at these points and converges weakly to $ u_0 $. This implies that $ ( \mathcal{P}_{ V, \e} )$ has at least $ 2^r -1 $ interior blowing up solutions which converge weakly to $ u_0 $. The proof of the theorem is thereby completed.

\subsection{Construction for Interior Clustered Bubbles and Residual Mass}
In this subsection we prove Theorems \ref{t5} and \ref{t6} by constructing interior blowing-up solutions with clustered bubbles at a critical point of $V$. Assuming $ n\geq7 $, we first prove Theorem \ref{t5}. To this aim, let $ b \in \O $ be a critical point of $V$ and let $ (\ov{y}_1,\cdots,\ov{y}_N) $ be a non-degenerate critical point of the function $ F_{N,b} $ defined in \eqref{F}. Following the proof of Theorem \ref{t5}, we introduce the following set:
\begin{align}\label{s1}
\mathcal{u}(N , b ,\e , u_0)  = & \lbrace  (\alpha, \lambda, a, v) \in (\mathbb{R}_+)^{N + 1} \times (\mathbb{R}_+)^N  \times \Omega^N  \times H^1_0(\Omega) :|\a_i-1| < \, \e \ln^2 \, \e , \\  
& \frac{1}{c} < \l_i^2 \, \e <c, |a_i- b -\e^\eta \, \sigma \,\ov{y}_i|<\e \, \, \, \forall 1\leq i\leq N , v \in E_{a,\l,u_0} \, \,  \mbox{and} \, \, \,  \Vert v \Vert<  \sqrt{\e} \rbrace,   \notag
\end{align}
where   $E_{a,\lambda,u_0}$ is defined by \eqref{KA6}, 
\be\label{s2}
\eta=\frac{n-4}{2n}\quad\mbox{and}\quad \sigma = \left(\frac{c_0\bar{c}_2}{\rho_n}\right)^{1/n} \left(\frac{\ov{c}_3}{\gamma_nV(b)}\right)^\eta
\ee
with $\bar{c}_2$, $\ov{c}_3$, $\rho_n$ and $\g_n$ are the constants defined in Propositions \ref{rp43} and \ref{pas-geq5}. 

Following the proof of Theorem \ref{t4}, we observe that system  \eqref{5.11}-\eqref{5.pv} are satisfied. To solve this system, we take the following change of variables. 
\begin{align}\label{s3}
\b_i= & 1 - \a_i^{p-1}, 0\leq i \leq N ; \quad  \frac{1}{\l_i^2}  = \frac{\ov{c}_3}{\gamma_n\, V( b )} \,  \e \,(1+\Gamma_i) \, ;  \, \quad  a_i - b =\e^\eta \sigma ( \ov{y}_i + \tau_i ) , \,  1\leq i\leq N ,
\end{align}
We observe that
\begin{align}
 & \e_{ij}= \frac{1}{(\l_i\, \l_j \, |a_i-a_j|^2)^{\frac{n-2}{2}}} \, (1+O(\e^{1-2 \eta}))= O(\e^{2\eta+1}) ,  \label{s4} \\ 
 & \frac{\partial \e_{ij}}{\partial a_i} = \frac{(n-2)\e^{\eta +1} \, \ov{c}_3^{\frac{n-2}{2}}} {\sigma^{n-1}(\g_n \, V(y))^{\frac{n-2}{2}}} \, \frac{( \ov{y}_j + \tau_j -\ov{y}_i - \tau_i )}{| \ov{y}_i  + \tau_i - \ov{y}_j  - \tau_j |^n} (1+f(\Gamma_i,\Gamma_j))(1+O(\e^{1-2 \eta})) = O(\e^{\eta +1}), \label{s5} 
\end{align}
where 
\begin{align}\label{s6}
f(\Gamma_i , \Gamma_j)=\frac{n-2}{4} \,\Gamma_i+\frac{n-2}{4} \,\Gamma_j +O(\Gamma_i^2)+O(\Gamma_j^2) . 
\end{align}
With \eqref{s3} and \eqref{s4}, the remainder terms appearing in Propositions \ref{p2}, \ref{rp41}, \ref{rp42}, \ref{rp43}, and \ref{p9} can be estimated as 
\be\label{s7}
\Vert v_\e\Vert =O(\e); \quad  R_{\a_0} = O(\e); \quad R_{\a i} = O(\e) ;  \quad R_\l (i), \quad R_{ai}  \leq c \e^{ \min( 2 , n/4 ) } |\ln \e| . \ee
Following the proof of Lemma 3 of \cite{EM}, we obtain that the constants $ A_i^{\prime} s $,$ B_i^{\prime} s $,$ c_{ij}^{\prime} s $ and $ D^{\prime} s $ satisfy 
\be\label{s8}
D=O(\e \, \ln^2 \, \e), A_i=O(\e \, \ln^2 \, \e), B_i=O(\e) \, \, \, \mbox{and} \, \, \, C_{ij} = O(\e^{3/2+\eta}) \quad \forall \, \,  i\leq  N , \, \forall \, \,  j\leq n . 
\ee
Next, we express equations \eqref{5.11}–\eqref{5.14} in terms of the change of variables introduced in \eqref{s3}.\\
First, from estimate \eqref{s7}, Propositions \ref{rp41} and \ref{rp42} we obtain the following :
\be\label{s9} 
\b_0 = O(\e); \qquad  \b_i=O(\e \, |\ln \e|) \quad 1\leq i\leq  N .
\ee
Second, from estimate \eqref{s4}, \eqref{s7}, \eqref{s8}, \eqref{s9} and Proposition \ref{rp43} we find that
$$
\ov{c}_3 \, \e  - \frac{ \g_n \,V(a_i)}{\l_i^2} = O \big( \e^{ 2 \eta + 1 }  +  \e^{( n-2) / 4}   \big) .
$$
Writing 
\begin{align*}
V(a_i) & = V( b ) + O( | a_i - b |^2 )  = V( b ) + O(\e^{2 \eta} | \ov{y}_i + \tau_i |^2 ) , 
\end{align*}
we observe that 
\be\label{s10} 
\Gamma_i =  O \big( \e^{2\eta} +  \e^{( n - 6 ) / 4 }    \big) \quad  \forall \, \, 1\leq i\leq  N . 
\ee
Lastly, applying Proposition \ref{p9} together with the estimates \eqref{s5}, \eqref{s7}, \eqref{s8}, and \eqref{s9}, we obtain
\begin{align*}
 \rho_n  \, \frac{\n V(a_i)}{\l_i^3} - & \frac{c_0 \ov{c}_2 (n-2) }{\sigma^{n-1}}   \Big(\frac{\ov{c}_3}{\g_n \,V( b )}\Big)^{\frac{n-2}{2}} \frac{\e^{\eta+1}}{\l_i} \sum_{j\neq i} \frac{( \ov{y}_j + \tau_j  - \ov{y}_i - \tau_i )}{ | \ov{y}_j + \tau_j - \ov{y}_i -  \tau_i |^n} \times \\
  & \times \Big(1+\frac{(n-2)}{4}\Gamma_i + \frac{(n-2)}{4}\Gamma_j  \Big) = O \Big( | \Gamma |^2 \e^{ \g + 3/2} \Big) +   O \big(  \e^ { \min( 2 , n/4 ) } |\ln \e| \big) . 
\end{align*} 
Writing 
\begin{align*} 
\n V(a_i)& =D^2 \, V( b )(\e^\eta \, \sigma ( \ov{y}_i + \tau_i ),\cdot) + O( \e^{2\eta} | \ov{y}_i + \tau_i |^2),
\end{align*}
we find that 
\begin{align*}
D^2 \, V( b ) ( \ov{y}_i + \tau_i , \cdot ) - & (n-2) \sum_{j\neq i} \frac{( \ov{y}_j + \tau_j - \ov{y}_i - \tau_i ) }{ | \ov{y}_j + \tau_j - \ov{y}_i - \tau_i |^n}(1+\frac{(n-6)}{4}\Gamma_i \, + \frac{(n-2)}{4}\Gamma_j )\\
& =O \big(\e^\eta | \ov{y}_ i + \tau_i |^2  + | \Gamma | ^ 2 + \e^{ \min( 1/2 , (n-6) /4 ) - \eta} | \ln \e | \big) \\
& = O\big(  | \Gamma | ^ 2 +  \e^{ \min( 1/2 , (n-6) /4 ) - \eta} | \ln \e | \big).
\end{align*}
Observe that 
\begin{align*}
\frac{ \ov{y}_j + \tau_j - \ov{y}_i - \tau_i }{ | \ov{y}_j + \tau_j - \ov{y}_i - \tau_i |^n } = \frac{ \ov{y}_j - \ov{y}_i }{ | \ov{y}_j - \ov{y}_i |^n} -n \langle \frac{ \ov{y}_j - \ov{y}_i }{ | \ov{y}_j - \ov{y}_i |^2} , \tau_j - \tau_i \rangle \frac{( \ov{y}_j - \ov{y}_i )}{ | \ov{y}_j - \ov{y}_i |^n} + \frac{\tau_j - \tau_i }{ | \ov{y}_j - \ov{y}_i |^n} + O( | \tau |^ 2 ) .
\end{align*}
We deduce that 
\begin{align}
& D^2  \,V( b )\, (\tau_i,\cdot) - (n-2) \sum_{j\neq i} \frac{\tau_j-\tau_i}{| \ov{y}_j - \ov{y}_i |^n} -(n-2)\sum_{j\neq i} \frac{ \ov{y}_j - \ov{y}_i }{ | \ov{y}_j - \ov{y}_i |^n}\Big(\frac{(n-6)}{4} \Gamma_i+\frac{(n-2)}{4}\Gamma_j\Big) \notag \\
& +n(n-2)\sum_{j\neq i} \langle \frac{ \ov{y}_j -  \ov{y}_i }{ | \ov{y}_j - \ov{y}_i |^2}, \tau_j - \tau_i \rangle \frac{( \ov{y}_j - \ov{y}_i )}{ | \ov{y}_j - \ov{y}_i |^n} = O \big(  | \tau |^2 + |\Gamma |^2 + \e^{ \min( 1/2 , (n-6) /4 ) - \eta} | \ln \e | \big) .\label{s11}
\end{align}
Consequently, for sufficiently small $\e$, the system \eqref{5.11}–\eqref{5.14} is equivalent to that defined by \eqref{s9}–\eqref{s11}. In addition, it is straightforward to verify that equations \eqref{s11} with $ 1\leq i\leq N $ are equivalent to
$$
 \frac{1}{2} D^2  \, F_{N,b} ( \ov{y}_1,\cdots, \ov{y}_N)(\tau_1,\cdots,\tau_N )  - (n-2) \wedge = O\big( |\tau |^2 +  |\Gamma |^2 + \e^{ \min( 1/2 , (n-6) /4 ) - \eta} | \ln \e | \big) , $$
 where $$ \wedge := ( \wedge_1, \cdots, \wedge_k ) \quad \mbox{ with } \, \,   \wedge _i = \sum_{j \neq i} \frac{(b_j-b_i)}{|b_j-b_i|^n} \Big(\frac{(n-6)}{4} \Gamma_i +\frac{(n-2)}{4}\Gamma_j \Big) \, \, \mbox{for} \, \, 1\leq i\leq N . $$
 Using the fact that $ ( \ov{y}_1,\cdots, \ov{y}_N) $ is a non-degenerate critical point of $ F_{N,b} $, we derive that the system \eqref{s4}-\eqref{s11} has a solution $ (\b_\e,\Gamma_\e,\tau_\e) $ for $ \e $ small. This proves that $ (\mathcal{P}_{V,\e}) $ possess a solution $ u_{\e,b,u_0}=  \a_{0, \e}u_0 + \sum_{i=1}^N \a_{i,\e} \,P \d_{(a_{i,\e},\l_{i,\e})} + v_\e $ satisfying the required properties.  The proof of Theorem \ref{t5} is now complete..

To prove Theorem \ref{t6}, let $ b_1,b_2 $ be two non-degenerate critical points of $ V $ and $ a_{1,\e},a_{2,\e} \in \O $ satisfying $ a_{1,\e}\to b_1 $ and $ a_{2,\e}\to b_2$ as $\e\to 0$. This implies that 
\begin{align*}
(P\d_{(a_{1,\e},\l_{1,\e})}, P\d_{(a_{2,\e},\l_{2,\e})})=O(\e_{ij})=O\Big( \frac{1}{(\l_{1,\e} \l_{2,\e})^{(n-2) / {2}}} \Big) = O(\e^{ (n-2) / {2}}) . 
\end{align*}
Hence, for $ n\geq 7 $,  the interaction between bubbles belonging to distinct blocks is negligible and contributes only to the remainder terms in equations \eqref{s9}–\eqref{s11}. This allows us to handle each block independently. Following the proof of Theorem \ref{t5}, we then obtain the desired results.

\section{The Incompatibility of Blow-Up and Residual Mass in Low Dimensions}
This section is dedicated to proving the impossibility of  coexistence between blowing-up points and residual mass in low dimensions. More precisely, our aim is to prove Theorems \ref{t0} and \ref{t1}.
The following lemma serves as our starting point.
\begin{lem}\label{L41}
Let $n \geq  3 $, $\Omega$ be any smooth bounded domain in $\R^n$, and $u_\e$ be a solution of $(\mathcal{P}_{V,\e})$ of the form  \eqref{KA4} satisfying \eqref{KA5}. Then, for all  $i \in\{1, \ldots, N\}$, the following  holds:
$$
\e \,\ln \l_{i,\e} \longrightarrow 0  \quad\mbox{ as}\quad \e \longrightarrow 0.
$$
\end{lem}
\begin{pf}
Let $ u_\e = \a_{0,\e} \, u_0+\sum_{j=1}^{N} \a_{j,\e} \, P\d_{(a_{j,\e},\l_{j,\e})} +v_\e $ satisfying \eqref{KA5}. In the sequel of this section, we simplify the exposition by omitting the $ \e $ subscript on the parameters.  Multiplying $(\mathcal{P}_{V,\e})$ by $ P\d_i := P\d_{(a_{i },\l_{ i })}$,  integrating on $\O$, and using \eqref{r41}, \eqref{k50}, \eqref{k51} and the fact that $ v_\e \in E_{a,\lambda,u_0}$, we obtain  
\be \label{blue1} 
\a_i S_n + o(1)  = \int_\O u_\e ^{p-\e} P\d_i . \ee
However, we have 
\begin{align}
\int_\O u_\e ^{p-\e} P\d_i  = & \, \a_i^{p-\e} \int_{\O} P\d_i^{p+1-\e}+O\Big(\sum_{j \neq i}\int_{\O} \Big(P\d_i^{p-\e} \, P\d_j+P\d_j^{p-\e} \, P\d_i\Big)\Big) \notag \\
&\quad +O\Big(\int_{\O} \Big(P\d_i^{p-\e} \, u_0 +P\d_i^{p-\e} \, |v_\e| + u_0^{p-\e} P\d_i+|v_\e|^{p-\e} \, P\d_i \Big)\Big) \notag  \\
& =  \, \a_i^{p-\e} \int_{\O} P\d_i^{p+1-\e} + o(1).   \label{BN1}
\end{align}
Moreover, as in $(2.9)$ in \cite{BEGR}, we have
\be\label{blue9}
\int_{\O} P\d_i^{p+1-\e} =\l_i^{\frac{-\e(n-2)}{2}} (S_n+O(\e))+o(1) . 
\ee
Combining \eqref{blue1}-\eqref{blue9}, the lemma follows.  
\end{pf} 

Next, we are going to prove Theorem \ref{t0}.\\
\begin{pfn} {\bf of Theorem \ref{t0}}
Proceeding by contradiction, we assume the existence of a solution $ u_\e $ to $(\mathcal{P}_{V,\e})$ that satisfies the properties introduced in \eqref{KA5}. By Lemma \ref{L41}, we obtain that Propositions \ref{p2}, \ref{rp41}, \ref{rp42}, \ref{rp43}, and \ref{p9} are valid, and the left-hand sides in Propositions \ref{rp41}–\ref{p9} equal zero. There is no loss of generality in assuming $ \l_1\leq\l_2\leq \cdots \leq \l_k $.
Thus, using  Proposition \ref{rp43}, we get
\be\label{N9}
 - c_0\bar{c}_2 \sum_{j \neq i}  \, \l_i \, \frac{\partial \e_{ij}}{\partial \l_i} + \, \bar{c}_3 \, \e +  \, \frac{n-2}{2}\,\frac{\bar{c}_2}{\l_i^{ (n-2) / {2}}} u_0(a_i) 
=o\Big( \e + \sum \frac{1}{\l_j^{ (n-2) / {2}}} +  \sum_{j \neq k}\e_{jk}\Big).
\ee
However, By Lemma $ 6.5 $ of \cite{BVP}, we have 
\begin{align}\label{N11}
c \, \e_{ij}  \leq -\l_i \, \frac{\partial \e_{ij}}{\partial \l_i} -\eta \, \l_j \, \frac{\partial \e_{ij}}{\partial \l_j} \quad if \, \, \l_i\leqslant \l_j \quad \mbox{and}\quad \eta > 3/2 . 
\end{align}
Multiplying \eqref{N9} by $2^{i-1}$ and summing for $i = 1, \dots, N$ gives
\begin{align}\label{N12}
\sum_{j \neq r} \e_{jr}+\sum_{j=1}^k \frac{1}{\l_j^{ (n-2) / {2}}} + \e = o\Big(\sum_{j\neq r} \e_{jr}+\sum_{j=1}^k \frac{1}{\l_j^{ (n-2) / {2}}}+\e\Big),
\end{align}
where we made use of \eqref{N11}.\\
This leads to a contradiction, and the theorem is therefore proved.
\end{pfn}

The remainder of this section is devoted to prove Theorem \ref{t1}. We then assume that the dimension $ n \in \{4,5\}$.  Let $ u_\e $ be a solution of Problem $(\mathcal{P}_{V,\e}) $ satisfying \eqref{dec} and \eqref{cdec} introduced in Theorem \ref{t1}. Notice that we can be more precise: $ u_\e $ can be written in the form \eqref{KA4}, and \eqref{KA5} is then satisfied (with $ N = 1 $). First, using Lemma \ref{L41}, we know that $ \lim \e \ln \l = 0 $. Furthermore, Propositions \ref{p2}, \ref{rp41},  \ref{rp42}, \ref{rp43} and  \ref{pas-geq5} are valid. 
In addition, since the left hand side of these propositions is zero, we obtain
\begin{align}
& |  1 - \a_0^{p-1-\e} | \leq  c \Big( \e +\Vert {v_\e} \Vert^{2}+  \frac{1}{\l ^{ (n-2) / {2}}}  \Big)  \label{a-fin-4} ,  \\
& | 1 - \a ^{p-1-\e} c_0^{-\e}\l ^{\frac{ - \e(n-2) }{2}} | \leq  c \Big(  \Vert v_\e \Vert^{2} +\e +\frac{1}{(\l d )^{n-2}} \Big)+ c \Big( \frac{\ln \l }{ \l } \,  \mbox{ if } \, n=4 \, ; \,   \frac{1}{\l ^{ {3} / {2}}}  \,  \mbox{ if } \, n=5 \Big), \label{a-fin-5}\\
& \| v_\e \|^2 \leq   c \Big(  \e^2 + \frac{1}{\lambda ^{ n-2} } + \frac{1}{(\l d )^{2n-4}}  \Big) .  \label{est-v-eps-01}
  \end{align} 
We remark that, in Proposition \ref{pas-geq5}, if the point $ a $ is close to the boundary, then the principal terms are of the order of $ \frac{ c }{\l ^{n/2} }$ and $ \frac{ c }{ ( \l d )^{n-1} }$. Thus, for  dimension $ n = 5 $, the estimate of $ \| v _\e \|^2 $ (which appears in the proposition) is small compared to the principal terms. However, when $ n = 4 $, the term $ \frac{ c }{ \l ^2 } $ arises as a principal term and also appears in the estimate of  $ \| v_\e \|^2 $. Consequently, for $ n = 4 $, we need to resolve this technical difficulty by refining the estimate of some integrals that involve  $ v_\e $.  \\
Now, we will focus on proving the theorem for  dimension $ n = 5 $. 

\begin{pfn}{\bf of Theorem \ref{t1} in the case $ n = 5 $.}
Putting \eqref{a-fin-4}, \eqref{a-fin-5} and \eqref{est-v-eps-01} in Propositions \ref{rp43} and  \ref{pas-geq5}, we derive the following balancing conditions: 
\begin{align}
&  - \a c_0 \bar{c}_2\frac{3}{2}\frac{H(a ,a )}{\l ^{3}}  -\a  \, V(a ) \frac{ \gamma_5 }{\l ^{2}} + \bar{c}_3 \a \, \e    +   \a_0 \bar{c}_2 \frac{ 3 }{2}\frac{ u_0(a ) }{\l ^{ \frac{3 } {2}}}    = O\Big(\e^2 + \frac{1}{ \l ^3 } \Big)  +  o \Big(  \frac{d }{\l ^{ {3} / {2}}}+ \frac{1}{(\l d )^{3}} \Big) \label{fin-01} \\ 
&  \frac{\a}{2}\,\frac{c_0\,\bar{c}_2}{\l^{4}}\frac{\partial H}{\partial a}(a,a)  -\a_0\,\bar{c_5} p \frac{\n u_0(a)}{\l^{5/2}} 
 = o\Big(\frac{1}{(\l d)^{4}} + \frac{1}{\l^{5/2}} \Big)+ O\Big( \e^2 +  \frac{d^{p+1-\e}}{\l^{ 5 /2}} \Big) . \label{fin-02}
\end{align}
According to Theorem \ref{t0}, it suffices to prove that  $ a $ is far away from the boundary.
Arguing by contradiction, assume that  $ d:= d(a, \partial \O) \to 0$. In this case, from \cite{R-G}, we know that 
\be \label{H-bord} H(a,a) = \frac{ 1 }{ (2d)^{n-2}} + O \Big( \frac{1}{ d ^{n-3} } \Big) \qquad \mbox{ and } \qquad \frac{ \partial H}{ \partial \nu } (a,a) = \frac{ n-2 }{ (2d)^{n-1}} + O \Big( \frac{1}{ d ^{n-2} } \Big).  \ee 
Now, multiply \eqref{fin-02} by $ \nu $ and using \eqref{H-bord} - \eqref{u0-nu},  we derive that 
$$  \frac{\a}{2}\, {c_0\,\bar{c}_2} \frac{ 1 + o(1) }{ (\l d)^{4}} + \a_0\,\bar{c_5} p \frac{c}{\l^{5/2} } =  o\Big(\frac{1}{(\l d)^{4}} + \frac{1}{\l^{5/2}} \Big)+ O\Big( \e^2 +  \frac{d^{p+1-\e}}{\l^{ 5 /2}} \Big) $$ 
which implies that 
$$ \frac{ 1 }{ ( \l d )^4} + \frac{ 1 }{\l^{5/2} } \leq c \e ^2 . $$
Putting this information in \eqref{fin-01}, we obtain 
\be \label{fin-03}  \bar{c}_3 \a \, \e  +  \a_0 \bar{c}_2 \frac{ 3 }{2}\frac{ u_0(a ) }{\l ^{ 3 / {2}}}  =  o \Big(  \e + \frac{d }{\l ^{ {3} / {2}}} \Big)  \ee 
which gives a contradiction (since the function $ u_0 $ satisfies: $ u_0 (x) \geq c d(x, \partial \O) $ for each $ x \in \O $). \\
This completes the proof of Theorem \ref{t1} for $ n = 5 $. 
\end{pfn}


It remains to prove Theorem \ref{t1} in the case $n=4$. Recall that $ \| v_\e \| ^2 $ appears in Proposition \ref{pas-geq5}. For $n=4$, this term contains $ \frac{1}{\l^2}$, which is a principal term. Hence, we need to overcome this difficulty. To this end, we decompose  $ v_\e $ into odd and even parts.
 To be more precise, let $ B_2:= B(a, d/2) $ and let us define the projection of $ v_\e $ onto $ H_0^1( B_2) $ as the solution of the following PDE
\be \label{Q-eps} - \D  \wtilde{v}_\e = ( - \D + V ) v_\e    \quad  \mbox{ in } B_2 , \qquad \mbox{ and } \qquad  \wtilde{v}_\e = 0 \quad \mbox{ on } \, \partial B_2 . \ee 
In addition, for $ k \in \{ 1 , \cdots , n \}$, we decompose $  \wtilde{v}_\e $ as 
\be \label{odd-even}  \wtilde{v}_\e =  \wtilde{v}_{e , \e} +  \wtilde{v}_{o,\e} \ee where  $ \wtilde{v}_{e,\e} $ is   even and  $ \wtilde{v}_{o,\e} $ is odd with respect to the variable $(x-a)_k$.  We start by the following lemma 

\begin{lem} \label{est-1-2} Let $ n = 4 $. Then the following statements hold
\be \label{eq:est-1-2} (i) \quad | v_\e -  \wtilde{v}_\e | \leq \frac{ c }{ d } \| v_\e \| \quad \mbox{ in } \, B_4 := B(a,d/4) \, \, ; \qquad \quad (ii) \quad \|  \wtilde{v}_\e \|_{ H_0^1(B) } \leq \| v_\e \| . \ee
\end{lem} 
\begin{pf} Using \eqref{Q-eps}, we deduce that 
\begin{align*} \|  \wtilde{v}_\e \|^2_{ H_0^1(B_2) } = \int_{B_2} | \nabla  \wtilde{v}_\e |^2  = \int_{B_2} (- \D  \wtilde{v}_\e )  \wtilde{v}_\e  = \int_{B_2} (- \D {v}_\e + V v _\e  )  \wtilde{v}_\e &  \leq  c \|  \wtilde{v}_\e \|_{ H_0^1(B_2) }  \| v_\e \|_{H^1(B_2)}  \end{align*}
which implies the assertion $(ii)$  of the lemma. 
Concerning assertion $(i)$, let $ G_{B_2} $ be the Green's function on $ B_2 $ with Dirichlet boundary condition. Since $ \D (v_\e -  \wtilde{v}_\e)= V v_\e $ in $ B_2 $,  it holds
\be \label{GB1}   (  v_\e -  \wtilde{v}_\e )(x) =  - \int _{B_2} G_{B_2}(x,y) V v_\e (y) dy -  \int_{ \partial B_2 } \frac{ \partial G_{B_2}}{ \partial \nu } (x,y) v_\e (y) dy . \ee
Now, let $ x \in B_4 $, it follows that $ | x-y | \geq c\,  d $ for each $ y \in \partial B_2 $. Hence, \eqref{GB1} implies that 
\be \label{GB2} \Big|  \int_{ \partial B_2 } \frac{ \partial G_{B_2}}{ \partial \nu } (x,y) v_\e (y) dy \Big|  \leq \frac{ c }{ d^{ 3 }}  \int_{ \partial B_2 } | v_\e (y) |  dy . \ee
But, using the continuity of the embedding  (with $ n = 4 $) 
$ H ^ 1( B_2 ) \to L ^{3 } ( \partial B_2 ), $
we deduce that 
$$  \int_{ \partial B_2 } | v_\e (y) |  dy \leq  \Big( \int_{ \partial B_2 } | v_\e (y) |^{3 }  dy \Big) ^{\frac{ 1 } {3} } \Big( \int_{ \partial B_2 } 1  dy \Big) ^{\frac{ 2 }{3} }  \leq c d ^{ 2} \| {v}_\e \|_{ H ^1(B_2) } . $$ 
Putting this information \eqref{GB2}, we get 
$$ \Big|  \int_{ \partial B_2 } \frac{ \partial G_{B_2}}{ \partial \nu } (x,y) v_\e (y) dy \Big| \leq c \frac{ d^{ 2} }{ d^{ 3 }}  \| {v}_\e \|_{ H^1(B_2) }  \leq  \frac{ c }{ d }  \| {v}_\e \|_{ H_0^1( \O ) }   $$ 
which gives the estimate of the second integral in \eqref{GB1}. Concerning the first one, using the Holder's inequality, we get 
\begin{align*} \Big| \int _B G_B(x,y) V v_\e (y) dy  \Big| & \leq c \int _B \frac{1}{ | x-y |^{2}} | v_\e | 
  \leq c \| v_\e \| _{L^{ 4 } (B)} \Big( \int_B \frac{ 1 }{ | x-y |^{ 8/3} } \Big) ^{ 3/4} \leq c \, d \, \| v_\e \| 
\end{align*}
which completes the proof of the lemma. 
\end{pf}

Now, we focus on estimating the odd part $ \wtilde{v}_{o, \e} $. We have 
\begin{lem}  Let $ n = 4   $. The function $ \wtilde{v}_{o, \e} $ defined in \eqref{odd-even} satisfies 
 \be \label{eq:est-odd}\|  \wtilde{v}_{o, \e} \|_{H_0^1(B_2) }  \leq c | 1 - \a_0^{2-\e} | + c  \| v _\e \|   \Big[ \frac{ 1 }{  \l d  } + \frac{ d }{ \l} \ln ( \l d )  + d^{4-\e} \Big]  +   \Big[ \e + \frac{d}{ \l} + \frac{1}{ (\l d )^2}      \Big]  + \| v_\e \|^2  . \ee
 \end{lem}
\begin{pf} Multiplying the first equation of $ (P_\e) $ by $ \wtilde{v}_{o, \e} $ and integrating over $ B_2 $, we derive that 
\be \label{fin-41} \int_{B_2} (- \D ) [ \a_0 u_0 + \a P \d_{a, \l} + v_\e ]  \wtilde{v}_{o, \e} + \int_{B_2} V  [ \a_0 u_0 + \a P \d_{a, \l} + v_\e ] \wtilde{v}_{o, \e}   = \int_{B_2} u_\e  ^{3-\e } \wtilde{v}_{o, \e} . \ee
Observe that 
\be \label{fin-42} \int_{B_2} ( - \D v_\e + V v_\e ) \wtilde{v}_{o, \e} = \int_{B_2} ( - \D \wtilde{v}_{o, \e} ) \wtilde{v}_{o, \e} = \| \wtilde{v}_{o, \e} \| ^2 _{ H_0^1(B_2)} . \ee 
In addition, by oddness, we have 
\begin{align} 
& \int_{B_2} ( - \D P \d_{a, \l} ) \wtilde{v}_{o, \e} = \int_{B_2} \d_{a, \l} ^3 \wtilde{v}_{o, \e} = 0 , \label{fin-43} \\
 &  \int_{B_2} V P \d_{a, \l} \wtilde{v}_{o, \e} =   \int_{B_2} V  \d_{a, \l} \wtilde{v}_{o, \e} -  \int_{B_2} V  \theta_{a, \l} \wtilde{v}_{o, \e} \notag   \\
  & \qquad \qquad \qquad = V(a) \times 0 + O \Big( \int_{B_2} | x-a | \d_{a, \l} | \wtilde{v}_{o, \e} |  + \frac{ 1 }{  \l d ^2 } \int_{B_2} | \wtilde{v}_{o, \e} |\Big) \notag  \\
   & \qquad \qquad \qquad =  O \Big( \frac{1}{ \l } \int_{B_2} \frac{1}{  | x-a | } |  \wtilde{v}_{o, \e} | + 
     \|  \wtilde{v}_{o, \e} \|_{H_0^1(B_2) } \Big[  \frac{ 1}{ \l } \frac{ d^{ 3}}{ d^{ 2}}   \Big] \Big) = O \Big( \frac{ d }{ \l }  \|  \wtilde{v}_{o, \e} \|_{H_0^1(B_2) } \Big) . 
 \label{fin-44} \end{align}
Concerning the right hand side, as before, let $ W := \a_0 u_0 + \a P \d_{a, \l}  $, we have 
\be \label{fin-21} \int_{B_2} u_\e  ^{3-\e } \wtilde{v}_{o, \e}  = \int_{B_2} W^{3-\e } \wtilde{v}_{o, \e}  + (3-\e )\int_{B_2} W ^{2-\e } v_\e \wtilde{v}_{o, \e}  + O \Big( \| v_\e \|^2 \| \wtilde{v}_{o, \e}  \| \Big) . \ee
Observe that 
\begin{align*} &  \int_{B_2}  W ^{2-\e } v_\e \wtilde{v}_{o, \e} \\
 & = \a ^{2-\e} \int_{B_2} \d_{a, \l}^{2-\e} v_\e \wtilde{v}_{o, \e}  + O \Big( \int_{B_2} \d_{a, \l} \theta_{a, \l} | v_\e | | \wtilde{v}_{o, \e} | + \int_{B_2} \d_{a, \l} u_0 | v_\e | | \wtilde{v}_{o, \e} | + \int_{B_2} u_0^{2-\e} | v_\e | | \wtilde{v}_{o, \e} | \Big) \\
 & = \a ^{2-\e} \int_{B_2} \d_{a, \l}^{2-\e} v_\e \wtilde{v}_{o, \e}  + O \Big( \| v_\e \| \| \wtilde{v}_{o, \e} \| \Big[ \frac{ \ln ( \l d )}{ (\l d )^2} + \frac{ d }{ \l} \ln ( \l d )  + d^{4-\e} \Big] \Big). 
  \end{align*}
But, using Lemma \ref{est-1-2}, by oddness, we get 
\begin{align*} \int_{B_2} \d_{a, \l}^{2-\e} v_\e \wtilde{v}_{o, \e} & = \int_{B_4} \d_{a, \l}^{2-\e}  \wtilde{v}_\e \wtilde{v}_{o, \e} + \int_{B_4} \d_{a, \l}^{2-\e} (v_\e -  \wtilde{v}_\e) \wtilde{v}_{o, \e} + \int_{B_2 \setminus B_4} \d_{a, \l}^{2-\e} v_\e \wtilde{v}_{o, \e}   \\
& = \int_{B_4} \d_{a, \l}^{2}  \wtilde{v}_{o, \e}^2 + o ( \| \wtilde{v}_{o, \e} \| ^2 ) + O \Big( \frac{ \| v_\e \| }{ d } \int_{B_4} \d_{a, \l}^{2} | \wtilde{v}_{o, \e} | + \| v _\e \| \| \wtilde{v}_{o, \e} \| \frac{ 1 }{ ( \l d )^2 } \Big) . 
\end{align*}
Hence we obtain 
\be \label{fin-11} \int_{B_2}  W ^{2-\e } v_\e \wtilde{v}_{o, \e} = \int_{B_2} \d_{a, \l}^{2}  \wtilde{v}_{o, \e}^2 + o ( \| \wtilde{v}_{o, \e} \| ^2 ) + O \Big(  \| v _\e \| \| \wtilde{v}_{o, \e} \|  \Big[ \frac{ 1 }{  \l d  } + \frac{ d }{ \l} \ln ( \l d )  + d^{4-\e} \Big]\Big)   \ee
which gives the estimate of the second integral of \eqref{fin-21}. Concerning the first one, observe that
$$ \int_{B_2} W^{3-\e } \wtilde{v}_{o, \e} = \a ^{3-\e} \int_{B_2} P\d_{a, \l}^{3-\e} \wtilde{v}_{o, \e} + \a_0 ^{3-\e} \int_{B_2} u_0^{3-\e} \wtilde{v}_{o, \e} + O \Big(  \int_{B_2} \d_{a, \l}^{2} u_0 | \wtilde{v}_{o, \e} |  +  \int_{B_2} \d_{a, \l} u_0^{2-\e} | \wtilde{v}_{o, \e} | \Big) . $$
In addition, it holds 
\begin{align*}
& \int_{B_2} \d_{a, \l} u_0^{2-\e} | \wtilde{v}_{o, \e} | \leq c \frac{ d^{2-\e} }{ \l}  \int_{B_2} \frac{1}{ | x-a |^2 }   | \wtilde{v}_{o, \e} | \leq c  \frac{d^{2-\e}}{ \l } \|  \wtilde{v}_{o, \e} \| \Big( \int_{B_2} \frac{1}{ | x-a |^{8/3} } \Big) ^{3/4} \leq c  \frac{d^{3-\e}}{ \l } \|  \wtilde{v}_{o, \e} \| ,  \\
& \int_{B_2} \d_{a, \l}^{2} u_0 | \wtilde{v}_{o, \e} | \leq c d \int_{B_2} \d_{a, \l}^{2}  | \wtilde{v}_{o, \e} | \leq c d  \|  \wtilde{v}_{o, \e} \| \Big( \int_{B_2} \d_{a, \l }^{8/3} \Big) ^{3/4} \leq c \frac{ d }{ \l}  \|  \wtilde{v}_{o, \e} \| . 
\end{align*}
For the other integrals, we have 
\begin{align*}
& \int_{B_2} u_0^{3-\e} \wtilde{v}_{o, \e} =  \int_{B_2} u_0^{3} \wtilde{v}_{o, \e}  + O \Big( \e \int_{B_2}|  \wtilde{v}_{o, \e} | \Big) = \int_{B_2} \n u_0 \cdot \n  \wtilde{v}_{o, \e} + \int_{B_2} V u_0 \wtilde{v}_{o, \e} + O ( \e \| \wtilde{v}_{o, \e} \| ) ,   \\
 & \int_{B_2} P\d_{a, \l}^{3-\e} \wtilde{v}_{o, \e} = \int_{B_2} \d_{a, \l}^{3-\e} \wtilde{v}_{o, \e}  + O \Big( \int_{B_2} \d_{a, \l}^{2} \theta_{a, \l} | \wtilde{v}_{o, \e} | \Big) = O \Big( \frac{ \|  \wtilde{v}_{o, \e} \| }{ ( \l d )^2} \Big) .
\end{align*}
Thus we obtain 
\be \label{fin-22} \int_{B_2} W^{3-\e } \wtilde{v}_{o, \e} = \a_0 ^{3-\e} \Big( \int_{B_2} \n u_0 \cdot \n  \wtilde{v}_{o, \e} + \int_{B_2} V u_0 \wtilde{v}_{o, \e} \Big) +  O \Big( \| \wtilde{v}_{o, \e} \| \Big[ \e + \frac{d}{ \l} + \frac{1}{ (\l d )^2}      \Big] \Big) . 
\ee
Combining \eqref{fin-21}, \eqref{fin-11} and \eqref{fin-22}, we get 
\begin{align} \label{fin-23} 
\int_{B_2} & u_\e  ^{3-\e }  \wtilde{v}_{o, \e}  =  3  \int_{B_2} \d_{a, \l}^{2}  \wtilde{v}_{o, \e}^2  + o ( \| \wtilde{v}_{o, \e} \| ^2 ) +  \a_0 ^{3-\e} \Big( \int_{B_2} \n u_0 \cdot \n  \wtilde{v}_{o, \e} + \int_{B_2} V u_0 \wtilde{v}_{o, \e} \Big) \notag  \\
 &  + O \Big(  \| v _\e \| \| \wtilde{v}_{o, \e} \|  \Big[ \frac{ 1 }{  \l d  } + \frac{ d }{ \l} \ln ( \l d )  + d^{4-\e} \Big]  +  \| \wtilde{v}_{o, \e} \| \Big[ \e + \frac{d}{ \l} + \frac{1}{ (\l d )^2}      \Big]  + \| v_\e \|^2 \| \wtilde{v}_{o, \e} \| \Big) . 
 \end{align}
Combining \eqref{fin-41}--\eqref{fin-44} and \eqref{fin-23}, we obtain  
\begin{align} & \| \wtilde{v}_{o, \e} \| ^2 _{ H_0^1(B_2)} - 3 \int_{B_2} \d_{a, \l}^{2}  \wtilde{v}_{o, \e}^2 + o ( \| \wtilde{v}_{o, \e} \| ^2 ) + \a_0 ( 1 -  \a_0 ^{2-\e} ) \Big( \int_{B_2} \n u_0 \cdot \n  \wtilde{v}_{o, \e} + \int_{B_2} V u_0 \wtilde{v}_{o, \e} \Big) \notag  \\
 & = O \Big(   \| v _\e \| \| \wtilde{v}_{o, \e} \|  \Big[ \frac{ 1 }{  \l d  } + \frac{ d }{ \l} \ln ( \l d )  + d^{4-\e} \Big]  +  \| \wtilde{v}_{o, \e} \| \Big[ \e + \frac{d}{ \l} + \frac{1}{ (\l d )^2}      \Big]  + \| v_\e \|^2 \| \wtilde{v}_{o, \e} \| \Big) . \label{fin-51}
\end{align}
Notice that $\wtilde{v}_{o, \e}$ is not orthogonal to $\partial P\d /\partial a_k$. Following the proof of Lemma $2.8$ in \cite{BF} (see $6.1$ and $6.8$ in \cite{BF}), we deduce that 
\be \label{fin-52} \| \wtilde{v}_{o, \e} \| ^2 _{ H_0^1(B_2)} - 3 \int_{B_2} \d_{a, \l}^{2}  \wtilde{v}_{o, \e}^2 \geq c \| \wtilde{v}_{o, \e} \| ^2 _{ H_0^1(B_2)} . \ee
Finally, \eqref{fin-51} and \eqref{fin-52} completes the proof of \eqref{eq:est-odd}. 
\end{pf}

Now, we are able to state the analogous of Proposition \ref{pas-geq5} for  dimension $ n = 4 $.  
\begin{pro} \label{pas-n-4} Let $ n = 4 $ and $ k \in \{ 1, \cdots, n \}$. Then the following statement holds 
\be \label{a-fin-11} - 3 \frac{\ov{ c }_5}{ \l ^2 }  \frac{ \partial u_0 (a) }{\partial a_k}   + \frac{1}{2} \a  \frac{c_0\bar{c}_2}{\l^{3}}\frac{\partial H}{\partial a_k}(a,a) = o\Big(\frac{1}{(\l d)^{3}} + \frac{1}{\l^{2}}\Big) + O \Big( \e ^2 + \frac{ d }{ \l^2 } + \frac{ \e }{ \l  }  \Big) .   \ee 
\end{pro}
\begin{pf} First, using Lemma \ref{L41}, we know that $ \lim \e \ln \l = 0 $. Furthermore, Propositions \ref{p2}, \ref{rp41},  \ref{rp42} and \ref{rp43}  are valid.
Now, we follow the proof of Proposition \ref{pas-geq5}, but with the solution $ u_\e $. Let $ k \in \{1, \cdots, n \}$, multiplying the first equation of $ ( \mathcal{P}_{V,\e} ) $ by $ \l^{-1} \partial P \d_{a, \l} / \partial a_k $, we obtain 
\be \label{k52as-bis}
 \a_0 (u_0,\frac{1}{\l} \frac{\partial P\d_{a, \l}}{\partial a_k}) + \a \int_{\O} V \, P\d_{a, \l} \, \frac{1}{\l} \frac{\partial P\d_{a, \l}}{\partial a_k} 
+ \a \int_{\O} \n P\d_{a, \l} \cdot  \n (\frac{1}{\l} \frac{\partial P\d_{a, \l}}{\partial a_k}) =  \int_{\O} u_\e ^{p - \e}  \frac{1}{\l} \frac{\partial P\d_{a, \l}}{\partial a_k}   . 
\ee
Notice that the second integral is computed in \eqref{k54bisas} for $ n \geq 5 $ and, in the same way, for $ n = 4 $, we get 
$$ \int_{\O} V \, P\d_{a, \l} \, \frac{1}{\l} \frac{\partial P\d _{a, \l}}{\partial a_k } = \n V(a) \, \rho_n \frac{\ln(\l d)}{\l^{3}} + \, o\Big(\frac{1}{(\l d)^{3}} + \frac{1}{\l^{2}}\Big) = o\Big(\frac{1}{(\l d)^{3}} + \frac{1}{\l^{2}}\Big). $$
For the third integral, using Lemma A.10 of \cite{BCH}, we have 
$$ \int_{\O} \n P\d_{a, \l} \cdot \n (\frac{1}{\l} \frac{\partial P\d_{a, \l}}{\partial a_k}) = - \frac{1}{2} \frac{ c_0 \ov{c}_2 }{ \l^{ 3} } \frac{ \partial H }{ \partial a_k } (a, a ) + O \Big( \frac{1}{ ( \l d )^4 } \Big).  $$
Now, we will focus on the right hand side of \eqref{k52as-bis}. We need to refine \eqref{ms1as}. Using the same notation, we have 
\begin{align} \label{cl3} 
\int_{\O}   u_\e ^{3 - \e}  \frac{1}{\l} \frac{\partial P\d_{a, \l} }{\partial a_k} = & \int_{\O} W^{3-\e} \, \frac{1}{\l} \frac{\partial P\d _{a, \l} }{\partial a_k} + (3-\e)\int_{\O}W^{ 2-\e} \, {v_\e } \, \frac{1}{\l} \frac{\partial P\d_{a, \l}}{\partial a_k} \notag\\
&+\,\frac{(3-\e)(2-\e)}{2} \int_\O W^{1-\e} \, {v_\e ^2} \, \frac{1}{\l} \frac{\partial P\d_{a, \l}}{\partial a_k} +\,O(\| v_\e \|^3) .
\end{align}
Concerning the first integral in \eqref{cl3}, combining  \eqref{as10}, \eqref{as12}, \eqref{as12b}, \eqref{as13} and  \eqref{as18}, we obtain 
\begin{align*}
 \int_{\O} & W^{3-\e} \,  \frac{1}{\l} \frac{\partial P\d _{a, \l} }{\partial a_k}  
  =  \a_0 ^{3-\e} ( u_0 , \frac{1}{\l} \frac{\partial P\d _{a, \l} }{\partial a_k}) + (3 - \e ) \a^{2-\e} c_0^{-\e} \l^{-\e} \frac{ \ov{ c }_5}{ \l } \frac{ \partial u_0 (a) }{\partial a_k}   \\
  &- \a^{3-\e} c_0^{-\e} \l^{-\e}  \frac{c_0\bar{c}_2}{\l^{3}}\frac{\partial H}{\partial a_k}(a,a) 
  +O\Big(\frac{\e}{(\l d)^{2}}+\frac{1}{(\l d)^{4}} + \frac{\e d }{ \l} + \frac{1 }{ \l^2 ( \l d ) }    + \frac{ d^{4-\e} }{ \l^2} + \frac{ \ln \l }{ \l^2 (\l d ) } \Big). 
 \end{align*}
In addition, combining \eqref{ms2as}--\eqref{as7}, we deduce that 
$$ \int_{\O}W^{ 2-\e} \, {v_\e } \, \frac{1}{\l} \frac{\partial P\d_{a, \l}}{\partial a_k}  = O \Big( { \| v_\e \| } \Big[  \e + \frac{ d }{ \l } +  \frac{1}{ (\l d ) ^2 }  \Big] \Big) . $$
It remains to estimate the last integral in \eqref{cl3}. Observe that 
 \be \label{a-fin-1}  \int_\O W^{1-\e} \, {v_\e ^2} \, \frac{1}{\l} \frac{\partial P\d_{a, \l}}{\partial a_k} =  \int_\O ( \a P \d_{a, \l} )^{1-\e} \, {v_\e ^2} \, \frac{1}{\l} \frac{\partial P\d_{a, \l}}{\partial a_k} + O \Big(  \int_\O ( \a_0 u_0)^{1-\e} \, {v_\e ^2} \, \frac{1}{\l} \Big| \frac{\partial P\d_{a, \l}}{\partial a_k} \Big| \Big) . \ee 
In addition,  it holds 
$$ \int_\O ( \a_0 u_0)^{1-\e} \, {v_\e ^2} \, \frac{1}{\l} \Big| \frac{\partial P\d_{a, \l}}{\partial a_k} \Big| \leq c \int_{ \O }  {v_\e ^2}  \d_{a, \l} \leq c \| v_\e \| ^2 \Big( \int_\O \d_{a, \l}^2 \Big)^{1/2} \leq c \frac{ \ln (\l)^{1/2} }{ \l } \| v_\e \| ^2  .$$
Concerning the first integral in \eqref{a-fin-1}, let $ B_4 := B(a, d/4)$. Then we write 
\begin{align} \int_\O ( P \d_{a, \l} )^{1-\e} \, {v_\e ^2} \, \frac{1}{\l} \frac{\partial P\d_{a, \l}}{\partial a_k} = & \int_{B_4}  \d_{a, \l} ^{1-\e} \, {v_\e ^2} \, \frac{1}{\l} \frac{\partial \d_{a, \l}}{\partial a_k} \notag  \\
 & + O \Big( \int_{B_4}  \d_{a, \l}  \theta_{a, \l} \, {v_\e ^2}   + \int_{B_4}  \d_{a, \l}  {v_\e ^2} \, \frac{1}{\l} \Big| \frac{\partial \theta_{a, \l}}{\partial a_k} \Big| + \int_{ \O \setminus B_4} \d_{a, \l} ^2 {v_\e ^2} \Big) . \label{a-fin-2} \end{align}
The remainder terms can be estimated as 
\begin{align*} \int_{B_4}  \d_{a, \l}  \theta_{a, \l} \, {v_\e ^2}   +   \int_{B_4}  \d_{a, \l}  {v_\e ^2} \, \frac{1}{\l} \Big| \frac{\partial \theta_{a, \l}}{\partial a_k} \Big|+ \int_{ \O \setminus B_4} \d_{a, \l} ^2 {v_\e ^2} &  \leq  \int_{ \O \setminus B_4 }  \d_{a, \l}^2   {v_\e ^2} + \frac{ c }{ \l d ^2 }  \int_{B_4}  \d_{a, \l}  {v_\e ^2} \\
& \leq c \frac{ \ln ( \l d )^{1/2} }{ ( \l d )^2 }  \| v_\e \|^2 . 
 \end{align*}
For the first integral, using \eqref{eq:est-1-2}, by oddness, we have 
\begin{align*}
 \int_{B_4}  \d_{a, \l} ^{1-\e} \, {v_\e ^2} \, \frac{1}{\l} \frac{\partial \d_{a, \l}}{\partial a_k} &  = \int_{B_4}  \d_{a, \l} ^{1-\e} \, ( \wtilde{v}_{0,\e} +\wtilde{v}_{e,\e} )  ^2 \, \frac{1}{\l} \frac{\partial \d_{a, \l}}{\partial a_k} +  \int_{B_4}  \d_{a, \l} ^{1-\e} \, ({v_\e }-\wtilde{v}_\e )({v_\e } + \wtilde{v}_\e ) \, \frac{1}{\l} \frac{\partial \d_{a, \l}}{\partial a_k} \\
 & = 2 \int_{B_4}  \d_{a, \l} ^{1-\e} \,  \wtilde{v}_{0,\e} \wtilde{v}_{e,\e}  \, \frac{1}{\l} \frac{\partial \d_{a, \l}}{\partial a_k} + O \Big( \frac{ \| v_\e \| }{ d } ( \| v_\e \| + \|  \wtilde{v}_\e \| ) \Big( \int_{B_4} \d_{a, \l} ^{8/3}  \Big)^{3/4} \\
  & =  O\Big( \|  \wtilde{v}_{0,\e} \| \| \wtilde{v}_{e,\e} \| + \frac{1}{ \l d }  \| v_\e \| ^2 \Big) . 
 \end{align*} 
Hence \eqref{a-fin-2} becomes 
$$  \int_\O ( P \d_{a, \l} )^{1-\e} \, {v_\e ^2} \, \frac{1}{\l} \frac{\partial P\d_{a, \l}}{\partial a_k} =  O\Big( \|  \wtilde{v}_{0,\e} \| \| \wtilde{v}_{e,\e} \| + \frac{1}{ \l d }  \| v_\e \| ^2 \Big) . $$
Combining the previous estimates with  \eqref{eq:est-odd} and \eqref{est-v-eps-01}, the proof of Lemma \ref{pas-n-4} follows 
\end{pf}

Now we are able to prove Theorem \ref{t1} for $ n = 4 $. 

\begin{pfn}{ \bf of Theorem \ref{t1} in the case $ n = 4 $. }
Using \eqref{a-fin-4}, \eqref{a-fin-5} and the fact that $ u_\e $ is a critical point of $ I_\e $, Proposition \ref{rp43} implies that 
 \be  \label{a-fin-12}- \a c_0\bar{c}_2 \frac{H(a ,a )}{\l ^ 2 }   + \bar{c}_3 \, \e \, \a  + \,  \a_0 \bar{c}_2  \frac{ u_0(a) }{\l } = o \Big( \e + \frac{d}{ \l } + \frac{ 1 }{ (\l d )^2 } \Big) .    \ee 
Notice that, we have $ u_0(a) \geq c d $ and $ H(a,a) \geq c / d^2 $.  Hence we deduce that $ \e \leq { c } / { ( \l d )^2} $. 
 According to Theorem \ref{t0}, it suffices to prove that  $ a $ is far away from the boundary. 
Arguing by contradiction, assume that  $ d \to 0 $. In this case, since $ u_0 $ is a solution of $ ( \mathcal{P}_{V,0})$, it follows that $ \partial u_0 / \partial \nu < 0 $ on the boundary. Let $ \ov{a} = \lim a \in \partial \O $ and $ \nu $ be the outward normal vector to $ \partial \O $ at $ \ov{a} $. Using \eqref{a-fin-11}, we get 
$$ - 3 \frac{\ov{ c }_5}{ \l ^2 }  \frac{ \partial u_0 (a) }{\partial \nu}   + \frac{1}{2} \a  \frac{c_0\bar{c}_2}{\l^{3}}\frac{\partial H}{\partial \nu}(a,a) = o\Big(\frac{1}{(\l d)^{3}} + \frac{1}{\l^{2}}\Big) + O \Big( \e ^2 + \frac{ d }{ \l^2 } + \frac{ \e }{ \l  }  \Big)    $$
and therefore, using \eqref{H-bord} and the fact that $ \e \leq c / ( \l d )^2 $, we derive that 
$$ \frac{ c }{ \l^2} + \frac{ c }{ ( \l d )^3 } =  o\Big(\frac{1}{(\l d)^{3}} + \frac{1}{\l^{2}}\Big)    $$
 which gives a contradiction. Hence the proof of Theorem \ref{t1} is completed for $ n = 4 $. 
 \end{pfn}

\section{The Influence of the Sign of the Normal Derivative of $V$ on the Boundary}

This  section will discuss how the sign of the normal derivative of the potential $V$ on the boundary influences the concentration points of single blowing-up solutions of $(\mathcal{P}_{V,\e})$ with residual mass. More precisely, our aim is to prove Theorems \ref{t2} and \ref{t3}.
\subsection{The Normal Derivative of $V$ Is Positive on the Boundary}

The goal of this subsection is to prove Theorem \ref{t2}. More precisely, we assume that $ n \geq 6$,  $\frac{\partial V}{\partial\nu} >0$ on the boundary of $\Omega$ and our aim is to prove that any single blow-up solution $u_\e$ of  $(\mathcal{P}_{V,\e})$ with residual mass has to blow up in the interior of $\Omega$. Proceeding by contradiction, we write $u_\e$ in the form \eqref{KA4} (with $N=1$) such that \eqref{KA5} holds, and assume that $d:=d(a_\e,\partial\O) \to 0$ as $\e\to 0$. Lemma \ref{L41} implies that Propositions \ref{p2}, \ref{rp41}, \ref{rp42} and \ref{rp43} are valid. In addition, since $u_\e$ is a solution, it follows that the left-hand sides in Proposition \ref{rp41}-\ref{rp43} equal to zero. Therefore, we derive the following balancing conditions
\begin{align}
 & \| v_\e \|^2 \leq   c \Big(  \e^2 + \frac{1}{\lambda ^{ 4 } } + \frac{1}{(\l d )^{n + 2}} + (\mbox{if } n=6) \Big[ \frac{ \ln ( \l )^{4/3}}{\lambda ^{ 4 } } + \frac{ \ln^2 ( \l d )}{(\l d )^{ 8 }} \Big]   \Big) , \notag  \\
& |  1 - \a_0^{p-1-\e} | \leq  c \Big( \e +  \frac{1}{\l ^{ 4 }}  + \frac{1}{\l ^{ (n-2)/2 }}  +  \frac{1}{(\l d )^{n + 2}} + (\mbox{if } n=6) \frac{ \ln^2 ( \l d )}{(\l d )^{ 8 }} \Big) , \notag  \\
& | 1 - \a ^{p-1-\e} c_0^{-\e}\l ^{\frac{ - \e(n-2) }{2}} | \leq  c \Big( \e + \frac{ 1 }{ \l ^2 }+\frac{1}{(\l d )^{n-2}} \Big) , \notag \\
& - c_0 \ov{c}_2 \frac{n-2}{2} \frac{ H(a,a)}{ \l^{n-2}} - V(a)  \frac{ \g_n }{ \l^2} + \ov{c}_3 \, \e = o \Big( \e + \frac{1}{ \l^2} + \frac{1}{ ( \l d )^{n-2} } \Big) . \label{po12}
 \end{align}
In addition,  since $ a $ is close to $\partial \O$, there exists a unique $ \ov{a} \in \partial \O $ such that $ | a - \ov{ a } | = d := d(a, \partial \O)$. Multiplying the estimate given in Proposition \ref{pas-geq5} by the outward normal vector  $ \nu $ to $ \partial \O $ at the point $ \ov{ a} $, we obtain 
$$ \a \frac{ \rho_n}{ \l ^3 } \frac{ \partial V}{ \partial \nu } (a) + \a c_0 \ov{c}_2 \frac{1}{2} \frac{ 1 }{ \l ^{n-1}}  \frac{ \partial H}{ \partial \nu } (a, a ) - \a_0 p \frac{ \ov{c}_5}{ \l^{n/2}}  \frac{ \partial u_0}{ \partial \nu } (a) = o \Big( \frac{1}{ (\l d ) ^{ n-1}} + \frac{1}{\l ^ 3} \Big) + O ( \e^2) . $$
Using \eqref{H-bord}, \eqref{u0-nu} and  $\frac{\partial V}{\partial\nu} >0$, we derive that 
$$ \frac{1}{ \l^3} + \frac{1}{ (\l d )^{n-1}}  \leq c \e ^2 . $$ 
Putting this information in \eqref{po12}, we obtain $ \e = o( \e) $ which presents a contradiction. \\
This completes the proof of the theorem.

\subsection{The Normal Derivative of $V$ Is Negative on the Boundary}

The aim of this subsection is to prove Theorem \ref{t3}. Let $ n \geq 7 $ and  $b\in\partial\O$ be a non-degenerate critical point of the restriction of $V$ to the boundary satisfying $\frac{\partial V}{\partial\nu}(b)<0$. To construct a solution of the form $u_\e=\a_{0,\e}\,u_0 +\a_\e P\d_{a_\e,\l_\e} +v_\e:=\a_0\,u_0 +\a P\d_{a,\l}+v$ such that $a_\e\to b$ as  $\e \to 0$, we employ the same method as in the proof of Theorem \ref{t4} given in Section $5$. The approach reduces the problem to a finite-dimensional system after $v_\e$ is found via Proposition \ref{p2}. Hence, given $(\a_0, \a, \l, a)$ in the following set 
\begin{align}\label{b1}
\mathbb{A}(\e) =\lbrace  (\alpha_0, \alpha, \lambda, a) \in & \mathbb{R}_+\times \mathbb{R}_+ \times \mathbb{R}_+ \times \Omega : \vert\alpha_0-1\vert <  \,\varepsilon \,\ln^2 \varepsilon, \quad \vert\alpha -1\vert <  \,\varepsilon \,\ln^2 \varepsilon,\nonumber\\ 
& \frac{1}{c} < \lambda^2 \, \e < c,\, \vert a - b \vert <  \,\varepsilon^{1/5} \quad \frac{1}{c}\leq \frac{d(a,\partial\O)}{\e^{ ( n-4) / [2(n-1)] }}\leq c  \rbrace, 
\end{align}
where $c$ is a fixed positive constant, we need to solve the  system \eqref{5.11}-\eqref{5.14} with $N=1$.\\
Let $\bar{a} \in\partial \O$ such that $|a-\bar{a}|=d(a,\partial\O)$ and let $\nu_{\bar{a}}$ (resp $\nu_b$) be the unit outward normal vector at $\bar{a}$ (rep $b$). Without loss of generality, we can assume that $a=\bar{a}-d\nu_{\bar{a}}$, where $d=d(a,\partial\O)$. According to Lemma $7.5$ of \cite{BJMAA}, we know that there exists a smooth function $g: \, T_{b} ( \partial \O) \to \R$ such that, for each $x\in T_{b} ( \partial \O) \cap B(b, r_0)$ (with $r_0> 0$ small) we have $x - g(x)\nu_{b} \in \partial \O$,   $g(b) = \nabla g(b) = 0$, and $ D^2( V_{|\partial \O } )$. Recall that $ b $ is a critical point of $ V_{| \partial \O }$ which implies that 
\be \label{moh5} \nabla V(b) \cdot y = 0 \qquad \forall \, \, y \in T_{b}( \partial \O ) . \ee 
Furthermore,  $ D^2 V $ are related by the following formula:
\be \label{127**} D^2 (V_{ | \partial \O })(b) (\o_1, \o_2 ) = D^2 V (b) (\o_1, \o_2 ) - \frac{ \partial V}{ \partial \nu} (b) D^2 g (b) ( \o_1  , \o_2 ) . \ee 
Furthermore, we have 
\begin{itemize}
\item $\ov{a}$ can be written as 
$
\ov{a} := b + \ov{a}' + \mathcal{O}(| \ov{a} -b | ^2), $  with $ \ov{a}' \in T_{b} (\partial \O).$
\item For each $y \in T_{b} (\partial \O)$, it holds
 \be \label{moh2} \langle \nu_{\ov{a}}, y \rangle = D^2g(b) ( \ov{a}'-b, y ) +  \mathcal{O}(| \ov{a} -b | ^2 |y|). \ee 
\end{itemize} 
Observe that, Propositions \ref{rp43} and \ref{pas-geq5} can be written as 
\begin{align}
 & \Big\langle   I^{\prime}_\e(u),\l  \frac{\partial P\d }{\partial \l }\Big\rangle =   - \a \, V(a )  \frac{ \gamma_n }{\l ^{2}} + \bar{c}_3 \, \e \, \a +  O \Big(  \e^2 + \frac{ 1 }{\l^{4} } + \frac{ \ln \l }{ \l ^{n/2}}  + \frac{1}{ ( \l d )^{n-2}} + \frac{ d }{ \l ^{(n-2)/2}}  \Big) , \label{mo13} \\
&  \Big\langle I'_\e(u) ,\frac{1}{\l} \frac{\partial P\d}{\partial a}\Big\rangle  =  \a \,\rho_n \, \frac{\n V(a)}{\l^{3}}   -\frac{\a}{2}\,\frac{c_0\,\bar{c}_2}{\l^{n-1}}\frac{\partial H}{\partial a}(a,a)  + O \Big(\e ^2 +  \frac{ 1 }{\l^{\frac{n}{2}} } + \frac{ 1 }{\l^{4} }+ \frac{ d }{ ( \l d )^{n-1} }  +  \frac{ 1 }{ ( \l d )^{n} } \Big).  \label{mo14}
\end{align}
We notice that, since $ a $ is close to the boundary, we have (see Lemma 7.1 of \cite{BJMAA})
$$ \frac{\partial H}{\partial \nu_{\ov{a} } }(a,a) = \frac{ n-2 }{ (2 d )^{n-1} } + O \big( \frac{1}{ d^{n-2}} \big)  \qquad \mbox{ and } \qquad \frac{\partial H}{\partial a } (a,a) \cdot y =  O \big( \frac{ | y | }{ d^{n-2}} \big)  \, \, \forall \, y \perp \nu_{\ov{a}} . $$ 
Thus we get 
\begin{align}
&  \rho_n \, \frac{\partial V(a)}{\partial \nu_{ \ov{a} }}   - \frac{1}{2}\,\frac{c_0\,\bar{c}_2}{\l^{n-4}}\frac{\partial H}{\partial \nu_{\ov{a}}}(a,a)  =  \rho_n \, \frac{\partial V(b)}{\partial \nu_{ b }}  - \frac{n-2}{2^n}\,\frac{c_0\,\bar{c}_2}{\l^{n-4} d^{n-1} } + O \big( | a-b | +  \frac{1}{\l^{n-4} d^{n-2} } \big) , \label{mo11} \\
&  \rho_n \, \nabla V(a)\cdot \ov{y}   - \frac{1}{2}\,\frac{c_0\,\bar{c}_2}{\l^{n-4}}\frac{\partial H}{\partial a} (a,a) \cdot \ov{y}  =  \rho_n \, \nabla V(a)\cdot \ov{y} 
+ O \big( \frac{1}{\l^{n-4} d^{n-2} } \big) \qquad \forall \, \, \ov{y} \perp \nu_{\ov{a} } . \label{mo12}
 \end{align}
Now, taking the change of variables $(\b_0, \b, \Lambda, D,  \xi )$ defined by 
\begin{align*}
\beta_0 &= 1-\a_0^{p-1}, \quad \beta = 1-\a^{p-1},\\
\frac{1}{\l^2}&=\frac{\bar{c}_3}{\gamma_n}\,\frac{1}{V(b)}\,\e\,(1+\Lambda)\\
d^{n-1}&=   \frac{ n-2 }{2^n\rho_n}\,\frac{ c_0 \bar{c}_2}{|\frac{\partial V}{\partial\nu}(b)|}\Big[\frac{\bar{c}_3}{\gamma_n}\,\frac{1}{V(b)}\,\e\Big]^{(n-4)/2}\big(1+D\big)\\
a&=b+\xi -g(b+\xi)\nu_b-d\nu_{\bar{a}},
\end{align*}
and applying Propositions \ref{rp41}, \ref{rp42}, \eqref{mo13}-\eqref{mo12},  we see that the system \eqref{5.11}-\eqref{5.14} is equivalent to
$$ 
\begin{cases}
\b_0=O(\e)\\
\b=O(\e|\ln\e|)\\
\L = O \Big( | \xi | + \e ^{ \frac{ n-4 }{ 2 (n-1) } } \Big) , \\
\frac{n-4}{2}\Lambda  - D  = O\big( \e^{\frac{3}{2(n-1)}}  + \Lambda^2 +D^2 + | \xi | \big) , \\
D^2(V_{| \partial \O})(b)(\xi, . ) =O( \e^ {\frac{ 3 }{ 2(n-1)} } + | \xi |^2 ) \qquad \mbox{ on }\, \,   T_b(\partial\O).
\end{cases}
$$
In fact, the first and the second ones are immediate. Concerning the third one, using \eqref{mo13} and the change of variables, we obtain 
$$ - \frac{ \gamma_n}{\l^2 } V(a)  + \ov{c}_3 \e = O \Big( \e^2 + \e^{n/4} | \ln \e | + \e ^{ \frac{ 3(n-2)}{2(n-1)} } + \e^{ \frac{n-4}{2(n-1)} + \frac{n-2}{4} } \Big) $$ which implies that 
$$ \L = O \Big( | \xi | + \e ^{ \frac{ n-4 }{ 2 (n-1) } } \Big) . $$
Hence the third equation is proved. Regarding the fourth equation in the previous system, putting the change of variables in \eqref{mo11} and \eqref{mo13}, standard computations imply the fourth equation of the system. Finally, putting the change of variables in \eqref{mo12} and \eqref{mo14}, it yields  
\be \label{moh1} \nabla V(a)\cdot \ov{y}  = O \big( \e^{ \frac{3}{2(n-1)} } \big) \qquad  \, \, \forall \, \ov{ y } \perp \nu_{\ov{a}}  . \ee 
Now, let $ y \in T_b( \partial \O )$ and take $ \ov{y} := y - \langle y , \nu_{\ov{a} } \rangle \nu_{\ov{a} } $. It is easy to see that  $ \ov{y} \in T_{ \ov{ a } }( \partial \O ) $ and $ | \ov{y} | \leq | y | $. Thus, using \eqref{moh1} and \eqref{moh2}, we derive that 
\begin{align}  & \nabla V(a)\cdot {y}  = \nabla V(a)\cdot \ov{y} + \langle y , \nu_{\ov{a} } \rangle \nabla V(a) \cdot \nu_{\ov{a} } \notag \\
 & = \big[  D^2g(b) ( \xi , y ) +  \mathcal{O}(| \xi | ^2 |y|) \big]  \big( \nabla V(b) \cdot \nu_{\ov{a}} + D^2 V(b) ( a-b , \nu_{\ov{a}} ) + O( | a - b |^2 ) \big)  + O \big( \e^{ \frac{3}{2(n-1)} } | \ov{y}|  \big) \notag  \\
  & = D^2g(b) ( \xi , y )  \frac{ \partial V}{ \partial \nu }(b) +  \mathcal{O} \big(  [ | \xi | ^2  + d ^2  + \e^{ \frac{3}{2(n-1)} }  ] | y | \big) .  \label{moh4}
 \end{align}
Now, expanding the left hand side of \eqref{moh4} and using the fact that $ y \in T_b( \partial \O )$,  $ b $ is a critical point of $ V_{ | \partial \O} $ and \eqref{moh5}, we get
\begin{align}  \nabla V(a)\cdot {y} & =  \nabla V(b)\cdot {y} + D^2 V(b) ( a-b, y ) + O ( | a-b |^2 | y | ) \notag  \\
 & = D^2 V(b) ( \xi , y )   + O \big( \big[ | \xi |^2 + d \big] | y | \big) . \label{moh6}
 \end{align}
Combining \eqref{moh4}, \eqref{moh6} with \eqref{127**}, the last equation in the system is therefore proved. 

Finally, using the fact that $D^2V_{ | \partial \O} (b)$ is invertible and applying Brouwer's fixed point Theorem, we derive that the system has a solution for $\e$ small. This leads to the desired result. 

\section{Conclusion}
Through refine asymptotic expansions of the gradient of the Euler-Lagrange functional associated to $(\mathcal{P}_{V,\varepsilon})$, we showed that, in contrast to the case of weak convergence to zero, interior bubbling solutions with a nonzero weak limit cannot arise in low dimensions. Removing the assumption that blow-up points are confined to the interior, we proved that in dimensions $n=4$ and $n=5$, single blow-up points cannot coexist with residual mass. We also clarified the role of the sign of the normal derivative of the potential on the boundary: a positive sign forces any single blow-up solution with residual mass to blow up in the interior, whereas a negative sign at some boundary point allows for the construction of boundary blow-up solutions with residual mass. In addition, we constructed both simple and non-simple interior blow-up solutions with residual mass without any assumption on the sign of the normal derivative.
However, several promising directions for future research and open questions still remain:

\begin{itemize}
\item[(i)]
In this paper, we have established the impossibility of the coexistence of a single blow-up and a residual mass in dimensions $n=4$ and $n=5$. A natural extension of this work would be to investigate the case of multiple blow-up points in the presence of a residual mass for $3\leq n\leq 5$. The six-dimensional case also deserves attention, though it is expected to resemble the behavior observed in lower dimensions.
\item[(ii)]
We have also observed that a positive sign of the normal derivative of the potential on the boundary forces single blow-up solutions with residual mass to blow up in the interior. What about the case of multiple blow-up points with residual mass? Additionally, we have constructed single boundary blow-up points with residual mass when the sign of the normal derivative of the potential is negative. An interesting extension would be to examine the case of multiple boundary blow-up points with residual mass, particularly for clusters of bubbles.
\item[(iii)]
The present work considers a slightly subcritical exponent within the framework of Sobolev embedding. A natural direction for future research would be to extend the analysis to the slightly supercritical regime, that is, when $\varepsilon <0$ with $|\varepsilon|$ small.
\end{itemize}

\noindent{\bf Author Contributions:} 
 R.A., M. B.A.  and K.E.M.: conceptualization, methodology, investigation, writing original  draft, writing-review and editing. All authors have read and agreed to the published version of the manuscript.

\noindent{\bf Data Availability Statement:} Data are contained within the article.

\noindent {\bf  Use of Generative-AI tools declaration:} The authors declare that they have not used Artificial Intelligence (AI) tools in the creation of this article.

\noindent{\bf Acknowledgments:} 
The Researchers would like to thank the Deanship of Graduate Studies and Scientific Research at Qassim University for financial support (QU-APC-2026).

\noindent{\bf Conflicts of Interest:} The authors declare no conflict of interest.

\section{Appendix}
 In this appendix, we collect some useful estimates used throughout the paper. Throughout this appendix, we assume that $\l d(a, \partial\O)$ is sufficiently large.
For $(a,\l)\in \O\times (0,\infty)$, we set
$$
\theta_{(a,\l)}=\d_{(a,\l)}- P\d_{(a,\l)},
$$
and, for simplicity of notation, we will henceforth write $\theta$,  $\d$,  $P\d$,  in place of $\theta_{(a,\l)}$,  $\d_{(a,\l)}$, and $P\d_{(a,\l)}$  respectively.\\
Let $G$ denote the Green’s function of the Laplacian in $\O$ subject to Dirichlet boundary conditions, and let $H$ represent its regular part, that is:
\begin{align*}
&G(x,y)=|x-y|^{2-n}-H(x,y)\quad \mbox{in}\quad\O^2,\quad G=0\quad \mbox{ on }\quad\partial(\O^2),\quad \D_x H=0 \quad \mbox{in}\quad\O^2.
\end{align*}
\begin{lem}\label{A1}
The following estimates hold:
\begin{align*}
&0\leq P\d_{a,\l}\leq \d_{a,\l},\quad 0\leq \theta_{a,\l}\leq \d_{a,\l}, \quad  \Big| \l \frac{ \partial \theta_{a,\l}}{\partial\l} \Big| \leq c \theta_{a,\l} ,\quad \Big| \frac{1}{\l}\frac{ \partial  \theta_{a,\l}}{\partial a} \Big|  \leq c \frac{\theta_{a,\l}}{\l d} , \\
&\l\frac{\partial P\d_{a,\l}}{\partial\l}=O(P\d_{a,\l}),\quad \frac{1}{\l}\frac{\partial P\d_{a,\l}}{\partial a}=O(P\d_{a,\l}),\quad \frac{1}{\l}\frac{\partial P\d_{a,\l}}{\partial a}=O\big(\frac{\d_{a,\l}}{\l |x-a|}+\frac{\d_{a,\l}}{\l d}\big), \\
&\theta_{a,\l}(x)=c_0\frac{H(a,x)}{\l^{(n-2)/2}} + O\Big(\frac{1}{(\l d)^{(n+2)/2}d^n}\Big),\quad \|\theta_{a,\l}\|_{L^{p+1}}=O\Big(\frac{1}{(\l d)^{(n-2)/2}}\Big),\\
&\l\frac{\theta_{a,\l}}{\partial\l}(x) = -\frac{n-2}{2}\frac{c_0H(a,x)}{\l^{(n-2)/2}}+ O\Big(\frac{1}{(\l d)^{(n+2)/2}d^n}\Big),\quad \|\l\frac{\theta_{a,\l}}{\partial\l}\|_{L^{p+1}}=O\Big(\frac{1}{(\l d)^{(n-2)/2}}\Big),\\
\end{align*}
\begin{align*}
&\frac{1}{\l}\frac{\theta_{a,\l}}{\partial a}(x)=\frac{c_0}{\l^{n/2}}\frac{\partial H(a,x)}{\partial a}+ O\Big(\frac{1}{(\l d)^{(n+4)/2}d^{n+1}}\Big),\quad  \|\frac{1}{\l}\frac{\theta_{a,\l}}{\partial a}\|_{L^{p+1}}=O\Big(\frac{1}{(\l d)^{n/2}}\Big),
\end{align*}
where $c_0$ is the constant defined in \eqref{delta} and $d=d(a,\partial\O)$.
\end{lem}
\begin{pf}
With the exception of Estimates $3-7$, the proofs of the remaining estimates can be found in \cite{B, BC, R-G}. We provide the proof of the fourth estimate. The others are proved similarly, so we omit their proofs. Since $ {\partial {\theta}_{a, \l} } / {\partial a_j} $ satisfies 
$$  \D  \Big( \frac{\partial {\theta}_{a, \l} }{\partial a_j}\Big) = 0 \, \,  \mbox{ in } \O  \quad \mbox{ and } \quad  
  \frac{\partial {\theta}_{a, \l} }{\partial a_j} =   \frac{\partial {\d}_{a, \l} }{\partial a_j}  \mbox{ on } \partial  \O,  $$
 for each $ x \in \O$,  it holds 
$$  \frac{1}{\l } \frac{\partial {\theta}_{a, \l} }{\partial a_j} (x) = - \int_{ \partial \O } \frac{ \partial G}{\partial \nu} (x,y) \frac{1}{\l } \frac{\partial {\d}_{a, \l} }{\partial a_j} (y) dy .$$
But, we know that 
$$  \Big| \frac{1}{\l } \frac{\partial {\d}_{a, \l} }{\partial a_j} (y) \Big| \leq \frac{ c }{ \l |y-a|  } \d_{a, \l }(y)  \quad \forall \, y \in \O\setminus \{a\} \qquad \mbox{ and } \qquad  \frac{ \partial G}{\partial \nu} (x,y) \leq 0 \quad \forall \, y \in \partial \O . $$
Thus, we obtain 
$$ \Big|  \frac{1}{\l } \frac{\partial {\theta}_{a, \l} }{\partial a_j} (x) \Big|  \leq  \frac{ c }{\l d }\,  \int_{ \partial \O }  -\frac{ \partial G}{\partial \nu} (x,y)   {\d}_{a, \l} (y) dy \leq  \frac{ c }{\l d } \theta _{a, \l} (x) . $$
Hence the proof is completed. 
\end{pf}


\end{document}